\documentclass[11pt,twoside]{amsart}

\usepackage{amsmath,amssymb}
\usepackage{amsfonts}
\usepackage{amsthm,amscd}
\usepackage{amsmath}
\usepackage{amssymb}
\usepackage{amstext}
\usepackage{color}
\usepackage{latexsym}
\usepackage{amscd,graphics}
\usepackage{enumerate}

\usepackage{tikz}
\usetikzlibrary{calc}
\usetikzlibrary{decorations.pathreplacing,angles,quotes}

\usepackage{float}

\usepackage{bbm}
\usepackage{empheq}
\usepackage{mathtools}
\PassOptionsToPackage{hyphens}{url}
\usepackage{hyperref}
\hypersetup{colorlinks = false, urlcolor = blue, bookmarksopen = true}

\usepackage{geometry}
\allowdisplaybreaks

\newcommand{\cA}{\mathcal{A}}
\newcommand\cB{{\mathcal B}}

\newcommand{\cC}{\mathcal{C}}
\newcommand{\cE}{\mathcal{E}}
\newcommand{\cF}{\mathcal{F}}
\newcommand{\cG}{\mathcal{G}}
\newcommand{\cI}{\mathcal{I}}

\newcommand{\cL}{\mathcal{L}}

\newcommand{\cM}{\mathcal{M}}
\newcommand{\cN}{\mathcal{N}}
\newcommand{\cP}{\mathcal{P}}
\newcommand{\hP}{\mathring{\cP}}
\newcommand{\cQ}{\mathcal{Q}}
\newcommand\cR{{\mathcal R}}
\newcommand{\cS}{\mathcal{S}}
\newcommand{\cW}{\mathcal{W}}

\newcommand{\bH}{\mathbb{H}}

\newcommand{\bx}{\bar{x}}

\newcommand{\bTheta}{\overline{\Theta}}
\newcommand{\bGamma}{\overline{\Gamma}}

\newcommand{\tpsi}{\widetilde{\psi}}

\newcommand{\hW}{\widehat{\cW}}
\newcommand{\hLambda}{\hat{\Lambda}}

\newcommand{\musrb}{\mu_{\tiny{\mbox{SRB}}}}
\newcommand{\hatmusrb}{\hat \mu_{\tiny{\mbox{SRB}}}}

\newcommand{\vf}{\varphi}

\newcommand{\ve}{\varepsilon}

\newcommand{\diam}{\mbox{diam}}

\newcommand{\comments}[1]{} 

\newtheorem{proposition}{Proposition}[section]
\newtheorem{lemma}[proposition]{Lemma}
\newtheorem{sublemma}[proposition]{Sublemma}
\newtheorem{theorem}[proposition]{Theorem}

\newtheorem{corollary}[proposition]{Corollary}

\newtheorem{definition}[proposition]{Definition}

\theoremstyle{remark}
\newtheorem{remark}[proposition]{Remark}
\theoremstyle{definition}

\numberwithin{equation}{section}

\begin{document}

\title[Equilibrium measures for billiard flows]{A family of natural equilibrium measures\\ for Sinai billiard flows}

\author{J\'er\^ome Carrand\textsuperscript{(1),(2)}}
\email{jerome.carrand@sns.it}
\address{(1) Sorbonne Universit\'e and Universit\'e Paris Cit\'e, CNRS,   Laboratoire de Probabilit\'es, Statistique et Mod\'elisation,
	F-75005 Paris, France}
\address{(2) Centro di Ricerca Matematica Ennio De Giorgi, SNS,  
56100 Pisa,  Italy (current address)}

\date{\today}
\begin{abstract}
The Sinai billiard flow on the two-torus, i.e., the periodic Lorentz gas, is a continuous flow, but it is not everywhere differentiable. Assuming finite horizon, we relate the equilibrium states of the flow with those of the Sinai billiard map $T$ -- which is a discontinuous map. We propose a definition for the topological pressure $P_*(T,g)$ associated to a potential $g$. We prove that for any piecewise H\"older potential $g$ satisfying a mild assumption, $P_*(T,g)$ is equal to the definitions of Bowen using spanning or separating sets. We give sufficient conditions under which a potential gives rise to equilibrium states for the Sinai billiard map. We prove that in this case the equilibrium state $\mu_g$ is unique, Bernoulli, adapted and gives positive measure to all nonempty open sets. For this, we make use of a well chosen transfer operator acting on anisotropic Banach spaces, and construct the measure by pairing its maximal eigenvectors. Last, we prove that the flow invariant probability measure $\bar \mu_g$, obtained by taking the product of $\mu_g$ with the Lebesgue measure along orbits, is Bernoulli and flow adapted. We give examples of billiard tables for which there exists an open set of potentials satisfying those sufficient conditions. 
\end{abstract}
\thanks{JC is grateful to ITS--ETH Zurich for their invitation in May 2022. Thanks to Viviane Baladi, Mark Demers, Alexey Korepanov and Françoise P\`ene for their helpful discussions and comments, and for pointing out several typos. 
Research supported by the European Research Council (ERC) under the European Union's Horizon 2020 research and innovation programme (grant agreement No 787304).}
\maketitle

\section{Introduction}

\subsection{Billiards and equilibrium states}

In this work, we are concerned with a class of dynamics with singularities: the dispersing billiards introduced by Sinai \cite{S} on the two-torus. A Sinai billiard on the torus is (the quotient modulo $\mathbbm{Z}^2$, for position, of) the periodic planar Lorentz gas (1905) model for the motion of a single dilute electron in a metal. The scatterers (corresponding to atoms of the metals) are assumed to be strictly convex (but not necessarily discs). Such billiards have become fundamental models in mathematical physics.

To be more precise, a Sinai billiard table $Q$ on the two-torus $\mathbbm{T}^2$ is a set $Q = \mathbbm{T}^2 \smallsetminus B$ with $B = \sqcup_{i=1}^D B_i$ for some finite number $D\geqslant 1$ of pairwise disjoint closed domains $B_i$, called scatterers, with $C^3$ boundaries having strictly positive curvature -- in particular, the scatterers are strictly convex. The billiard flow $\phi_t$ is the motion of a point particle travelling at unit speed in $Q$ with specular reflections off the boundary of the scatterers. Identifying outgoing collisions with incoming ones in the phase space, the billiard flow is continuous. However, the grazing collisions -- those tangential to scatterers -- give rise to singularities in the derivative \cite{chernov2006chaotic}. The Sinai billiard map $T$ -- also called collision map -- is the return map of the single point particle to the scatterers. Because of the grazing collisions, the Sinai billiard map is a discontinuous map.

Sinai billiard maps and flows both preserve smooth invariant probability measures, respectively $\musrb$ and $\tilde \mu _{\rm SRB}$, which have been extensively studied: $(T,\musrb)$ and $(\phi_t,\tilde \mu _{\rm SRB})$ are uniformly hyperbolic, ergodic, K-mixing \cite{S,BuS,SC}, and Bernoulli \cite{GO,ChH}. The measure $\musrb$ is $T$-adapted \cite{KS} in the sense of the integrability condition:
\begin{align*}
\int |\log d(x,\cS_{\pm 1})| \,\mathrm{d}\musrb < \infty \, ,
\end{align*}
where $\cS_{\pm 1}$ is the singularity set for $T^{\pm 1}$.  Both systems enjoy exponential decay of correlations \cite{Y,DZ1}. Since the billiard has many periodic orbits, it thus has many other ergodic invariant measures, but until very recently most of the results apply to perturbations of $\musrb$ \cite{CWZ,DRZ}.

\medskip
In the case of an Anosov flow, it is known since the work of Bowen \cite{bowen1972periodic} that the Kolmogorov-Sinai entropy is upper-semicontinuous, which guarantees the existence of measures of maximal entropy, or more generally, of equilibrium states. Because of the singularities, billiard flows are not Anosov and therefore methods used in the context of Anosov flows cannot be applied easily. The upper-semicontinuity of the entropy is not known at the moment, and, more generally, the existence of equilibrium states has to be treated one potential at the time. 

In a recent paper, Baladi and Demers \cite{BD2020MME} proved, under a mild technical assumption and assuming finite horizon, that there exists a unique measure of maximal entropy $\mu_*$ for the billiard map, and that $\mu_*$ is Bernoulli, $T$-adapted, charges all nonempty open sets and does not have atoms. Their construction of this measure relies on the use of a transfer operator acting on anisotropic Banach spaces, similar to those used by \cite{DZ1} in order to study $\musrb$. Combining their work with those of Lima--Matheus \cite{LM} and Buzzi \cite{Bu}, Baladi and Demers proved that their exists a positive constant $C$ such that 
\begin{align}
C e^{h_* m} \leqslant \# {\rm Fix}\, T^m \, , \quad \forall m \geqslant 1 \, ,
\end{align}
where $ \# {\rm Fix}\, T^m$ denotes the number of fixed points of $T^m$, and $h_*$ is the topological entropy of the map $T$ from \cite{BD2020MME}. Baladi and Demers also give a condition under which $\mu_*$ and $\musrb$ are different.

In a subsequent paper, Baladi and Demers~\cite{BD2} constructed a family of equilibrium states $\mu_t$ for $T$ associated to the family of geometric potentials $-t \log J^u T$, where $J^u T$ is the unstable Jacobian of $T$ and $t \in (0,t_*)$ for some $t_* > 1$. In the case $t=1$, $\mu_t=\musrb$. The construction again relies on the use of a family of transfer operators $\cL_t$ acting on anisotropic Banach spaces. For each $t \in (0,t_*)$, they proved that $\mu_t$ is the unique equilibrium state associated with the potential $-t \log J^u T$, that $\mu_t$ is mixing, $T$-adapted, has full support and does not have atoms. Baladi and Demers also showed that each transfer operator $\cL_t$ has a spectral gap, from which they deduced the exponential rate of mixing for each measure $\mu_t$, for $C^1$ observables.

Even more recently, Demers and Korepanov \cite{DK} proved a polynomial decay of correlations for the measure $\mu_*$ for H\"older observables under an assumption slightly stronger than the one used in \cite{BD2020MME}. Nonetheless, thanks to a result annonced by T\'oth \cite{toth}, the assumptions from \cite{BD2020MME} and \cite{DK} are satisfied for generic billiard tables.

\medskip 

In this paper, we give a sufficient condition under which a piecewise H\"older potential $g$ admits equilibrium states for $T$. Under this assumption, we prove that the equilibrium state is in fact unique, Bernoulli, $T$-adapted and charges all nonempty open sets. We prove that its lift into a flow invariant measure is Bernoulli and flow-adapted (that is, the logarithm of the norm of the differential of the flow is integrable). We also identify the potential $g = - h_{\rm top}(\phi_1) \tau$ to be such that its corresponding equilibrium states for $T$ -- whenever they exist -- are in bijection with measures of maximal entropy of the billiard flow.

Notice that the geometric potentials $- t \log J^u T$ are not piecewise H\"older, and thus the work of Baladi and Demers \cite{BD2} on those potentials is distinct from ours.

\subsection{Statement of main results -- Organization of the paper}

Since transfer operator techniques are simpler to implement for maps than for flows, we will be concerned with the associated billiard map $T:M \to M$ defined to be the first collision map on the boundary of $Q$, where $M = \partial Q \times [-\pi/2, \pi/2]$. We assume as in \cite{Y,BD2020MME}, that the billiard table $Q$ has \emph{finite horizon}, in the sense that the billiard flow does not have any trajectories making only tangential collisions -- in particular, this implies that the return time function $\tau$ to a scatterers is bounded.

The first step is to find a suitable notion of topological pressure $P_*(T,g)$ for the discontinuous map $T$ and a potential $g:M \to \mathbbm{R}$. In order to define it, we introduce as in \cite{BD2020MME}, the following increasing family of partition of $M$. Let $\cP$ be the partition into maximal connected sets on which both $T$ and $T^{-1}$ are continuous, and let $\cP_{-k}^n = \bigvee_{i=-k}^n T^{-i} \cP$. Then the sequence $\sum_{P \in \cP_0^n} \sup_P e^{S_n g}$ is submultiplicative, where $S_n g = \tfrac{1}{n}\sum_{i=0}^{n-1} g \circ T^i$ is the Birkhoff sum of $g$. We can thus define the topological pressure by
\begin{definition}\label{def:topological_pressure}
$P_*(T,g) \coloneqq \lim\limits_{n \to +\infty} \frac{1}{n} \log \sum_{P \in \cP_0^n} \sup_P e^{S_n g} $
\end{definition}
Section~\ref{sect:pressure} is dedicated to the study of this quantity. In particular, we prove (Proposition~\ref{prop:pressure_bowen_definitions}) that whenever the potential $g$ is smooth enough -- piecewise H\"older -- and $P_*(T,g) - \sup g >0$ then $P_*(T,g)$ coincides with both Bowen's definitions using spanning sets and separating sets. We also prove (Lemma~\ref{lemma:var_ppl_easy_dir}) that for each $T$-invariant measure $\mu$, we have $P_*(T,g) \geqslant h_{\mu}(T) + \int g \,\mathrm{d}\mu$. Finally, we show that if $g = - h_{\rm top}(\phi_1) \tau$ admits an equilibrium state $\mu_g$, then the measure $\bar \mu_g = (\int \tau \,\mathrm{d}\mu_g)^{-1} \mu_g \otimes \lambda$ is a measure of maximal entropy for the billiard flow, seen as a suspension flow over $T$, where $\lambda$ is the Lebesgue measure in the flow direction.

To state our existence results (in Section~\ref{sect:measure_mu_g}), we need to quantify the recurrence to the singular set. Fix an angle $\varphi_0$ close to $\pi/2$ and $n_0 \in \mathbbm{N}$. We say that a collision is $\varphi_0$-grazing if its angle with the normal is larger than $\varphi_0$ in absolute value. Let $s_0 = s_0(\varphi_0,n_0) \in (0,1]$ denote the smallest number such that
\begin{align}\label{eq:s0}
\text{any orbit of length $n_0$ has at most $s_0 n_0$ collisions which are $\varphi_0$-grazing.}
\end{align}
Due to the finite horizon condition, we can choose $\varphi_0$ and $n_0$ such that $s_0 < 1$. We refer to \cite[\S 2.4]{BD2020MME} for further discussion on this quantity. From \cite{chernov2006chaotic}, $\Lambda = 1 + \kappa_{\min} \tau_{\min} >1 $ is the expanding factor in the hyperbolicity of $T$, where $\kappa_{\min}$ is the minimal curvature of the scatterers and $\tau_{\min}$ is the minimum of the return time function $\tau$. Define $\cS_0 = \{ (r,\varphi) \in M \mid |\varphi|=\pi /2 \}$ the set of grazing collisions, and $\cS_{\pm n} = \cup_{i=0}^n T^{\mp i} \cS_0$ the singular set of $T^{\pm n}$. Call $\cN_\ve(\cdot)$ the $\ve$-neighbourhood of a set. Then
\begin{theorem}\label{Main_theorem_1}
If $g$ is a bounded, piecewise H\"older\footnote{See precise definition in the beginning of Section~\ref{sect:pressure}.} potential such that $P_*(T,g) - \sup g > s_0 \log 2$ and $\log \Lambda > \sup g - \inf g$, then there exists a probability measure $\mu_g$ such that 
\begin{enumerate}
\item[(i)] $\mu_g$ is $T$-invariant, $T$-adapted and for all $k \in \mathbbm{Z}$, there exists $C_k > 0$ such that $\mu_g(\cN_\varepsilon(\cS_k)) \leqslant C_k |\log \varepsilon|^{-\gamma}$, where $\gamma > 1$ is such that $P_*(T,g) - \sup g > \gamma s_0 \log 2$.
\item[(ii)] $\mu_g$ the unique equilibrium state of $T$ under $g$: that is $P_*(T,g) = h_{\mu_g}(T) + \int g \,\mathrm{d}\mu_g$ and $P_*(T,g) > h_{\mu}(T) + \int g \,\mathrm{d}\mu$ for all $\mu \neq \mu_g$.
\item[(iii)] $\mu_g$ is Bernoulli\footnote{Recall that Bernoulli implies K-mixing, which implies strong mixing, which implies ergodic.} and charges all nonempty open sets.
\end{enumerate}
If the assumption $\log \Lambda > \sup g - \inf g$ is weakened into the condition SSP.1 (Definition~\ref{def:SSP1}), then item (i) still holds. If the assumption $\log \Lambda > \sup g - \inf g$ is weakened into the condition SSP.2 (Definition~\ref{def:SSP2}), then items (i), (ii) and (iii) hold.
\end{theorem}
The above theorem will follow from  Proposition~\ref{prop:exist}, Lemma~\ref{lemma:measure_Sk}, Corollary~\ref{thm:equilibrium states}, and Propositions~\ref{prop:uniqueness}, \ref{prop:mu_g_is_bernoulli}, \ref{prop:full_support}. Furthermore, assuming the sparse recurrence to singularities condition from \cite{BD2020MME}, we provide in Remark~\ref{remark:neighbourhood_zero_potential} an open set of potentials, each having SSP.1 and SSP.2.

The tool used to construct the measure $\mu_g$ is a transfer operator $\cL_g$ with $\cL_g f = \left( f\, e^g / J^s T \right) \circ T^{-1}$, similar to the one used in \cite{BD2020MME} corresponding to the case $g \equiv 0$. This operator and the anisotropic Banach spaces on which it acts are defined in details in Section~\ref{sect:banach_spaces_and_transfer_operators}. Section~\ref{sect:growth_lemma} contains key combinatorial growth lemmas, controlling the growth in weighted complexity of the iterates of a stable curve. These lemmas will be crucial since the quantity they control appears in the norms of the iterates of $\cL_g$. In Section~\ref{sect:norm_estimates}, we prove a (degenerated) ``Lasota--Yorke" type inequality (Proposition~\ref{prop:upper_bounds_norms}) --~giving an upper bound on the spectral radius of $\cL_g$~-- as well as a lower bound on the spectral radius (Theorem~\ref{thm:spectral_radius}). 

Section~\ref{sect:measure_mu_g} is devoted to the construction and the properties of the measure $\mu_g$. From the estimates on the norms from the previous section, we are able to construct left and right maximal eigenvectors ($\tilde \nu$ and $\nu$) for $\cL_g$. We construct the measure $\mu_g$ by pairing these eigenvectors.  We then prove the estimates on the measure of a neighbourhood of the singular sets (Lemma~\ref{lemma:measure_Sk}). Section~\ref{sect:absolute_continuity} contains the key result of the absolute continuity of the stable and unstable foliations with respect to $\mu_g$, as well as the proof that $\mu_g$ has total support -- this is done by extending $\nu$ into a measure and exploiting the $\nu$-almost everywhere positive length of unstable manifolds from Section~\ref{sect:positive_length_unstable_manifolds}. In Section~\ref{sect:bernoulli_var_principle}, we show that $\mu_g$ is ergodic, from which we bootstrap to K-mixing using a Hopf argument. Adapting \cite{ChH} with modifications from \cite{BD2020MME}, we deduce from the hyperbolicity and the K-mixing that $\mu_g$ is Bernoulli. Still in Section~\ref{sect:bernoulli_var_principle}, we give an upper-bound on the measure of weighted Bowen balls, from which we deduce, using the Shannon--MacMillan--Breiman theorem, that $\mu_g$ is an equilibrium state for $T$ under the potential $g$ (Corollary~\ref{thm:equilibrium states}). Finally, the Section~\ref{sect:uniqueness} is dedicated to the uniqueness of the equilibrium state $\mu_g$.

In Section~\ref{sect:billiard_flow}, we prove using arguments from \cite{chernov2006chaotic} that $(\phi_t,\bar \mu_g)$ is K-mixing (Proposition~\ref{prop:billiard_flow_k-system}), and again, using the hyperbolicity of the billiard flow, we adapt \cite{ChH} in order to prove that $(\phi_t,\bar \mu_g)$ is Bernoulli (Proposition~\ref{prop:billiard_flow_bernoulli}). Finally, we prove that $\bar \mu_g$ is flow adapted in the sense of the integrability condition formulated in Proposition~\ref{prop:flow_adapted}. We summarize these results about the billiard flow in the following theorem.

\begin{theorem}
Let $g$ be a potential satisfying the assumptions from Theorem~\ref{Main_theorem_1}, and let $\bar \mu_g \coloneqq (\int \tau \,\mathrm{d}\mu_g)^{-1} \mu_g \otimes \lambda$. Then $\bar \mu_g$ is a $\phi_t$-invariant Borel probability measure that is an equilibrium states for $\phi_t$ and any potential $\tilde g$ such that $g= \lambda(\tilde g) - P(\phi_1,\tilde g)\tau$, where $\lambda(\tilde g)(x)=\int_0^{\tau(x)} \tilde g (\phi_t(x)) \, \mathrm{d}t$. Furthermore, $\bar \mu_g$ is flow adapted and $(\phi_t,\bar \mu_g)$ is Bernoulli.
\end{theorem}

In a subsequent joint work with Baladi and Demers \cite{BCD}, we bootstrap from the results of the present paper to show that if $h_{\rm top}(\phi_1) \tau_{\min} > s_0 \log 2$ then the potential $-h_{\rm top}(\phi_1) \tau$ satisfies the sufficient assumptions SSP.1 and SSP.2 in our Theorem~\ref{Main_theorem_1}, thus constructing a measure of maximal entropy for the billiard flow. This is done by studying the family of potentials $-t \tau$ and proving that the maximal value $t_\infty$ of $t$ such that $- t' \tau$ has SSP.1 and SSP.2 for all $0 \leqslant t' \leqslant t$, satisfies $t_\infty > h_{\rm top}(\phi_1)$. By Remark~\ref{remark:neighbourhood_zero_potential} and Corollary~\ref{corol:eq_state_and_MME}, for every small enough $|t|$, $-t \tau$ has SSP.1 and SSP.2 (thus $t_\infty >0$), and the case $t=h_{\rm top}(\phi_1)$ corresponds to measures of maximal entropy for the billiard flow.

\section{Topological Pressure, Variational Principle and Abramov Formula}\label{sect:pressure}

In this section, we formulate definitions of topological pressure for the billiard map $T$ under a potential $g$, and prove that -- under some conditions on $g$ -- they are equivalent. Using a classical estimate, we then prove one direction of the variational principle. Finally, making use of the Abramov formula, we relate equilibrium states of $T$ with the ones of the billiard flow. More specifically, we identify the potential for $T$ which is related to -- up to existence -- the measures of maximal entropy of $\phi_t$.

We first introduce notation: Adopting the standard coordinates $x=(r,\varphi)$ on each connected component $M_i$ of 
\[
M \coloneqq \partial Q \times \left[ -\frac{\pi}{2},\frac{\pi}{2} \right] = \bigsqcup_{i=1}^D \partial B_i \times \left[ -\frac{\pi}{2},\frac{\pi}{2} \right],
\]
where $r$ denotes arclength along $\partial B_i$, $\varphi$ is the angle the post-collisional trajectory makes with the normal to $\partial B_i$ and $M_i = \partial B_i \times \left[ -\tfrac{\pi}{2},\tfrac{\pi}{2} \right]$. In these coordinates, the collision map $T :M \to M$ preserve a smooth invariant probability measure $\musrb$ given by $\mathrm{d}\musrb = (2|\partial Q|)^{-1} \cos \varphi \, \mathrm{d}r \mathrm{d}\varphi$.

We now define the sets where $T$ and its iterates are discontinuous. Let $\cS_0 \coloneqq \{ (r,\varphi) \in M \mid |\varphi|= \pi/2 \}$ denote the set of grazing collisions. For each nonzero $n \in \mathbbm{N}$, let 
\[
\cS_{\pm n} \coloneqq \bigcup_{i=0}^n T^{\mp i} \cS_0,
\]
denote the singularity set for $T^{\pm n}$. It would be natural to study the map $T$ restricted to the invariant set $M \smallsetminus \cup_{n \in \mathbbm{Z}} \cS_n$ where $T$ is continuous, however the set of curves $\cup_{n \in \mathbbm{Z}} \cS_n$ is dense in $M$ \cite[Lemma~4.55]{chernov2006chaotic}. We thus introduce the classical family of partitions of $M$ as follows.

For $k$, $n \geqslant 0$, let $\cM_{-k}^{n}$ denote the partition of $M \smallsetminus (\cS_{-k} \cup \cS_n)$ into its maximal connected components. Note that all elements of $\cM_{-k}^n$ are open sets on which $T^i$ is continuous, for all $-k \leqslant i \leqslant n$. Since the thermodynamic sums over elements of $\cM_0^n$ of a potential $g$ will play a key role in the estimates on the norms of the iterates of the transfer operator $\cL_g$ in Section~\ref{sect:norm_estimates}, it should be natural -- by analogy to the case of continuous maps -- to define the topological pressure from these sums.

\medskip

Another natural family of partitions is defined as follows. Let $\cP$ denote the partition of $M$ into maximal connected components on which both $T$ and $T^{-1}$ are continuous. Define $\cP_{-k}^n = \bigvee_{i=-k}^n T^{-i} \cP$ and remark that $T^i$ is continuous on each element of $\cP_{-k}^n$, for all $-k \leqslant i \leqslant n$. 

The interior of each element of $\cP$ corresponds to precisely one element of $\cM_{-1}^1$, but its refinements $\cP_{-k}^n$ may also contain some isolated points -- this happens if three or more scatterers have a common grazing collision. These partitions already appeared in the work of Baladi and Demers, where they proved \cite[Lemma~3.2]{BD2020MME} that the number of isolated points in $\cP_{-k}^n$ grows linearly in $n+k$.

Finally, denote $\hP_{-k}^n$ the collection of interior of elements of $\cP_{-k}^n$. In \cite[Lemma~3.3]{BD2020MME}, Baladi and Demers proved that $\hP_{-k}^n = \cM_{-k-1}^{n+1}$, for all $n$, $k \geqslant 0$. It should be natural that the topological pressures obtained from these three families of partitions coincide. This is the content of Theorem~\ref{thm:pressure}.

\medskip

In order to formulate the result on the equivalence between definitions of topological pressure for $T$, we need to be more specific about the definition of \emph{piecewise} H\"older.

We say that a function $g$ is $(\cM_0^1,\alpha)$-H\"older, $0<\alpha <1$, if $g$ is $\alpha$-H\"older continuous on each element of the partition $\cM_0^1$. 
We define the $C^{\alpha}$ norm $|g|_{C^{\alpha}}$ to be the maximum of the usual $C^{\alpha}$ norms $|g|_{C^{\alpha}(A)}$, for $A \in \cM_0^1$, that is
\[ |g|_{C^{\alpha}} \coloneqq \max \{ |g|_{C^{0}(A)} + H^{\alpha}_A(g) \mid A \in \cM_0^1 \}, \quad \text{where } H^{\alpha}_A(g) = \sup\limits_{x,y \in A} \frac{|g(x) - g(y)|}{d(x,y)^{\alpha}}. \]
We simply say that a function $g$ is piecewise H\"older if there exists $\alpha \in (0,1)$ such that $g$ is $(\cM_0^1,\alpha)$-H\"older. 

Similarly, we say that a function $g$ is $\cM_0^1$-continuous if $g$ is bounded and continuous on each element of the partition $\cM_0^1$. We define the $C^0$ norm $|g|_{C^0}$ to be the maximum of the usual $C^0$ norms $|g|_{C^0(A)}$, over $A \in \cM_0^1$, that is $|g|_{C^{0}} \coloneqq \max \{ |g|_{C^{0}(A)} \mid A \in \cM_0^1 \}$.

\begin{theorem}\label{thm:pressure}
Let $g :M \to \mathbbm{R} $ be a potential bounded from above. Then \[  P_*(T,g) \coloneqq \lim\limits_{n \to +\infty} \frac{1}{n} \log \sum\limits_{A \in \mathcal{P}_{0}^n} \sup\limits_{x \in A} e^{(S_n g)(x)} \] exists and is called the (topological) pressure of $g$. Moreover, the map $g \mapsto P_*(T,g)$ is convex.

When $g$ is $\cM_0^1$-continuous and $P_*(T,g) - \sup g > 0$, the following limits exist and are equal to $P_*(T,g)$.

\begin{tabular}{rlcrl}
\textit{(i)} &$\lim\limits_{n \to +\infty} \frac{1}{n} \log \sum\limits_{A \in \mathring{\mathcal{P}}_{0}^n} \sup\limits_{x \in A} e^{(S_n g)(x)}$, & &
\textit{(ii)} &$\lim\limits_{n \to +\infty} \frac{1}{n} \log \sum\limits_{A \in \mathcal{M}_{0}^n} \sup\limits_{x \in A} e^{(S_n g)(x)}$.
\end{tabular}

Furthermore, when $g$ is also $(\cM_0^1,\alpha)$-H\"older continuous, then the following limits are equal to $P_*(T,g)$.

\begin{tabular}{rlrl}
\textit{(iii)} &$\lim\limits_{n \to +\infty} \frac{1}{n} \log \sum\limits_{A \in \mathcal{P}_{0}^n} \inf\limits_{x \in A} e^{(S_n g)(x)}$, &
\textit{(iv)} &$\lim\limits_{n \to +\infty} \frac{1}{n} \log \sum\limits_{A \in \mathring{\mathcal{P}}_{0}^n} \inf\limits_{x \in A} e^{(S_n g)(x)}$, \\
\textit{(v)} &$\lim\limits_{n \to +\infty} \frac{1}{n} \log \sum\limits_{A \in \mathcal{M}_{0}^n} \inf\limits_{x \in A} e^{(S_n g)(x)}$. & &
\end{tabular} \\
Finally, for a bounded potential, the sequence $n \mapsto \log \sum\limits_{A \in \cM_0^n} \sup\limits_{x \in A} e^{(S_{n-1}g)(x)}$ is subadditive.
\end{theorem}

\begin{proposition}\label{prop:pressure_bowen_definitions}
Let $g$ be a $\cM_0^1$-continuous potential. Let $P_{\rm span}(T,g)$ and $P_{\rm sep}(T,g)$ be the pressure obtained using Bowen's definition with, respectively, spanning sets and separating sets. Then $P_{\rm span}(T,g) \leqslant P_*(T,g)$ and $P_{\rm sep}(T,g) \leqslant P_*(T,g)$. When $P_*(T,g) - \sup g >0$, then $P_*(T,g)=P_{\rm sep}(T,g)$, and if furthermore $g$ is $(\cM_0^1,\alpha)$-H\"older, then $P_*(T,g)=P_{\rm span}(T,g)$.
\end{proposition}

The proof of the last three forms of $P_*(T,g)$ in Theorem~\ref{thm:pressure} relies crucially on the following lemma.

\begin{lemma}\label{lemma:sup_less_poly_inf}
For every $(\cM_0^1,\alpha)$-H\"older continuous potential $g$ there exists a constant $C_g$ such that for all $n \geqslant 1$ and all $P \in \cP_0^n$, \[ \sup\limits_P e^{S_n g} \leqslant C_g \inf\limits_P e^{S_n g}. \]
The estimate still holds, for the same constant $C_g$, when $\cP_0^n$ is replaced by $\cP_{-l}^n$, $\mathring{\cP}_{-l}^n$ or $\cM_{-l}^n$, for any fixed $l \geqslant 0$.
\end{lemma}

Before the proofs of these results, we first recall that $T$ is uniformly hyperbolic in the sense that the cones 
\begin{align}\label{eq:stable_cones}
\begin{split}
\cC^u &\coloneqq \{ (dr,d\varphi) \in \mathbbm{R}^2 \mid \kappa_{\min} \leqslant d\varphi /dr \leqslant \kappa_{\max} +1/ \tau_{\min} \}, \\
\cC^s &\coloneqq \{ (dr,d\varphi) \in \mathbbm{R}^2 \mid -\kappa_{\min} \geqslant d\varphi /dr \geqslant - \kappa_{\max} -1/ \tau_{\min} \}, 
\end{split}
\end{align}
are strictly invariant under $DT$ and $DT^{-1}$, respectively, whenever these derivatives exist (see \cite{chernov2006chaotic}). Here $\kappa_{\max} = \max \kappa$, $\kappa_{\min} = \min \kappa$, where $\kappa$ is the curvature of the scatterer boundaries, and $\tau_{\min}= \min \tau$, where $\tau$ is the return time function. Furthermore, there exists $C_1 >0$ such that for all $n \geqslant 1$,
\begin{align*}
||D_x T^n (v) || \geqslant C_1 \Lambda^n ||v||, \, \forall v \in \cC^u \, , \quad ||D_x T^{-n} (v) || \geqslant C_1 \Lambda^n ||v||, \, \forall v \in \cC^s \, ,
\end{align*}
where $\Lambda =1 + 2\kappa_{\min} \tau_{\min}$ is the minimum hyperbolicity constant.

\begin{proof}[Proof of Lemma~\ref{lemma:sup_less_poly_inf}]
Let $d_n$ denote the $n$-th Bowen distance, that is the distance given by
\begin{align*}
d_n(x,y) = \max_{0 \leqslant i \leqslant n} d(T^i x,T^i y) \, ,
\end{align*}
where $d(x,y)$ is the Euclidean metric on each $M_i$, with $d(x,y)= 10 D \, \max_i \diam(M_i)$ if $x$ and $y$ belong to different $M_i$ (this definition ensure we have a compact set). Let $\varepsilon_0 >0$ be as in \cite[(3.3)]{BD2020MME}, that is: if $d_n(x,y) < \varepsilon_0$ then $x$ and $y$ lie in the same element of $\cM_0^n$. 
Therefore, by the uniform hyperbolicity of $T$, if $d(T^i(x),T^i(y)) \leqslant \varepsilon_0/2$ for all $|i|\leqslant n$ then $d(x,y) \leqslant C_1 \Lambda^{-n} \varepsilon_0/2$.

Given a potential $g$, for all integers $m$, define the $m$-th variation by \[\mathrm{Var}_m(g,T,\varepsilon) \coloneqq \sup \{ |g(x) - g(y)| \mid d(T^j x,T^j y) \leqslant \varepsilon, \, |j|\leqslant m \}. \]
 
When $g$ is $(\cM_0^1,\alpha)$-H\"older, we get that $\mathrm{Var}_m(g,T,\frac{\varepsilon_0}{2C_1}) \leqslant C (\frac{\varepsilon_0}{2} \Lambda^{-m})^\alpha$. Therefore \[ \sum\limits_{m \geqslant 0} \mathrm{Var}_m \Big( g,T,\frac{\varepsilon_0}{2C_1} \Big) \eqqcolon K < \infty.\]

By uniform hyperbolicity of $T$, there exists $k_{\varepsilon}$ such that $\mathrm{diam}(\cM_{-k_{\varepsilon}}^{n+1}) < \varepsilon_0/2C_1$ for all $n \geqslant k_{\varepsilon}$. It then follows from the proof of \cite[Lemma 3.5]{BD2020MME} that if $x$ and $y$ lie in the same element of $\cP_{-k_{\varepsilon}}^{k_{\varepsilon}+n}$, then $d_n(x,y) \leqslant \varepsilon_0/2C_1$, for all $n\geqslant 0$.

Let $P \in \cP_{-k_{\varepsilon}}^{k_{\varepsilon}+n}$ and let $x$, $y \in P$. Let $0 \leqslant k \leqslant n$. Then for all $|j| < m_k \coloneqq \min(k,n-k)$, $d(T^j(T^k x), T^j(T^k y)) < \varepsilon_0/2C_1$ and so $| g(T^k x) - g(T^k y)| \leqslant \mathrm{Var}_{m_k}(g,T,\frac{\varepsilon_0}{2C_1})$. Therefore 
\[ |S_n g (x) - S_n g (y) | \leqslant 2 \sum\limits_{m=0}^{\lfloor \frac{n}{2} \rfloor +1} \mathrm{Var}_m \Big( g,T,\frac{\varepsilon_0}{2C_1} \Big) \leqslant 2K < \infty.\]

Now, let $P \in \cP_0^n$ for some $n \geqslant 2k_{\varepsilon}$. Notice that $\cP_0^n= \bigvee_{i=k_\varepsilon}^{n-k_\varepsilon} T^{-i} \cP_{-k_\varepsilon}^{k_\varepsilon}$, in other words for all $l$ such that $k_{\varepsilon} \leqslant l \leqslant n-k_{\varepsilon}$, $T^lP$ is included in an element of $\cP_{-k_\varepsilon}^{k_\varepsilon}$. Finally, by decomposing each orbit into three parts, we get that for all $x$, $y \in P$,
\begin{align*}
e^{S_n g(x) - S_n g(y)} &= e^{S_{k_\varepsilon} g(x) - S_{k_\varepsilon} g(y)} e^{S_{n-2k_\varepsilon} g(T^{k_\varepsilon} x) - S_{n-2k_\varepsilon} g(T^{k_\varepsilon} y)} e^{S_{k_\varepsilon} g(T^{n-k_\varepsilon}x) - S_{k_\varepsilon} g(T^{n-k_\varepsilon}y)} \\
&\leqslant e^{2k_\varepsilon (\sup g - \inf g)} e^{2K}.
\end{align*}
The claim holds for $n\geqslant 2k_\varepsilon$ by taking the sup over $x$ and the inf over $y$ in $P$. Since there are only finitely many values of $n$ to correct for, by taking a larger constant, the claimed estimate holds for all $n \geqslant 1$. 

Fix some $l \geqslant 0$. Since an element $P \in \cP_{-l}^n$ is contained in a unique $\tilde{P} \in \cP_0^n$, we get that \[ \sup\limits_{P} e^{S_n g} \leqslant \sup\limits_{\tilde P} e^{S_n g} \leqslant C \inf\limits_{\tilde P} e^{S_n g} \leqslant C \inf\limits_{P} e^{S_n g}. \]

Now, assume that $\mathring{P} \neq \emptyset$. Then, by the continuity of $S_n g$ on $P$, the estimate also holds when the $\sup$ and the $\inf$ are taken over $\mathring{P}$. In other words, the claim is true for all $P \in \mathring{\cP}_{_l}^n$.

Since by \cite[Lemma 3.3]{BD2020MME}, $\mathring{\cP}_{-l}^n = \cM_{-l-1}^{n+1}$, the claim is true for all $P \in \cM_{-l}^n$, for fixed $l \geqslant 1$. We finish the proof with the case $P \in \cM_{0}^n$. Remark that letting $A \in \cM_{-1}^n$, then $T^{-1}A \in \cM_0^{n+1}$. Therefore 
\begin{align*}
e^{-\sup g} \sup\limits_{T^{-1}A} e^{S_{n+1} g} &\leqslant \sup\limits_{T^{-1}A} e^{S_{n+1} g - g} = \sup\limits_{A} e^{S_{n} g} \leqslant C \inf\limits_{A} e^{S_n g} = C  \inf\limits_{T^{-1}A} e^{S_{n+1} g - g} \\
&\leqslant C  e^{-\inf g} \inf\limits_{T^{-1}A} e^{S_{n+1} g}.
\end{align*}
Only in this last case, we need to replace $C$ by $C e^{\sup g - \inf g} \geqslant C$.
\end{proof}

\begin{proof}[Proof of Theorem~\ref{thm:pressure}]
Let $p_n = \sum\limits_{A \in \mathcal{P}_{0}^n} \sup\limits_{x \in A} e^{(S_n g)(x)}$. Then, for $k \geqslant n$, \begin{align*}
p_{n+k} &= \sum\limits_{B \cap C \in \mathcal{P}_{0}^n \bigvee T^{-n} \mathcal{P}_{0}^k } \sup\limits_{x \in B \cap C} e^{(S_n g)(x) + (S_k g)(T^n x)} \leqslant p_n \, p_k.
\end{align*}
Therefore $(\log p_n)_n$ is a sub-additive sequence. It is then classical that $\frac{1}{n}\log p_n$ converges to $\inf\limits_{n \geqslant 1} \frac{1}{n}\log p_n$, hence $P_*(T,g)$ exists. We now prove the statement about convexity. Let $g_1$ and $g_2$ be two potentials bounded from above and $p \in [0,1]$. Using the H\"older inequality, we get
\begin{align*}
\sum\limits_{A \in \mathcal{P}_{0}^n} \sup\limits_{A} e^{p \, S_n g_1 + (1-p)S_n g_2} 
\leqslant \Big( \sum\limits_{A \in \mathcal{P}_{0}^n} \sup_A e^{S_n g_1} \Big)^p  \Big( \sum\limits_{A \in \mathcal{P}_{0}^n} \sup_A e^{S_n g_2} \Big)^{1-p} \, , \quad \forall n \geqslant 1 \, .
\end{align*}
Taking the appropriate limits, we get that $P_*(T,p g_1 + (1-p)g_2) \leqslant p P_*(T,g_1) + (1-p)P_*(T,g_2)$, hence the claimed convexity.

For $(i)$, consider $\tilde{p}_n = \sum\limits_{A \in \mathring{\mathcal{P}}_{0}^n} \sup\limits_{x \in A} e^{(S_n g)(x)}$. In \cite[Lemma 3.2]{BD2020MME}, Baladi and Demers proved that $\# \{ A \in \mathcal{P}_{0}^n \mid \mathring{A} = \emptyset \} $ grows at most linearly. Hence, using the smoothness of $g$ for the last equality
\[ p_n = \sum\limits_{ A \in \mathcal{P}_{0}^n } \sup\limits_{x \in A} e^{(S_n g)(x)} \leqslant C n e^{n \sup g} + \sum\limits_{\substack{ A \in \mathcal{P}_{0}^n \\ \mathring{A}=\emptyset}} \sup\limits_{x \in A} e^{(S_n g)(x)} = C n e^{n \sup g} + \tilde p_n. \]
Since $p_n$ is submultiplicative, $p_n \geqslant e^{n P_*(T,g)}$. Now, from the assumption $P_*(T,g) - \sup g >0$, we must have $\liminf_n \tfrac{1}{n}\log \tilde p_n \geqslant P_*(T,g)$. Finally, since $\tilde p_n \leqslant p_n$, $(\frac{1}{n} \log \tilde{p}_n)_n$ converges to the same limit as $(\frac{1}{n} \log p_n)_n$ does.

For $(ii)$, we use \cite[Lemma 3.3]{BD2020MME} that $\mathring{\mathcal{P}}_{0}^n = \mathcal{M}_{-1}^{n+1}$. Hence \[ \sum\limits_{A \in \mathcal{M}_{0}^{n+1}} \sup\limits_{x \in A} e^{(S_{n+1} g)(x)} \leqslant \sum\limits_{A \in \mathcal{M}_{-1}^{n+1}} \sup\limits_{x \in A} e^{(S_{n+1} g)(x)} \leqslant \sum\limits_{A \in \mathring{\mathcal{P}}_{0}^{n}} \sup\limits_{x \in A} e^{(S_{n} g)(x)} \sup\limits_{x \in M} e^{g(x)}.\]
Furthermore, since $\mathcal{M}_{-1}^{n+1} = \mathcal{M}_{0}^{n+1} \bigvee \mathcal{M}_{-1}^{0}$, each element of $\mathcal{M}_{0}^{n+1}$ contains at most $\#\mathcal{M}_{-1}^{0}$ elements of $\mathcal{M}_{-1}^{n+1}$. Hence \[ \sum\limits_{A \in \mathcal{M}_{0}^{n+1}} \sup\limits_{x \in A} e^{(S_{n+1} g)(x)} \geqslant \frac{1}{\# \mathcal{M}_{-1}^0} \sum\limits_{A \in \mathring{\mathcal{P}}_{0}^{n}} \sup\limits_{x \in A} e^{(S_{n} g)(x)} \inf\limits_{x \in M} e^{g(x)}.\]

Point $(iii)$ (resp. $(iv)$, $(v)$) follows directly from the definition of $P_*(T,g)$ (resp. from point $(i)$, $(ii)$) and from Lemma~\ref{lemma:sup_less_poly_inf} since for all $A$ in $\cP_0^n$ (resp. $\mathring{\cP}_0^n$, $\cM_0^n$) \[ \inf\limits_{A} e^{S_n g} \leqslant \sup\limits_{A} e^{S_n g} \leqslant C_g \inf\limits_{A} e^{S_n g}. \] 

For the final claim, we prove that $\log \sum_{P \in \mathring{\cP}_{1}^{n}} \sup_{P} e^{S_n g}$ is subadditive. Take $P$ a nonempty element of $\mathring{\cP}_{1}^{n+m}$. It is the interior of an intersection of elements of the form $T^{-j}A_j$ for some $A_j \in \cP$, for $j=1,\ldots,n+m$. This is equal to the intersection of the interiors of $T^{-j}A_j$. But since $P$ is nonempty, none of the $T^{-j}A_j$ has empty interior, and so none of the $A_j$ has empty interior. Thus the interiors of $A_j$ are in $\mathring{\cP}$. Now, splitting the intersection of the first $n$ sets from the last $m$, we see that the intersection of the first $n$ sets forms an element of $\mathring{\cP}_1^n$. For the last $m$ sets, we can factor out $T^{-n}$ at the price of making the set slightly bigger:
\[ \mathrm{int}( T^{-n-j} A_{-n-j}) \subset T^{-n}(\mathrm{int}(T^{-j}(A_{-n-j})))\, , \quad 1 \leqslant j \leqslant m  \]
where $\mathrm{int}$ denotes the interior of a set. Thus
\begin{align*}
\sum\limits_{P \in \mathring{\cP}_1^{n+m} } \sup\limits_{P} \,&e^{S_{n+m}g} 
\leqslant \sum\limits_{\substack{A_{-j} \in \mathring{\cP} \\ 1\leqslant j \leqslant n+m }} \sup\{ e^{S_n g + S_m g \circ T^n}(x) \mid x \in \bigcap\limits_{j=1}^n T^{-j} A_{-j} \cap T^{-n} \bigcap\limits_{j=1}^m T^{-j} A_{-n-j} \} \\
&\leqslant \sum\limits_{\substack{A_{-j} \in \mathring{\cP} \\ 1\leqslant j \leqslant n }} \sup\{ e^{S_n g}(x) \mid x \in \bigcap\limits_{j=1}^n T^{-j} A_{-j} \} \sum\limits_{\substack{A_{-j} \in \mathring{\cP} \\ 1\leqslant j \leqslant m }} \sup \{ e^{S_m g}(x) \mid x \in \bigcap\limits_{j=1}^m T^{-j}A_{-j} \} \\
&\leqslant \sum\limits_{P \in \mathring{\cP}_1^{n} } \sup\limits_{P} e^{S_{n}g} \sum\limits_{P \in \mathring{\cP}_1^{m} } \sup\limits_{P} e^{S_{m}g} 
\end{align*}
Taking logs, the sequence is subadditive. And then so is the sequence with $\cM_0^n$ in place of $\mathring{\cP}_1^{n-1}$.
\end{proof}

\begin{proof}[Proof of Proposition~\ref{prop:pressure_bowen_definitions}.]
We first prove the claim about separating sets. Let $\ve >0$ and let $k_\ve$ be large enough so that $\diam^s(\cM_{-k_\ve -1}^0) \leqslant C \Lambda^{-k_\ve} < c_1 \varepsilon$ for some constant $c_1$ to be defined later. Therefore $\diam^u(\cM_{-k_\ve -1}^{n+1}) \leqslant C \Lambda^{-k_\ve} < c_1 \varepsilon$ for all $n \geqslant k_\ve$. By the uniform transversality between the stable and the unstable cones, we can choose $c_1$ such that $\diam(\cM_{-k_\ve -1}^{n+1}) < \ve$ for all $n \geqslant k_\ve$. 

Let $E$ be $(n,\ve)$-separated, for some $n \geqslant k_\ve$. It is shown in the proof of \cite[Lemma~3.4]{BD2020MME} that if $x,y \in E$ are distinct, then they cannot be contained in the same element of $\cP_{-k_\ve}^{k_\ve +n}$. Thus
\begin{align*}
\sum\limits_{x \in E} e^{S_n g(x)} &\leqslant \!\!\! \sum\limits_{A \in \cP_{-k_\ve}^{k_\ve +n}} \!\!\! |e^{S_n g}|_{C^0(A)} = \!\!\! \sum\limits_{A \in \cP_{0}^{2k_\ve +n}} \!\!\! |e^{S_n g \circ T^{k_\ve}}|_{C^0(A)} \leqslant e^{k_\ve(\sup g - \inf g)} \!\!\!  \sum\limits_{A \in \cP_{0}^{2k_\ve +n}} \!\!\! |e^{S_{2k_\ve + n} g}|_{C^0(A)}.
\end{align*}
Therefore $\lim\limits_{n \to \infty} \frac 1n \log \sup \{ \sum_{x \in E} e^{S_n g(x)} \mid \text{ $E$ is $(n,\ve)$-separated} \} \leqslant P_*(T,g)$, and this holds for any $\ve > 0$. Taking the limit $\ve \to 0$, we get $P_{\rm sep}(T,g) \leqslant P_*(T,g)$.

For the reverse inequality, assume that $g$ is such that $P_*(T,g) - \sup g > 0$. From the proof of \cite[Lemma~3.4]{BD2020MME}, there exists $\ve_0 >0$ such that for all $\ve < \ve_0$, any set $E$ which contains only one point per element of $\cM_0^n$ is $(n,\ve)$-separated. For all $A \in \cM_0^n$, there exists $x \in A$ such that $e^{S_n g(x)} \geqslant \tfrac{9}{10} \sup_A e^{S_n g}$. Let $E$ be the collection of such $x$. Thus
\[
\sum\limits_{x \in E} e^{S_n g(x)} \geqslant \frac{9}{10} \sum\limits_{A \in \cM_0^n} |e^{S_n g}|_{C^0(A)} \, .
\]
Therefore $ \lim\limits_{n \to \infty} \frac 1n \log \sup \{ \sum\limits_{x \in E} e^{S_n g(x)} \mid \text{ $E$ is $(n,\ve)$-separated} \} \geqslant P_*(T,g)$, and this holds for any $0 < \ve < \ve_0$. Taking the limit $\ve \to 0$, we get $P_{\rm sep}(T,g) \geqslant P_*(T,g)$, thus the claimed equality.

We now prove the claim concerning spanning sets. Let $\ve >0$ and let $k_\ve$ be such that $\diam(\cM_{-k_\ve -1}^{n+1}) < \ve$ for all $n \geqslant k_\ve$. Let $F$ be a set containing one point in each element of $\cP_{-k_\ve}^{n+1}$. From the proof of \cite[Lemma~3.5]{BD2020MME}, $F$ is $(n,\ve)$-spanning. Since 
\[
\sum\limits_{x \in F} e^{S_n g(x)} \leqslant e^{ k_\ve (\sup g - \inf g) } \sum\limits_{A \in \cP_0^{2k_\ve +n}} |e^{S_{2k_\ve +n} g} |_{C^0(A)} \,
\]
we get that $\lim\limits_{n \to \infty} \frac 1n \log \inf \{ \sum_{x \in F} e^{S_n g(x)} \mid \text{$F$ is $(n,\ve)$-spanning} \} \leqslant P_*(T,g)$, and thus for all $\ve > 0$. Taking the limit $\ve \to 0$, we get $P_{\rm span}(T,g) \leqslant P_*(T,g)$.

For the reverse inequality, assume that $g$ is a $(\cM_0^1,\alpha)$-H\"older potential such that $P_*(T,g) - \sup g >0$. Let $\ve < \ve_0$ and let $F$ be a $(n,\ve)$-spanning set. By the proof of \cite[Lemma~3.5]{BD2020MME}, each element of $\cM_0^n$ contains at least one element of $F$. Thus $\sum_{x \in F} e^{S_n g(x)} \geqslant \sum_{A \in \cM_0^n} \inf_{A} e^{S_n g}$. Taking the appropriate limits, we get that $P_{\rm span}(T,g) \geqslant P_*(T,g)$, thus the claimed equality.
\end{proof}

\subsection{Easy Direction of the Variational Principle for the Pressure}

Recall that given a $T$-invariant probability measure $\mu$ and a finite
measurable partition $\mathcal{A}$ of $M$, the entropy of $\mathcal{A}$ with respect to $\mu$ is
defined by $H_\mu(\mathcal{A}) = - \sum_{A \in \mathcal{A}} \mu(A) \log \mu(A)$, and the
entropy of $T$ with respect to $\mathcal{A}$ is
$
h_\mu(T, \mathcal{A}) = \lim_{n \to \infty} \frac 1n H_\mu\left(\bigvee_{i=0}^{n-1} T^{-i} \mathcal{A}\right) 
$.

\begin{lemma}\label{lemma:var_ppl_easy_dir}
Let $\varphi : M \to \mathbbm{R}$ be a measurable function. Then
\begin{align*}
\begin{split} P_*(T,\varphi) \geqslant P(T,\varphi)\end{split} \coloneqq \sup \{ h_{\mu}(T) + \int \varphi \, \mathrm{d}\mu \mid \mbox{$\mu$ is a $T$-invariant Borel probability measure} \}
\end{align*}
\end{lemma}

\begin{proof}
Let $\mu$ be a $T$-invariant probability measure on $M$. Notice that $\cP$ is a generator for $T$ since $\bigvee_{i=-\infty}^\infty T^{-i} \cP$ separates points in $M$. 
Thus $h_{\mu}(T) = h_{\mu}(T, \cP)$ (see for example \cite[Theorem~4.17]{Walters82ergth}).
Then,
\begin{align*}
\begin{split}
&h_{\mu}(T) +\int \varphi \,\mathrm{d}\mu = \lim\limits_{n \to \infty} \frac{1}{n} \sum\limits_{A \in \cP_0^n} \left( - \mu(A)\log \mu(A) + \int_A S_n \varphi \,\mathrm{d}\mu \right) \\
&\qquad \leqslant \lim\limits_{n \to \infty} \frac{1}{n} \sum\limits_{A \in \cP_0^n}  \mu(A) (\sup\limits_{A}(S_n \varphi) - \log \mu(A)) \leqslant \lim\limits_{n \to \infty} \frac{1}{n} \log \sum\limits_{A \in \cP_0^n} \sup\limits_{A} e^{S_n \varphi} \leqslant P_*(T,\varphi).
\end{split}
\end{align*}
where we used \cite[Lemma 9.9]{Walters82ergth} for the second inequality.
\end{proof}

\subsection{Abramov Formula and Choice of the Potential $g$}

In order to obtain the existence of MME for the billiard flow, we use the Abramov formula to relate equilibrium measure for $T$ and some potential $g$, to the MME of the flow.
First, we need the following (classical) lemma.

\begin{lemma}\label{lemma:pressure_vanishes}
Let $\varphi$ be a bounded non-negative measurable function such that $\varphi_0 \coloneqq \inf \{ \int \varphi \,\mathrm{d}\mu \mid T_*\mu = \mu \} > 0$. Then, there exists a unique real number $c_\varphi$ such that $P(T,-c_\varphi \varphi) = 0$.
\end{lemma}

\begin{proof}
We first prove that the function $t \mapsto P(T,t\varphi)$ is increasing. Let $\varepsilon >0$ and $t_1 < t_2$. There exists a $T$-invariant probability measure $\mu_1$ such that \begin{align*}
P(T,t_1\varphi) \leqslant h_{\mu_1}(T) + t_1 \int \varphi \,\mathrm{d}\mu_1 + \varepsilon 
\leqslant P(T,t_2\varphi) - (t_2-t_1)\varphi_0 + \varepsilon.
\end{align*}
By this computation, we also get that $\lim_{t \to \pm\infty} P(T,t\varphi) = \pm\infty$ .

Now we prove that $t \mapsto P(T,t\varphi)$ is continuous. Let $\varepsilon >0$ and $t \in \mathbbm{R}$. By the previous computation, we get that $\varepsilon \varphi_0 \leqslant P(T,(t+\varepsilon)\varphi) - P(T,t\varphi)$. Let $\mu_2$ be such that $P(T,(t+\varepsilon)\varphi) \leqslant h_{\mu_2}(T) + (t + \varepsilon)\int \varphi \,\mathrm{d}\mu_2 + \varepsilon$. Thus $P(T,(t+\varepsilon)\varphi) - P(T,t\varphi) \leqslant \varepsilon ( 1 + \sup \varphi)$. Therefore $t \mapsto P(T,t\varphi)$ is (strictly) increasing and continuous, so it must vanish at exactly one point, noted $-c_\varphi$.
\end{proof}

We can now use this lemma with the Abramov formula to get the following

\begin{corollary}\label{corol:eq_state_and_MME}
Equilibrium measures of $T$ under the potential $-h_{\rm top}(\phi_1)\tau$ and MME of the billiard flow (seen as a suspension flow) are in one-to-one correspondence through the bijection $\mu \mapsto \mu_\tau \coloneqq \tfrac{1}{\mu(\tau)} \mu \otimes \lambda$, where $\lambda$ is the Lebesgue measure.
\end{corollary}

\begin{proof}
Since $\tau \geqslant \tau_{\min} > 0$, the assumption of Lemma~\ref{lemma:pressure_vanishes} is satisfied for $\varphi = \tau$. Let $c$ be the constant given by Lemma~\ref{lemma:pressure_vanishes} such that $0 = P(T,-c \tau)$. Then, for every equilibrium state $m$ of $T$ under the potential $-c \tau$, we get
\[ 0 = h_m(T) -c \int \tau \,\mathrm{d}m \geqslant h_{\mu}(T) -c \int \tau \,\mathrm{d}\mu,\]
for all $T$-invariant measure $\mu$. Thus
\[ c = \frac{h_{m}(T)}{\int \tau \,\mathrm{d}m} \geqslant \frac{h_{\mu}(T)}{\int \tau \,\mathrm{d}\mu}. \]
Now, by the Abramov formula, $c = h_{m_{\tau}}(\phi_1) \geqslant h_{\mu_{\tau}}(\phi_1)$. In other words, $m_\tau$ is a MME for the billiard flow. Furthermore, since $\phi_1$ is a continuous map of a compact metric space, by \cite[Theorem 8.6]{Walters82ergth}, we get that $h_{\rm top}(\phi_1)=\sup \{ h_\mu (\phi_1) \mid (\phi_1)_* \mu = \mu \}$. Thus $c = h_{\rm top}(\phi_1)$.

To prove that the map is onto, we use that any $\phi_t$-invariant probability measure $\mu_\tau$ must be of the form $\tfrac{1}{\mu(\tau)}\mu\otimes\lambda$, for some $T$-invariant probability measure $\mu$. Thus, reversing the above computations, we get that if $\mu_\tau$ is a MME, then $\mu$ is an equilibrium state for $T$ under the potential $-h_{\rm top}(\phi_1)\tau$.
\end{proof}

Therefore, proving the existence and uniqueness of the MME for the billiard flow is equivalent to proving the existence and uniqueness of the equilibrium state of $T$ under the potential $g= -h_{\rm top}(\phi_1) \tau$. Notice that in the second case, $g$ is $(\cM_0^1,\frac{1}{2})$-H\"older continuous and the condition $P_*(T,g) - \sup g > 0$ from Theorem~\ref{thm:pressure} is realised since $P_*(T,g) - \sup g \geqslant P(T,-h_{\rm top}(\phi_1) \tau) + h_{\rm top}(\phi_1) \tau_{\min} > 0$.

\begin{remark}
Using similar arguments as in Corollary~\ref{corol:eq_state_and_MME}, we can relate the equilibrium states of $\phi_t$ under the (measurable) potential $\tilde g : \Omega \to \mathbbm{R}$ to the ones of $T$ under $g= \lambda(\tilde g) - P(\phi_1,\tilde g)\tau$, where $\lambda(\tilde g) :M \to \mathbbm{R}$ is given by \[ \lambda(\tilde g)(x)=\int_0^{\tau(x)} \tilde g (\phi_t(x)) \, \mathrm{d}t.\]
\end{remark}

\section{Growth Lemma and Fragmentation Lemmas}\label{sect:growth_lemma}

This section contains technical lemmas used throughout the rest of this paper, as well as the precise definition of the conditions SSP.1 and SSP.2. The first condition will be used to prove the ``Lasota--Yorke" bounds on the transfer operator $\cL_g$ in Proposition~\ref{prop:upper_bounds_norms}, as well as the lower bound on the spectral radius in Theorem~\ref{thm:spectral_radius}, while SSP.2 will be crucial for the absolute continuity (Corollary~\ref{corol:abs_c0_mug_unst_fol}) used to prove statistical properties (Propositions~\ref{prop:mug_is_ergodic} and~\ref{prop:mu_g_is_bernoulli}) and to compute the pressure (Corollary~\ref{thm:equilibrium states}). The first lemma (Lemma~\ref{lemma:Growth_lemma}) controls the growth in complexity of the iterates of a stable curve, with a weight $g$, whereas the subsequent lemmas are used both to obtain Proposition~\ref{prop:almost_exponential_growth} -- claiming that the thermodynamic sums grow at an exact exponential rate -- and to prove that some potentials satisfy the SSP.1 and SSP.2.

\medskip
In order to state the results from this section, we need to introduce a certain class of curves, as well as a mean to decompose them into manageable pieces.

First, denote by $\cW^s$ the set of all nontrivial smooth connected subsets $W$ of stable manifolds for $T$ so that $W$ has length at most $\delta_0$ (to be determined latter). Such curves have curvature bounded above by a fixed constant \cite[Prop.~4.29]{chernov2006chaotic}. Thus $T^{-1} \cW^s = \cW^s$, up to subdivision of curves. We define $\cW^u$ similarly from unstable manifolds of $T$.

Now, recalling the stable and unstable cones \eqref{eq:stable_cones}, we define the set of cone-stable curves $\hW^s$ containing smooth curves whose tangent vectors all lie in $\cC^s$, with length at most $\delta_0$ and curvature bounded above so that $T^{-1} \hW^s \subset \hW^s$, up to subdivision of curves. We define a set of cone-unstable curves $\hW^u$ similarly. These sets of curves will be relevant since $\cS_n$ and $\cS_{-k}$ are composed of curves in $\hW^s$ and $\hW^u$, respectively. Obviously, $\cW^s \subset \hW^s$.

\medskip
For $\delta \in (0,\delta_0]$ and $W \in \hW^s$, let $\cG^\delta_0(W) \coloneqq \{W \}$.  For $n \geqslant 1$, define the \emph{$\delta$-scaled subdivision} $\cG_n^{\delta}(W)$ inductively as the collection of smooth components of $T^{-1}(W')$ for $W' \in \cG_{n-1}^\delta(W)$, where elements longer than $\delta$ are subdivided to have length between $\delta/2$ and $\delta$. Thus $\cG_n^\delta(W) \subset \hW^s$ for each $n$ and $\cup_{U \in \cG_n^\delta(W)} U = T^{-n} W$.  Moreover, if $W \in \cW^s$, then $\cG_n^\delta(W) \subset \cW^s$.

Denote by $L_n^\delta(W)$ those elements of $\cG_n^\delta(W)$ having length at least $\delta/3$ (the \emph{long} curves), $S_n^\delta(W) \coloneqq \cG_n^\delta(W) \smallsetminus L_n^\delta(W)$ (the \emph{short} curves), and define $\cI_n^\delta(W)$ to comprise those elements $U \in \cG_n^\delta(W)$ for which $T^iU$ is contained in an element of $S_{n-i}^\delta(W)$ 
for all $0 \leqslant i \leqslant n-1$.

A fundamental fact \cite[Lemma~5.2]{chernov01} we will use is that the growth in complexity for the billiard is at most linear:
\begin{align}\label{eq:complex}
\begin{split}
\mbox{$\exists$ $K >0$ \mbox{ such that } $\forall$ $n \geqslant 0$, } & \mbox{the number of curves in $\cS_{\pm n}$ that intersect} \\
& \mbox{at a single point is at most $Kn$.}
\end{split}
\end{align}

\subsection{Growth Lemma}

\begin{lemma}\label{lemma:Growth_lemma}
For any $m \in \mathbbm{N}$, there exists $\delta_0 = \delta_0(m) \in (0,1)$ such that for all $W \in \hW^s$, if $|W|<\delta_0$, then for all $0 \leqslant l \leqslant 2m$, $T^{-l}W$ comprises at most $Km+1$ connected components.\\
Furthermore, for any $\delta \in (0,\delta_0]$, the $\delta$-scaled subdivisions satisfy the following estimates: for all $n \geqslant 1$, all $\gamma \in [0,\infty)$, all $W \in \hW^s$, and all $\cM_0^1$-continuous potential $g$, we have
\begin{itemize}
\item[a)] $\displaystyle \sum\limits_{W_i \in \cI_n^\delta(W)} \left( \frac{\log |W|}{\log |W_i|} \right)^\gamma |e^{S_n g}|_{C^0(W_i)} \leqslant 2^{( (n \vee n_0) s_0 +1)\gamma+1}(Km+1)^{n/m} e^{n \sup g} $
\item[b)] $\displaystyle \sum\limits_{W_i \in \cG_n^\delta(W)} \left( \frac{\log |W|}{ \log |W_i|} \right)^{\gamma} |e^{S_n g}|_{C^{0}(W_i)} \leqslant \min \left\lbrace 2C \delta^{-1} 2^{((n \vee n_0)s_0 + 1)\gamma} \sum\limits_{A \in \cM_0^n} |e^{S_n g}|_{C^{0}(A)}\, , \right.$

\flushright{$\displaystyle \left.  2^{2\gamma +1} C\delta^{-1} \sum\limits_{j=1}^n 2^{(j \vee n_0)s_0 \gamma}(Km+1)^{j/m} e^{j \sup g} \sum\limits_{A \in \cM_0^{n-j}} |e^{S_{n-j} g}|_{C^{0}(A)} \right\rbrace  $}
\end{itemize}
where $(n \vee n_0)= \max(n,n_0)$. \\
Moreover, if $|W| \geqslant \delta/2$, then both factors $2^{(ns_0 + 1)\gamma}$ can be replaced by $2^\gamma$.
\end{lemma}

\begin{proof}
By \cite[Exercise 4.50]{chernov2006chaotic}, there exist constants $\delta_{\rm CM}>0$ and $C\geqslant 1$ such that for all $W \in \hW^s$ with $|W|< \delta_{\rm CM}$, then $|T^{-1}W| \leqslant C |W|^{1/2}$. Notice also that there exists $\Lambda_1 \coloneqq \Lambda_1(\varphi_0)$ such that for $W \in \hW^s$ with $T^{-1}W \cap \{ |\varphi| > \varphi_0 \} = \emptyset$, then $|T^{-1}W| \leqslant \Lambda_1 |W|$. We want to combine these bounds to estimate $|T^{-n}W|$.

Let $\delta \in (0, \delta_{\rm CM}]$, $W \in \hW^s$ with $|W|<\delta$, and $W_i \in \cI_n^\delta(W)$. Let $V \subset W$ corresponding to $W_i$, that is $V = T^n W_i$. Thus, for all $1 \leqslant j \leqslant n$, we have $|T^{-j}V|=|T^{n-j}W_i| \leqslant \delta/3$.

We can decompose $V = \bigsqcup\limits_{i_0 \in I_0} V^{0}_{i_0,{\rm graz}} \cup \bigsqcup\limits_{j_0 \in J_0} V^{0}_{j_0,{\rm exp}}$ such that: for all $i_0 \in I_0$, $T^{-1}V^{0}_{i_0,{\rm graz}} \subset \{ |\varphi| \geqslant \varphi_0 \}$, and thus $|T^{-1}V^{0}_{i_0,{\rm graz}}|  \leqslant C |V^{0}_{i_0,{\rm graz}}|^{1/2}$, and for all $j_0 \in J_0$, $T^{-1}V^{0}_{j_0,{\rm exp}} \subset \{ |\varphi| < \varphi_0\}$, and thus $|T^{-1}V^{0}_{j_0,{\rm exp}}|  \leqslant \Lambda_1 |V^{0}_{j_0,{\rm exp}}|$. We can perform the same decomposition for each $V^{0}_{i_0,{\rm graz}}$ or $V^{0}_{j_0,{\rm exp}}$ instead of $V$:
\begin{align*}
V^{0}_{i_0,{\rm graz}} = \bigsqcup\limits_{i_1} V^{1,i_0}_{i_1,{\rm graz}} \cup \bigsqcup\limits_{j_1} V^{1,i_0}_{j_1,{\rm exp}} \,\, , \qquad
V^{0}_{j_0,{\rm exp}} = \bigsqcup\limits_{i_1} V^{1,j_0}_{i_1,{\rm graz}} \cup \bigsqcup\limits_{j_1} V^{1,j_0}_{j_1,{\rm exp}}.
\end{align*}
We iterate these decompositions until having a decomposition of $T^{-n}V=W_i$. Notice that since the stable curves $T^{-j}V$ have length at most $\delta/3 \leqslant \delta_{\rm CM}/3$ and are uniformly transverse to $\cS_0$, they can cross $\{ |\varphi| \geqslant \varphi_0 \}$ at most $B$ times, where $B>0$ is a constant uniform in $W$. Thus, at each step of the decomposition, a curve is split into at most $2B$ pieces.

Thus $W_i=T^{-n}V = \bigsqcup_{*, \alpha_k} V^{n,\alpha_0,\ldots,\alpha_{n-1}}_{\alpha_n,*}$, where $* \in \{ {\rm graz, exp} \}, \,  \alpha_k \in I_k \sqcup J_k$, and the union is made of at most $(2B)^n$ elements we can estimate the length.

Consider first the case $n \leqslant n_0$. By definition, $s_0$ is such that $s_0 = \sup_{M} \frac{1}{n_0} \sum_{k=0}^{n_0-1} \mathbbm{1}_{\{ |\varphi|> \varphi_0\} } \circ T^k <1$. Thus, for each $V^{n,\alpha_0,\ldots,\alpha_{n-1}}_{\alpha_n,*}$ there are at most $s_0 n_0$ indices $\alpha_k \in I_k$. Hence $|V^{n,\alpha_0,\ldots,\alpha_{n-1}}_{\alpha_n,*}| \leqslant C^2 \Lambda_1^{n_0} |V|^{2^{-s_0 n_0}}$. Therefore
\begin{align}\label{eq:sqrt_bound}
|W_i| = |T^{-n_0}V| \leqslant (2B)^{n_0} C^2 \Lambda_1^{n_0} |V|^{2^{-s_0 n_0}} \leqslant \tilde C |W|^{2^{-s_0 n_0}} \, , \quad \forall \, W_i \in \cI_{n}^\delta(W), \, n\leqslant n_0 , \, \delta \leqslant \delta_{\rm CM}
\end{align}

Now, consider the case $n=kn_0+l$, for $k \geqslant 1$ and $0\leqslant l < n_0$. By construction, if $W_i \in \cI_n^\delta(W)$, then $T^{l} W_i \subset W_i^0 \in \cI_{kn_0}^\delta(W)$ and $T^{n_0} W_i^j \subset W_i^{j+1} \in \cI_{(k-j-1)n_0}^\delta(W)$ for all $0 \leqslant j \leqslant k-1$. Thus, we can iterate \eqref{eq:sqrt_bound}:
\begin{align*}
|W_i| \leqslant \tilde C |W_i^0|^{2^{-s_0 n_0}} \leqslant \tilde C^{\sum_{m=0}^{j} 2^{-m s_0 n_0}} |W_i^{j}|^{2^{-j s_0 n_0}} \leqslant \tilde C^2 |W|^{2^{-(k+1)s_0 n_0}} 
\end{align*}
and so $|W_i| \leqslant \tilde C^2 |W|^{2^{-ns_0}}$ for all $W_i \in \cI_n^\delta(W)$, $n\geqslant n_0$ and all $W \in \hW^s$ with $|W| < \delta_{\rm CM}$.

Therefore, if $\delta \leqslant \min(\tilde C^{-2}, \delta_{\rm CM})$, we have
\begin{align}\label{eq:ratio_of_logs}
\left( \frac{\log|W|}{\log |W_i|} \right)^{\gamma} \leqslant \left( 2^{s_0n}\left( 1 - \frac{\log \tilde C^2}{\log|W_i|} \right)\right)^{\gamma} \leqslant 2^{(ns_0 +1)\gamma}, \, \forall \, W_i \in \cI_n^\delta(W),
\end{align}
since $|W_i| \leqslant \delta$. In the case $|W| \geqslant \delta /2$, since $\delta < 2$, we can directly obtain that the ration of logs is bounded by $2^\gamma$.

\smallskip
\noindent
(a) Let $m \geqslant 1$ and $W \in \hW^s$ with $|W|<\delta \leqslant \min(\tilde C^{-2}, \delta_{\rm CM})$. First, we want to estimate the number of smooth components of $T^{-l}W$, for $0\leqslant l \leqslant 2m$. The problem is the same as knowing the number of connected components of $W \smallsetminus \cS_{-l}$. Now, by \eqref{eq:complex}, at most $Kl$ curves in $\cS_{-l}$ can intersect at a given point. Since $W$ and $\cS_{-l}$ are uniformly transverse, for each $0 \leqslant l \leqslant 2m$ there exists $\delta_{(l)}$ such that if $|W|<\delta_{(l)}$ then $W \smallsetminus \cS_{-l}$ has at most $Kl+1$ connected components. Let $\delta_0 \coloneqq \min \{ \delta_{(l)} \mid 0\leqslant l \leqslant 2m \}$.

Let $n \geqslant 1$, $\delta \in (0 , \delta_0]$ and $W \in \hW^s$ with $|W| < \delta$. We want to estimate $\# \cI_n^\delta(W)$. We prove by induction that $\# \cI_{jm}^\delta(W) \leqslant (Km+1)^j$. For $j=1$, this follows from the choice of $\delta_0$. Since elements of $\cI_{(j+1)m}^\delta(W)$ are of the form $V \in \cI_{m}^\delta(W_i)$ for $W_i \in \cI_{jm}^\delta(W)$, we have 
\[ \# \cI_{(j+1)m}^\delta(W) \leqslant (Km+1)\# \cI_{jm}^\delta(W) \leqslant (Km+1)^{j+1}. \]
Now for estimating $\# \cI_{jm +l}^\delta(W)$, $0\leqslant l <m$, we only need to modify the last step:
\[ \# \cI_{jm+l}^\delta(W) \leqslant (K(m+l)+1)\# \cI_{(j-1)m}^\delta(W) \leqslant 2(Km+1)^{j}. \]
Therefore, $\# \cI_{n}^\delta(W) \leqslant 2(Km+1)^{n/m}$, for all $n \geqslant 1$. Combining this estimate with \eqref{eq:ratio_of_logs} and $e^{S_n g} \leqslant e^{n \sup g}$, we obtain (a).

\smallskip
\noindent
(b) Let $\delta \leqslant \delta_0$, and $W \in \hW^s$ with $|W|< \delta$. We start by estimating $\sum_{W_i \in \cG_n^\delta(W)} |e^{S_n g}|_{C^0(W_i)}$. Since the boundary of elements of $\cM_{-n}^0$ is contained in $\cS_{-l}$, by uniform transversality, each element of $\cM_{-n}^0$ is crossed at most one time by $W$. Thus, each element of $\cM_0^n$ is crossed at most one time by $T^{-n}W$. Now, since the diameter of elements of $\cM_0^n$ is bounded uniformly in $n$ by some constant $C$, there can be no more than $2C \delta^{-1}$ elements of $\cG_n^\delta(W)$ in a single element of $\cM_0^n$. Thus
\begin{align}\label{eq:G_n_leqs_M_0^n}
\sum\limits_{W_i \in \cG_n^\delta(W)} |e^{S_n g}|_{C^0(W_i)} \leqslant 2C \delta^{-1} \sum\limits_{A \in \cM_0^n} |e^{S_n g}|_{C^0(A)}
\end{align}

First, in the case $|W| \geqslant \delta/2$, the estimate 
\begin{align*}
\sum\limits_{W_i \in \cG_n^\delta(W)} \left( \frac{\log|W|}{\log|W_i|} \right)^{\bar{\gamma}} |e^{S_n g}|_{C^0(W_i)} \leqslant 2^{1+ \bar{\gamma}} C\delta^{-1} \sum\limits_{A \in \cM_0^n} |e^{S_n g}|_{C^0(A)},
\end{align*}
is enough for what we need.

Now, assume that $|W| < \delta/2$. Let $F_1(W)$ denote those $V \in \cG_1^\delta(W)$ whose length is at least $\delta/2$. Inductively, define $F_j(W)$, for $2 \leqslant j \leqslant n-1$, to contain those $V \in \cG_j^\delta(W)$ whose length is at least $\delta/2$, and such that $T^kV$ is contained in an element of $\cG_{j-k}^\delta(W) \smallsetminus F_{j-k}(W)$ for any $1 \leqslant k \leqslant j-1$.  Thus $F_j(W)$ contains elements of $\cG_j^\delta(W)$ that are ``long for the first time'' at time $j$.

We gather the $W_i \in \cG_n^\delta(W)$ according to their ``first long ancestors'' as follows.  We say $W_i$ has \emph{first long ancestor}\footnote{Note that ``ancestor'' refers to the backwards dynamics mapping $W$ to $W_i$.} $V \in F_j(W)$ for $1\leqslant j \leqslant n-1$ if $T^{n-j}W_i \subseteq V$.  Note that such a $j$ and $V$ are unique for each $W_i$ if they exist.  If no such $j$ and $V$ exist, then $W_i$ has been forever short
and so must belong to $\cI_n^\delta(W)$. Denote by $A_{n-j}(V)$ the set of $W_i \in \cG_n^\delta(W)$ whose first long ancestor is $V \in F_j(W)$, that is
\[ A_{n-j}(V) \coloneqq \{ W_i \in \cG_n^\delta(W) \mid T^{n-j} W_i \subset V \}. \]
By construction, we thus have the partition
\[ \cG_n^\delta(W) = \cI_n^\delta(W) \sqcup \bigsqcup\limits_{j=1}^{n-1} \bigsqcup\limits_{V \in F_j(W)} A_{n-j}(V).\]
Therefore
\begin{align*}
&\sum\limits_{W_i \in \cG_n^\delta(W)} \left( \frac{\log |W|}{ \log |W_i|} \right)^{\gamma} |e^{S_n g}|_{C^{0}(W_i)} \\
&\quad= \sum\limits_{j=1}^{n-1} \sum\limits_{V_l \in F_j(W)} \sum\limits_{W_i \in A_{n-j}(V_l)} \left( \frac{\log |W|}{ \log |W_i|} \right)^{\gamma} |e^{S_n g}|_{C^{0}(W_i)} + \sum\limits_{W_i \in \cI_n(W)} \left( \frac{\log |W|}{ \log |W_i|} \right)^{\gamma} |e^{S_n g}|_{C^{0}(W_i)} \\
&\quad\leqslant \sum\limits_{j=1}^{n-1} \sum\limits_{V_l \in F_j(W)}  \left( \frac{\log |W|}{ \log |V_l|} \right)^{\gamma}  |e^{S_j g}|_{C^{0}(V_l)} \sum\limits_{W_i \in A_{n-j}(V_l)} \left( \frac{\log |V_l|}{ \log |W_i|} \right)^{\gamma} |e^{S_{n-j} g}|_{C^{0}(W_i)} \\
&\quad\qquad + 2^{((n \vee n_0)s_0 +1)\gamma}(Km+1)^{n/m} e^{n \sup g} \\
&\quad\leqslant 2^{\gamma + 1} C \delta^{-1} \sum\limits_{j=1}^{n-1} \sum\limits_{V_l \in F_j(W)}  \left( \frac{\log |W|}{ \log |V_l|} \right)^{\gamma} |e^{S_j g}|_{C^{0}(V_l)} \sum\limits_{A \in \cM_0^{n-j}} |e^{S_n g}|_{C^{0}(A)} \\
&\quad\qquad + 2^{((n \vee n_0)s_0 +1)\gamma}(Km+1)^{n/m} e^{n \sup g} \\
&\quad\leqslant 2^{2\gamma +1} C \delta^{-1} \sum\limits_{j=1}^{n} 2^{(j \vee n_0)s_0 \gamma}(Km+1)^{j/m} e^{j\sup g} \sum\limits_{A \in \cM_0^{n-j}} |e^{S_n g}|_{C^{0}(A)}.
\end{align*}
where we have applied part (a) from time 1 to time $j$ and the first estimate in part (b) from time $j$ to time $n$, since each $|V_\ell| \geqslant \delta/2$.
\end{proof}

\subsection{Fragmentation Lemmas}\label{subsect:Frag_lemmas--round_1}

Until now, we have only used mild assumptions on a given potential $g$. Here, we introduce the conditions of Small Singular Pressure (SSP.1 and SSP.2), which are crucial for the estimates given in the lemmas from this subsection. These estimates will be used in Section~\ref{subsect:supermultiplicativity}, leading to a precise growth rate of the thermodynamic sums. We also prove that there exist potentials satisfying simultaneously SSP.1 and SSP.2.

In what follows, we will always assume that $P_*(T,g) - \sup g > s_0 \log 2$ and choose some $\delta_0$ accordingly: Let $m$ be large enough so that $\tfrac{1}{m} \log(Km+1) < P_*(T,g) - \sup g - s_0 \log 2$, and let $\delta_0=\delta_0(m)$ be as in Lemma~\ref{lemma:Growth_lemma}. Notice that $m$, and therefore also $\delta_0$, depend on $g$.

\medskip
We start by giving the definition of SSP.1. First, we have to introduce some notations.

Let $L_u^{\delta}(\cM_{-n}^0)$ denote the elements of $\cM_{-n}^0$ whose unstable diameter\footnote{Recall that the unstable diameter of a set is the length of the longest unstable curve contained in that set.} is at least $\delta/3$, for some $\delta \in (0,\delta_0]$. Similarly, let $L_s^{\delta}(\cM_0^n)$ denote the elements of $\cM_0^n$ whose stable diameter is at least $\delta/3$. Recall that the boundary of the partition formed by $\cM_{0}^n$ is comprised of stable curves belonging to $\cS_{n} = \cup_{j=0}^n T^{-j} \cS_0 \subset \hW^s $.

Define 
\begin{align*}
\ell_n^s(g,\delta) \coloneqq \inf \{ \sum_{V \in L_n^{\delta}(W)} |e^{S_n g}|_{C^0(V)} \mid W \in \hW^s, \, \tfrac{\delta}{3} \leqslant |W| \leqslant \delta \}, 
\end{align*}
and its ``time reversal" $\ell_n^u(g,\delta)$ similarly by replacing $\hW^s$ with $\hW^u$, $T$ with $T^{-1}$ and $g$ with $g \circ T^{-1}$ (that is $S_n g$ with $S^{-1}_n g \coloneqq \sum_{i=1}^n g \circ T^{-i} = S_n g \circ T^{-n}$)\footnote{Actually, this is equivalent to simply replace $g$ by the $(\cM_0^1, \alpha)$-H\"older potential $g \circ \cI \circ T$, where $\cI(r,\varphi)=(r,-\varphi)$. Notice that since $\cI \circ T^n$ acts like a permutation on $\cM_0^n$, we have $P_*(T,g)=P_*(T,g\circ \cI \circ T)$.}.

\begin{definition}[SSP.1]\label{def:SSP1}
A potential $g$ such that $P_*(T,g) - \sup g > s_0 \log 2$ is said to have $\varepsilon$-SSP.1 (small singular pressure), for some $\varepsilon>0$, if 
\begin{align}\label{SSP.1}
&\mbox{there exist } \delta = \delta(\ve) \in (0,\delta_0] \mbox{ and } n_1 \in \mathbbm{N} \mbox{ such that} \\
\nonumber &\frac{\sum\limits_{W_i \in L_n^\delta(W)} |e^{S_n g}|_{C^0(W_i)} }{ \sum\limits_{W_i \in \cG_n^\delta(W)} |e^{S_n g}|_{C^0(W_i)} } \geqslant \frac{1 - 2\varepsilon}{1 - \varepsilon}, \quad \mbox{$\forall n \geqslant n_1$ $\forall W \in \hW^s$ with $|W| \geqslant \delta/3$ \, ,}
\end{align}
\begin{align}\label{SSP.3}
\text{ the sequences } (e^{n \sup g} \, \ell_n^s(g,\delta)^{-1})_{n \geqslant n_1} \text{ and } (e^{n \sup g} \, \ell_n^u(g,\delta)^{-1})_{n \geqslant n_1} \text{ are summable,}
\end{align}
and the ``time reversal" of \eqref{SSP.1} holds\footnote{Here again, by time reversal of \eqref{SSP.1} we mean the same estimate but replacing $T$ with $T^{-1}$, $\hW^s$ with $\hW^u$ and $g$ with $g \circ T^{-1}$ (that is $S_n g$ with $S_n^{-1}g$). Here also, this is equivalent to just replace $g$ by $g \circ \cI \circ T$.}. Notice that \eqref{SSP.1} (resp. its time reversal) implies that $\ell_n^s(g,\delta)$ (resp. $\ell_n^u(g,\delta)$) is nonzero for all $n \geqslant n_1$.

A potential is said to have SSP.1 if it has $\varepsilon$-SSP.1 for some $\varepsilon \leqslant 1/4$.
\end{definition}

The following lemma bootstraps from Lemma~\ref{lemma:Growth_lemma} and will be crucial to get the lower bound on the spectral radius:

\begin{lemma}\label{lemma:short_curves_rare}
If $g$ is a $(\cM_0^1,\alpha)$-H\"older potential such that $P(T,g) - \sup g > s_0 \log 2$ and $\log \Lambda > \sup g - \inf g$, then $g$ satisfies \eqref{SSP.1}, as well as its time reversal, for all $\ve >0$. 
\end{lemma}

\begin{proof}
Fix $\varepsilon > 0$. Choose $n_1$ so large that $ 6 C C_1^{-1} n_1 (Kn_1 +1) e^{n_1(\sup g - \inf g - \log \Lambda)} < \varepsilon $, where $C$ is the constant from Lemma~\ref{lemma:sup_less_poly_inf} and $C_1$ is such that $|T^{-n}W| > C_1 \Lambda^n |W|$ whenever $W \in \hW^s$. Next, choose $\delta > 0$ sufficiently small that if $W \in \hW^s$ with $|W| < \delta$, then $T^{-n} W$ comprises at most $Kn+1$ smooth pieces of length at most $\delta_0$ for all $0 \leqslant n \leqslant 2n_1$ (where $\delta_0$ is defined using Lemma~\ref{lemma:Growth_lemma} and the assumption $P(T,g) - \sup g > s_0 \log 2$, as described above).

Let $W \in \hW^s$ with $|W| \geqslant \delta/3$. We shall prove the following equivalent inequality for $n \geqslant n_1$: 
\begin{align}\label{eq:Sn_over_Gn}
 \frac{\sum\limits_{W_i \in S_n^\delta(W)} |e^{S_n g}|_{C^0(W_i)} }{ \sum\limits_{W_i \in \cG_n^\delta(W)} |e^{S_n g}|_{C^0(W_i)} } \leqslant \frac{\varepsilon}{1 - \varepsilon}.
\end{align}
For $n \geqslant n_1$, write $n = k n_1 + l$ for some $0 \leqslant l < n_1$. If $k=1$, the above inequality is clear since $S^\delta_{n_1 + l}(W)$ contains at most $K(n_1+l)+1$ components by assumption on $\delta$ and $n_1$, while $|T^{-n_1 -l}W| \geqslant C_1 \Lambda^{n_1+l} |W| \geqslant C_1 \Lambda^{n_1+l} \delta/3$. Thus $\cG^\delta_n(W)$ must contain at least $C_1 \Lambda^{n_1+l} /3$ curves since each has length at most $\delta$. Thus,
\[ \frac{\sum\limits_{W_i \in S^\delta_n(W)} |e^{S_n g}|_{C^0(W_i)} }{\sum\limits_{W_i \in \cG^\delta_n(W)} |e^{S_n g}|_{C^0(W_i)}} \leqslant 3\frac{K(n_1 +l)+1}{C_1 \Lambda^{n_1+l}}\frac{e^{(n_1+l)\sup g}}{e^{(n_1+l)\inf g}} \leqslant 6 C_1^{-1} (Kn_1 +1) e^{n_1(\sup g - \inf g - \log \Lambda)} <\varepsilon,\]
where the last inequality holds by choice of $n_1$.

For $k>1$, we split $n=kn_1 +l$, $0\leqslant l < n_1$, into $k-1$ blocks of length $n_1$ and the last block of length $n_1 +l$. For each $V \in \cG_n^\delta(W) \smallsetminus \cI_n^\delta(W)$, let $j<n$ be the greatest integer such that $T^{n-j}V$ is contained in an element $V_a$ of $L_j^\delta(W)$ and for all $j<i<n$, $T^{n-i}V$ is contained in an element of $S_i^\delta(W)$. We call $V_a$ the \emph{most recent long ancestor} of $V$ and $j$ its \emph{age}. If such a $j$ does not exist, it means that for all $i < n$, $T^{n-i}V$ is short, that is $V \in \cI_n^\delta(W)$ and we set $j=0$ in this case.

We group elements of $S_n^\delta(W)$ by their age in $[(j-1) n_1, j n_1 -1]$, $1 \leqslant j \leqslant k-1$, and $[(k-1)n_1,n-1]$. In other words, we consider the following partition
\begin{align}\label{eq:decomposition_Sn(W)}
S_n^\delta(W) =
\bigsqcup\limits_{q=0}^{k-2} \left( \bigsqcup\limits_{j=qn_1}^{(q+1)n_1-1} \bigsqcup\limits_{V \in L_j^\delta(W)} \cI_{n-j}^\delta(V) \right)  
\sqcup \left( \bigsqcup\limits_{j=(k-1)n_1}^{n-1} \bigsqcup\limits_{V \in L_j^\delta(W)} \cI_{n-j}^\delta(V) \right).
\end{align}
We can therefore split the left hand side of \eqref{eq:Sn_over_Gn} into two manageable parts. For this, we rely on Lemma~\ref{lemma:Growth_lemma} for $\gamma=0$ and the fact that
\[ \cG_{n}^\delta(W) \supset \bigsqcup\limits_{V \in L_j^\delta(W)} \cG_{n-j}^\delta(V), \quad \forall \, 0 < j < n.\]
Thus, using Lemma~\ref{lemma:sup_less_poly_inf}, we have
\begin{align*}
&\frac{
	\sum\limits_{q=0}^{k-2} \sum\limits_{j=qn_1}^{(q+1)n_1-1} \!\!\!\! \sum\limits_{V \in L^\delta_{j}(W)} \sum\limits_{W_i \in \cI^\delta_{n-j}(V)} \!\!\!\!\!\! |e^{S_n g}|_{C^0(W_i)} 
	}
	{
	 \sum\limits_{W_i \in \cG_{kn_1+l}^\delta(W)} |e^{S_n g}|_{C^0(W_i)} 
	} 
	\leqslant \sum\limits_{q=0}^{k-2} \sum\limits_{j=qn_1}^{(q+1)n_1-1} \!\! \frac{ \sum\limits_{V \in L^\delta_{j}(W)} \!\!\!\! |e^{S_{j} g}|_{C^0(V)} \!\!\!\!\! \sum\limits_{W_i \in \cI^\delta_{n-j}(V)} \!\!\!\!\!\! |e^{S_{n-j} g}|_{C^0(W_i)}   }{ C^{-1} \!\!\!\!\! \sum\limits_{V \in L^\delta_{j}(W)} \!\!\!\! |e^{S_{j} g}|_{C^0(V)} \!\!\!\!\! \sum\limits_{W_i \in \cG^\delta_{n-j}(V)} \!\!\!\!\!\! e^{(n-j) \inf g}}  \\
&\qquad\qquad\qquad\qquad\quad \leqslant \sum\limits_{q=0}^{k-2} 6 C C_1^{-1} n_1 (Kn_1 + 1)^{k-q} e^{(k-q)n_1(\sup g - \inf g - \log \Lambda)} 
\leqslant \sum\limits_{q=0}^{k-2} \varepsilon^{k-q} = \sum\limits_{q=2}^{k} \varepsilon^{q}.
\end{align*}
Similarly, for the second part we have
\begin{align*}
\frac{
	\sum\limits_{j=(k-1)n_1}^{n-1} \sum\limits_{V \in L^\delta_{j}(W)} \sum\limits_{W_i \in \cI^\delta_{n-j}(V)} \!\!\! |e^{S_n g}|_{C^0(W_i)} 
	}
	{
	 \sum\limits_{W_i \in \cG_{kn_1+l}^\delta(W)} |e^{S_n g}|_{C^0(W_i)} 
	} 
	&\leqslant \!\! \sum\limits_{j=(k-1)n_1}^{n-1} \!\! \frac{ \sum\limits_{V \in L^\delta_{j}(W)} \!\!\! |e^{S_j g}|_{C^0(V)} \!\!\! \sum\limits_{W_i \in \cI^\delta_{n-j}(V)} \!\!\! |e^{S_{n-j} g}|_{C^0(W_i)}  }{ C^{-1} \!\!\! \sum\limits_{V \in L^\delta_{j}(W)} \!\!\! |e^{S_j g}|_{C^0(V)} \!\!\! \sum\limits_{W_i \in \cG^\delta_{n-j}(V)}  \!\!\! e^{(n-j) \inf g} } \\
&\leqslant 6 C C_1^{-1} n_1 (Kn_1 +1) e^{n_1(\sup g - \inf g - \log \Lambda)} \leqslant \varepsilon.
\end{align*}
Summing these two estimates, we obtain \eqref{eq:Sn_over_Gn}.

The time reversal is obtained from the same proof by changing the construction of the set $\cG_n^\delta(W)$ (and thus $L_n^\delta(W)$, $S_n^\delta(W)$ and $\cI_n^\delta(W)$) so that elements of $\cG_n^\delta(W)$ are contained in $T^n W$ (instead of $T^{-n}W$) for $W \in \hW^u$.
\end{proof}

Notice that if $\ve \leqslant 1/4$ and $\delta_1 \leqslant \delta_0$ and $n_1$ are the corresponding $\delta$ and $n_1$ from the $\ve$-SSP.1 condition, then we have for all $W \in \hW^s$ with $|W| \geqslant \delta_1/3$ and $n \geqslant n_1$
\begin{equation}\label{eq:delta1}
\sum_{W_i \in L_n^{\delta_1}(W)} |e^{S_n g}|_{C^0(W_i)} \geqslant \frac 23 \sum_{W_i \in \cG_n^{\delta_1}(W)} |e^{S_n g}|_{C^0(W_i)} \,,
\end{equation}
In particular, since $\cG_n^{\delta_1}(W) = L_n^{\delta_1}(W) \sqcup S_n^{\delta_1}(W)$, we also get that \begin{align*}
\sum_{W_i \in L_n^{\delta_1}(W)} |e^{S_n g}|_{C^0(W_i)} \geqslant 2 \sum_{W_i \in S_n^{\delta_1}(W)} |e^{S_n g}|_{C^0(W_i)} \, .
\end{align*}

The following lemma will be used to get both lower and upper bounds on the spectral radius via Proposition~\ref{prop:GnW_geq_M0n}:

\begin{lemma}\label{lemma:Lu_geq_M-n0}
Let $g$ be a $(\cM_0^1,\alpha)$-H\"older potential such that $P_*(T,g) - \sup g > s_0 \log 2$ and which has SSP.1. 
Let $\delta_1$ and $n_1$ be the corresponding parameters associated with SSP.1. Then there exist $C_{n_1}>0$ and $n_2 \geqslant n_1$ such that for all $n \geqslant n_2$,
\begin{align}\label{eq:long_cells_geq_all}
\begin{split}
\sum\limits_{A \in L_u^{\delta_1}(\cM_{-n}^0)} \!\!\!\!\!\!\!\!\!\! |e^{S^{-1}_n g}|_{C^0(A)} &\geqslant C_{n_1} \delta_1  \!\!\! \sum\limits_{A \in \cM_{-n}^0} \!\!\!\! |e^{S^{-1}_n g}|_{C^0(A)} \,\,\text{ and } \!\!\!\!\!
\sum\limits_{A \in L_s^{\delta_1}(\cM_{0}^n)} \!\!\!\!\!\!\!\! |e^{S_n g}|_{C^0(A)} \geqslant C_{n_1} \delta_1 \!\!\! \sum\limits_{A \in \cM_{0}^n} \!\! |e^{S_n g}|_{C^0(A)}.
\end{split}
\end{align}
Furthermore, if $g$ is a $(\cM_0^1,\alpha)$-H\"older potential with $P_*(T,g) - \sup g > s_0 \log 2$ and $\log \Lambda > \sup g - \inf g$, then $g$ has SSP.1.
\end{lemma}

\begin{proof}
We prove the lower bound for $L_s^{\delta_1}(\cM_0^n)$. The lower bound for $L_u^{\delta_1}(\cM_{-n}^0)$ then follows by time reversal. First, we need to define sets that will be relevant only here. Let 
\[ I_s(\cM_0^n) \coloneqq \{ A \in \cM_0^n \mid \mathrm{diam}^s (A) < \delta_1/3 \} \] 
be the complement of $L_s^{\delta_1}(\cM_0^n)$ in $\cM_0^n$, and 
\[I_s(T^{-j}\cS_0) \coloneqq \{ \text{unstable curves in $T^{-j}(\cS_0)$ with length less than $\delta_1/3$}\} . \] 
Define also $L_s(T^{-j}\cS_0)$ as the complement of $I_s(T^{-j}\cS_0)$ in $\cG_j^{\delta_1}(\cS_0)$.

We will deduce the claim by estimating the sum of norms of $e^{S_n g}$ over $I_s(\cM_0^n)$ by the one over $L_s^{\delta_1}(\cM_0^n)$. To do so, we estimate the sum over $I_s(\cM_0^n)$ with the sums over $I_s(T^{-j}\cS_0)$. Then, using \eqref{SSP.1} we estimate the sum over $I_s(T^{-j}\cS_0)$ with sums over $L_s(T^{-j}\cS_0)$. Finally, we estimate sums over $L_s(T^{-j}\cS_0)$ with a sum over $L_s^{\delta_1}(\cM_0^n)$ and treat some troublesome terms.

\smallskip
In order to estimate the sum over $I_s(\cM_0^n)$, first remark that if $A \in \cM_0^n$ then $\partial A \subset \cS_n = \bigcup_{j=0}^n T^{-j}\cS_0$. Let $A \in I_s(\cM_0^n)$. We distinguish two cases:

\noindent
(a) For some $1 \leqslant j \leqslant n$, $\partial A$ contains a point of intersection between two curves of $T^{-j} \cS_0$. Since such intersection point is the image by $T^{-j+1}$ of an intersection point between curves of $T^{-1} \cS_0$, which are finite, and thanks to the linear complexity~\eqref{eq:complex}, we get that there are at most $K_2 n$ elements of $I_s(\cM_0^n)$ in this case.

\noindent
(b) $\partial A$ only contains intersection points between curves belonging to $T^{-j}\cS_0$ for different $j$. Let $j_A$ be the maximal $1 \leqslant j \leqslant n$ such that $A \cap T^{-j}\cS_0 \neq \emptyset$, and $\gamma \in T^{-j_A}\cS_0$ such that $\gamma \cap A \neq \emptyset$. Notice that $\gamma$ must intersect other curves from $\partial A$. These curves belong to $T^{-j} \cS_0$ for some $j < j_A$. Applying $T^j$, it appears that $\gamma$ must terminate at these intersection points, and thus $\gamma \subset \partial A$. Since $\gamma$ is a stable curve, $\gamma$ belongs to $I_s(T^{-j_A} \cS_0)$ by assumption on $A$. Finally, such a curve $\gamma$ belong to at most $2$ elements of $I_s(\cM_0^n)$.

Therefore 
\begin{align}\label{eq:Is(M)_leq_Is(TjS0)}
\sum\limits_{A \in I_s(\cM_0^n)} |e^{S_n g}|_{C^0(A)} 
&\leqslant K_2 n e^{n \sup g} + C \sum\limits_{j=1}^n \sum\limits_{W \in I_s(T^{-j}\cS_0)} |e^{S_n g}|_{C^0_+(W)} + |e^{S_n g}|_{C^0_-(W)},
\end{align}
where we have extended $e^{S_n g}$ by H\"older continuity to $W$ from both sides -- and noted $|\cdot|_{C^0_+(W)}$ and $|\cdot|_{C^0_-(W)}$ the corresponding norms -- and $C$ is the constant from Lemma~\ref{lemma:sup_less_poly_inf}.

\smallskip
In order to use \eqref{SSP.1}, we decompose $\cS_0 = \bigsqcup_{i=1}^{l_0} U_i$ where each $U_i$ is a connected curve such that $\tfrac{\delta_1}{3} \leqslant |T^{-1} U_i| \leqslant \delta_1$. But first we need to compare the sum indexed by $I_s(T^{-j}\cS_0)$ with the one indexed by $I_s(\cG_{j-1}^{\delta_1}(U_i))$. Let $W \in I_s(T^{-j}\cS_0)$. Thus, each $W \cap T^{-j}U_i$ is a single maximal smooth component of length less than $\delta_1/3$. In other words, $W \cap T^{-j}U_i \in I_s(\cG_{j-1}^{\delta_1}(U_i))$. Therefore
\begin{align}\label{eq:Is(TjS)_leq_Is(G)}
\sum\limits_{W \in I_s(T^{-j}\cS_0)} |e^{S_n g}|_{C^0_\pm(W)} \leqslant \sum\limits_{i=1}^{l_0} \sum\limits_{W \in I_s(\cG_{j-1}^{\delta_1}(T^{-1}U_i))} |e^{S_n g}|_{C^0_\pm(W)}.
\end{align}

Now, using SSP.1 \eqref{SSP.1}, in the case $j > n_1$, we get that
\begin{align}\label{eq:Is(G)_leq_Ls(G)}
\sum\limits_{W \in I_s(\cG_{j-1}^{\delta_1}(T^{-1}U_i))} |e^{S_{n} g}|_{C^0_\pm(W)} \leqslant \frac{1}{2} e^{(n-j+1)\sup g} \sum\limits_{W \in L_s(\cG_{j-1}^{\delta_1}(T^{-1}U_i))} |e^{S_{j-1} g}|_{C^0_\pm(W)}.
\end{align}

In order to estimate this last sum with the sum indexed by $L_s(\cG_{n-1}^{\delta_1}(T^{-1}U_i))$, notice that 
\[ L_s(\cG_{n-1}^{\delta_1} (T^{-1}U_i)) \supset \bigsqcup\limits_{V \in L_s(\cG_{j-1}^{\delta_1}(T^{-1}U_i))} L_s(\cG_{n-j}^{\delta_1}(V)) .\]
Thus
\begin{align*}
\sum\limits_{W \in L_s( \cG_{n-1}^{\delta_1}(T^{-1}U_i) )} &|e^{S_{n} g}|_{C^0_\pm(W)} \geqslant \sum\limits_{W \in L_s( \cG_{j-1}^{\delta_1}(T^{-1}U_i) )} \sum\limits_{V \in L_s( \cG_{n-j}^{\delta_1}(W) )} |e^{S_{n-j} g + S_{j} g \circ T^{n-j}}|_{C^0_\pm(V)} \\
&\geqslant C^{-2} e^{\inf g} \sum\limits_{W \in L_s( \cG_{j-1}^{\delta_1}(T^{-1}U_i) )} |e^{S_{j-1} g}|_{C^0_\pm(W)} \sum\limits_{V \in L_s( \cG_{n-j}^{\delta_1}(W) )} |e^{S_{n-j} g}|_{C^0_\pm(V)} \\
&\geqslant C^{-2} e^{\inf g} \, \ell^s_{n-j}(g,\delta_1)  \sum\limits_{W \in L_s( \cG_{j-1}^{\delta_1}(T^{-1}U_i) )} |e^{S_{j-1} g}|_{C^0_\pm(W)},
\end{align*}
where we used Lemma~\ref{lemma:sup_less_poly_inf} for the second inequality, and the definition of $\ell^s_{n-j}(g,\delta_1)$ for the third inequality. Notice however that \eqref{SSP.1} ensures that $\ell^s_{n-j}(g,\delta_1) \neq 0$ only for $n-j \geqslant n_1$. We will treat these troublesome $j$ afterwards. Assume for now that $n_1 \leqslant j \leqslant n-n_1$.  
Combining the above lower bound with \eqref{eq:Is(TjS)_leq_Is(G)} and \eqref{eq:Is(G)_leq_Ls(G)}, we get 
\begin{align}\label{eq:Is(TjS)_leq_exp(n)_Ls(TnS)}
\sum\limits_{W \in I_s(T^{-j}\cS_0)} |e^{S_{n} g}|_{C^0_\pm(W)}
&\leqslant \bar{C} \, e^{(n-j)\sup g} \ell^s_{n-j}(g,\delta_1)^{-1} \sum\limits_{W \in L_s(T^{-n}\cS_0)} |e^{S_{n} g}|_{C^0_\pm(W)} \, ,
\end{align}
where we used that $\bigsqcup_{i=1}^{l_0} L_s(\cG_{j-1}^{\delta_1}(T^{-1}U_i)) \subset L_s (T^{-j} \cS_0)$ -- which is true if we choose the $\delta_1$-scaling $\cG_1(T^{-j}\cS_0)$ to be adapted with the decomposition $\cS_0 = \bigsqcup_i U_i$.

Now, if $n-n_1 \leqslant j \leqslant n$, then we obtain from similar computations 
\begin{align}\label{eq:Is(TjS)_leq_exp(n)_Ls(TnS)_2}
\sum\limits_{W \in I_s(T^{-j}\cS_0)} |e^{S_{n} g}|_{C^0_\pm(W)}
&\leqslant \frac{1}{2} C^2 e^{\inf g} e^{(n-j+1)\sup g} \ell_{n_1}^s(g,\delta_1)^{-1} \sum_{W \in L_s(T^{-j-n_1}\cS_0)} |e^{S_{n_1+j}g}|_{C^0_{\pm}(W)}
\end{align}

Finally, we estimate the sum over $L_s(T^{-n}\cS_0)$ with the sum over $L_s^{\delta_1}(\cM_0^n)$. We proceed similarly as for \eqref{eq:Is(M)_leq_Is(TjS0)}: Let $W \in L_s(T^{-n}\cS_0)$. We distinguish the two following cases:

\noindent
(a) $W$ intersects another curve from $T^{-n}\cS_0$. There are at most $2K_2$ elements of $L_s(T^{-n}\cS_0)$ in this case,

\noindent
(b) $W$ does not intersect other curves from $T^{-n}\cS_0$. In that case, $W$ must be contained in the boundary of an element of $\cM_0^n$, and thus an element of $L_s^{\delta_1}(\cM_0^n)$. Now, there are at most $2C \delta_1^{-1}$ elements of $L_s(T^{-n}\cS_0)$ in the boundary of a single element of $L_s^{\delta_1}(\cM_0^n)$, where $C$ is a large enough constant depending only on the billiard table. 

Thus
\begin{align}\label{eq:Ls(TnS)_leq_Ls(M)}
\sum\limits_{W \in L_s(T^{-n}\cS_0)} |e^{S_{n} g}|_{C^0_\pm(W)} &
\leqslant 2K_2 e^{n \sup g} + C \delta_1^{-1} \sum\limits_{A \in L_s^{\delta_1}(\cM_0^n)} |e^{S_n g}|_{C^0(A)}.
\end{align}
Similarly, for all $n-n_1 \leqslant j \leqslant n$,
\begin{align}\label{eq:Ls(TnS)_leq_Ls(M)_2}
\sum\limits_{W \in L_s(T^{-n_1-j}\cS_0)} |e^{S_{n_1+j} g}|_{C^0_\pm(W)} &
\leqslant 2K_2 e^{(n_1+j)\sup g} + C \delta_1^{-1} \sum\limits_{A \in L_s^{\delta_1}(\cM_0^{n_1+j})} |e^{S_{n_1+j} g}|_{C^0(A)}.
\end{align}

Putting together \eqref{eq:Is(M)_leq_Is(TjS0)}, \eqref{eq:Is(TjS)_leq_exp(n)_Ls(TnS)} and \eqref{eq:Ls(TnS)_leq_Ls(M)}, as well as \eqref{eq:Is(TjS)_leq_exp(n)_Ls(TnS)_2} and \eqref{eq:Ls(TnS)_leq_Ls(M)_2}, we get
\begin{align*}
&\sum\limits_{A \in I_s(\cM_0^n)} |e^{S_n g}|_{C^0(A)} 
\leqslant K_2 n e^{n \sup g} + C \sum\limits_{j=1}^{n_1-1} \sum\limits_{W \in I_s(T^{-j}\cS_0)} |e^{S_n g}|_{C^0_+(W)} + |e^{S_n g}|_{C^0_-(W)}  \\
&\quad + C \! \sum\limits_{j=n_1}^{n-n_1} \sum\limits_{W \in I_s(T^{-j}\cS_0)} \!\!\!\!\!\!\!\! |e^{S_n g}|_{C^0_+(W)} + |e^{S_n g}|_{C^0_-(W)} 
 + C \!\!\!\!\!\!\! \sum\limits_{j=n-n_1+1}^{n} \sum\limits_{W \in I_s(T^{-j}\cS_0)} \!\!\!\!\!\!\!\! |e^{S_n g}|_{C^0_+(W)} + |e^{S_n g}|_{C^0_-(W)} \\
&\,\, \leqslant (K_2n + \bar{C}_{n_1})e^{n \sup g} + \bar{C} \sum_{j=n_1}^{n-n_1} e^{j\sup g}\, \ell^s_{j}(g,\delta_1)^{-1} \sum\limits_{W \in L_s(T^{-n}\cS_0)} |e^{S_{n} g}|_{C^0_+(W)} + |e^{S_{n} g}|_{C^0_-(W)} \\
& \quad + C \sum_{j=n-n_1+1}^n e^{(n-j)\sup g} \ell_{n_1}^s(g,\delta_1)^{-1} \sum\limits_{W \in L_s(T^{-j-n_1}\cS_0)} |e^{S_{j+n_1} g}|_{C^0_+(W)} + |e^{S_{j+n_1} g}|_{C^0_-(W)} \\
&\,\, \leqslant (K_2n + \bar{C}_{n_1})e^{n \sup g} + \tilde{C} \left( 2K_2 e^{n \sup g} + C \delta_1^{-1} \sum\limits_{A \in L_s^{\delta_1}(\cM_0^n)} |e^{S_n g}|_{C^0(A)} \right) \\
&\quad + \sum_{j=n-n_1+1}^{n} C'_{n_1} C'_g \left( 2K_2 e^{(n_1 + j)\sup g} + C \delta_1^{-1} \sum_{A \in L_s^{\delta_1}(\cM_{0}^{j+n_1})} |e^{S_{j+n_1}g}|_{C^0(A)} \right),
\end{align*}
where in the last inequality we used \eqref{SSP.3} and the fact that $n-n_1+1 \leqslant j \leqslant n$ is equivalent to $0 \leqslant n-j \leqslant n_1-1$, that is, in the second sum over $j$ after the second inequality symbol, the $e^{(n-j)\sup g}$ are uniformly bounded (by $C'_g$).

We now relate the sum over $L_s^{\delta_1}(\cM_{0}^{j+n_1})$ to the sum over $L_s^{\delta_1}(\cM_{0}^{n})$. To do so, notice that if $A \in L_s^{\delta_1}(\cM_{0}^{n})$, then it contains at most $B^{j+n_1-n}$ elements of $L_s^{\delta_1}(\cM_{0}^{j+n_1})$, where $B = |\cP|$. On the other hand, an element $A' \in L_s^{\delta_1}(\cM_{0}^{j+n_1})$ is contained in exactly one element of $L_s^{\delta_1}(\cM_{0}^{n})$. Thus
\begin{align*}
&\sum_{A \in L_s^{\delta_1}(\cM_{0}^{j+n_1})} |e^{S_{j+n_1}g}|_{C^0(A)} = \sum_{A \in L_s^{\delta_1}(\cM_{0}^{j+n_1})} \sum_{\substack{A' \in L_s^{\delta_1}(\cM_{0}^{n}) \\ A \subset A' }}  |e^{S_{j+n_1}g}|_{C^0(A)} \\
&\quad \leqslant \sum_{A \in L_s^{\delta_1}(\cM_{0}^{j+n_1})} \sum_{\substack{A' \in L_s^{\delta_1}(\cM_{0}^{n}) \\ A \subset A' }} e^{n_1 \sup g} |e^{S_{j}g}|_{C^0(A')} 
\leqslant \sum_{A' \in L_s^{\delta_1}(\cM_{0}^{n})}  \sum_{\substack{A \in L_s^{\delta_1}(\cM_{0}^{j+n_1}) \\ A \subset A' }} e^{n_1 \sup g} |e^{S_{j}g}|_{C^0(A')} \\
&\quad \leqslant B^{j+n_1-n} e^{n_1 \sup g} \sum_{A \in L_s^{\delta_1}(\cM_{0}^{n})} |e^{S_{j}g}|_{C^0(A)},
\end{align*}
and therefore,
\begin{align*}
&\sum_{j=n-n_1+1}^{n} \sum_{A \in L_s^{\delta_1}(\cM_{0}^{j+n_0})} |e^{S_{j+n_1}g}|_{C^0(A)} \leqslant \sum_{j=n-n_1+1}^{n} B^{j+n_1-n} e^{n_1 \sup g} \sum_{A \in L_s^{\delta_1}(\cM_{0}^{n})} |e^{S_{j}g}|_{C^0(A)} \\
&\quad \leqslant \sum_{j=n-n_1+1}^{n} B^{j+n_1-n} e^{n_1 \sup g} e^{(n-j)\inf g} \sum_{A \in L_s^{\delta_1}(\cM_{0}^{n})} |e^{S_{n}g}|_{C^0(A)} 
\leqslant \tilde C_{n_1} \sum_{A \in L_s^{\delta_1}(\cM_{0}^{n})} |e^{S_{n}g}|_{C^0(A)}.
\end{align*}
Using this last estimate, we obtain

\begin{align*}
\sum_{A \in I_s(\cM_0^n)} \!\!\! |e^{S_n g}|_{C^0(A)} &\leqslant (K_2n + \bar C_{n_1} + C'_{n_1} C'_g K_2) e^{n \sup g} 
 + (\tilde C C + C'_{n_1} C'_g \tilde C_{n_1}) \delta_1^{-1} \!\!\!\!\!\!\! \sum_{A \in L_s^{\delta_1}(\cM_0^n) } \!\!\!\!\!\!\! |e^{S_n g}|_{C^0(A)},\\
&\leqslant C_{1} e^{n\sup g} + C_{2} \delta_1^{-1} \sum_{A \in L_s^{\delta_1}(\cM_0^n) } |e^{S_n g}|_{C^0(A)},
\end{align*}
where $\tilde C$ is a constant coming from the summability assumption \eqref{SSP.3}, and $\bar C_{n_1}$ depends only on $n_1$ and $g$.

Finally, since $I_s(\cM_0^n) \sqcup L_s^{\delta_1}(\cM_0^n) = \cM_0^n$, we get that
\begin{align*}
\sum\limits_{A \in L_s^{\delta_1}(\cM_0^n)} |e^{S_n g}|_{C^0(A)} 
&\geqslant \frac{\sum\limits_{A \in \cM_0^n} |e^{S_n g}|_{C^0(A)} - C_1 e^{n \sup g} }{1 + C_2 \delta_1^{-1}}.
\end{align*}
Since $ \lim_{n \to + \infty}\frac{1}{n} \log \sum_{A \in \cM_0^n} |e^{S_n g}|_{C^0(A)} = P_*(T,g)$ and by the assumption $P_*(T,g) > \sup g$, there is an integer $n_2$ such that for all $n \geqslant n_2$, 
\[ \sum\limits_{A \in \cM_0^n} |e^{S_n g}|_{C^0(A)} - C_1 e^{n \sup g} \geqslant \frac{1}{2} \sum\limits_{A \in \cM_0^n} |e^{S_n g}|_{C^0(A)}. \]
Thus, there exists $C_{n_1}>0$ such that for all $n \geqslant n_2$ \eqref{eq:long_cells_geq_all} holds.

\medskip

We now prove the second part of Lemma~\ref{lemma:Lu_geq_M-n0}. Assume that $g$ is a $(\cM_0^1,\alpha)$-H\"older potential with $P_*(T,g) - \sup g > s_0 \log 2$ and $\log \Lambda > \sup g - \inf g$. From the convexity of the topological pressure (Theorem~\ref{thm:pressure}), we get that $t \mapsto P_*(T,tg)$ is a convex function. Thus, the map $t \mapsto P_*(T,t(g - \sup g)) = P_*(T,tg)-t\sup g$ is continuous on $[0,1]$. Since for all $s < t$ we have
\begin{align*}
\sum_{A \in \cM_0^n} \!\! |e^{S_n t(g-\sup g)} |_{C^0(A)} \leqslant e^{n(t-s)\sup(g-\sup g)} \!\! \sum_{A \in \cM_0^n} \!\! |e^{S_n s(g-\sup g)} |_{C^0(A)} =  \!\!\sum_{A \in \cM_0^n} \!\! |e^{S_n s(g-\sup g)} |_{C^0(A)},
\end{align*}
the map is nonincreasing. Thus 
\[
P_*(T,g) - \sup g = P_*(T,g - \sup g) \leqslant P_*(T,0) = h_*,
\]
where $h_*$ is the topological entropy of $T$ from \cite{BD2020MME}. Therefore we have $h_* > s_0 \log 2$ and estimates from \cite{BD2020MME} can be used. For all $W \in \hW^s$ with $\delta_1 \geqslant |W| \geqslant \delta_1/3$ and all $n \geqslant n_1$,
\begin{align*}
\sum_{V \in L_n^{\delta_1}(W)} \!\!\!\!\!\! |e^{S_n g}|_{C^0(V)} \! \geqslant \! e^{n \inf g} \# L_n^{\delta_1}(W)  \! \geqslant \! \frac{2}{3} e^{n \inf g} \# \cG_n^{\delta_1}(W) \! \geqslant \! \frac{2}{3} c_0 e^{n \inf g} \# \cM_0^{n} 
\! \geqslant \! \frac{2}{3} c_0 e^{n (\inf g + P_*(T,0))}
\end{align*}
where we used \cite[Lemma~5.2]{BD2020MME} for the second inequality, and Propositions~4.6 and 5.5 from \cite{BD2020MME} in the third inequality\footnote{We can choose the scale $\delta_1$ from \cite{BD2020MME} to agree with the one here. The constant $c_0$ comes from \cite[Proposition~5.5]{BD2020MME} and depends on $\delta_1$.}.

Thus we get that $\ell_n^s(g,\delta_1) \geqslant \tfrac{2}{3} c_0 e^{n (\inf g + P_*(T,0))}$. Since\footnote{$\log \Lambda$ is a lower bound on the unstable Lyapunov exponent of $T$. Integrating against $\musrb$ gives the desired inequality.} $P_*(T,0) = h_* \geqslant \log \Lambda$, we then get the summability of the sequence $(e^{n\sup g} \ell_n^s(g,\delta_1)^{-1} )_{n\geqslant n_1}$. The summability of $e^{n\sup g} \ell_n^u(g,\delta_1)^{-1}$ is obtained similarly by considering lower bounds on $\# L_u^{\delta_1}(W)$, also given in \cite{BD2020MME}. 
\end{proof}

We now introduce the precise definition of SSP.2: 
\begin{definition}[SSP.2]\label{def:SSP2}
A potential $g$ is said to have $\varepsilon$-SSP.2 if it has $\varepsilon$-SSP.1, if there exists $\bar n_1 : (0,+\infty) \to \mathbbm{N}$ such that 
\begin{align}\label{SSP.2}
\frac{\sum\limits_{W_i \in L_n^\delta(W)} |e^{S_n g}|_{C^0(W_i)}}{\sum\limits_{W_i \in \cG_n^\delta(W)} |e^{S_n g}|_{C^0(W_i)}} \geqslant \frac{1 - 3 \varepsilon}{1 - \varepsilon}, \quad \forall W \in \hW^s , \, \forall n \geqslant \bar n_1(|W|) ,
\end{align}
and if the time reversal\footnote{As for \eqref{SSP.1}, we call time reversal of \eqref{SSP.2} the same estimate but with $\hW^s$ replaced by $\hW^u$, $T$ by $T^{-1}$ and $g$ by $g \circ T^{-1}$ (that is $S_n g$ replaced by $S^{-1}_n g$).} of \eqref{SSP.2} holds, where $\delta$ is the corresponding constant from $\ve$-SSP.1. A potential is said to have SSP.2 if it has $\ve$-SSP.2 for some $\ve \leqslant 1/4$.
\end{definition}

\begin{corollary}\label{corol:long_over_all}
If $g$ is a $(\cM_0^1,\alpha)$-H\"older potential such that $P(T,g) - \sup g > s_0 \log 2$ and $\log \Lambda > \sup g - \inf g$, then there exists $C_2 >0$ such that $g$ has $\varepsilon$-SSP.2 for all $\varepsilon >0$ and $\bar n_1(|W|) = C_2 n_1 \frac{|\log (|W| / \delta) |}{|\log \varepsilon |}$, where $\delta$ and $n_1$ are the corresponding constants from Lemma~\ref{lemma:short_curves_rare}.
\end{corollary}

\begin{proof}
From the Lemmas~\ref{lemma:short_curves_rare} and~\ref{lemma:Lu_geq_M-n0}, such a potential has SSP.1. We thus only prove \eqref{SSP.2}.

The proof is similar the one for Lemma~\ref{lemma:short_curves_rare}, except that for curves shorter than $\delta/3$ one must wait $n \lesssim |\log(|W|/\delta)|$ for at least one component of $\cG_n^\delta(W)$ to belong to $L_n^\delta(W)$.

More precisely, fix $\varepsilon >0$ and the corresponding $\delta$ and $n_1$ from Lemma~\ref{lemma:short_curves_rare}. Let $W \in \hW^s$ with $|W|<\delta/3$ and take $n>n_1$. Decomposing $\cG_n^\delta(W)$ and $S_n^\delta(W)$ as in the proof of Lemma~\ref{lemma:short_curves_rare}, we estimate the second part as before. For the first part, we have to split the sum between $\cI_n^\delta(W)$ and the rest, which is estimated as before.

For the first part, concerning $\cI_n^\delta(W)$, for $\delta$ sufficiently small, notice that since the flow is continuous, either $\# \cG_l^\delta(W) \leqslant Kl+1$ by (\ref{eq:complex}) or at least one element of $\cG_l^\delta(W)$ has length at least $\delta/3$. Let $n_2$ denote the first iterate $l$ at which $\cG_l^\delta(W)$ contains at least one element of length more than $\delta/3$. By the complexity estimate (\ref{eq:complex}) and the fact that $|T^{-n_2}W| \geqslant C_1 \Lambda^{n_2}|W|$ by hyperbolicity of $T$, there exists $\overline{C}_2 >0$, independent of $W \in \hW^s$, such that $n_2 \leqslant \overline{C}_2 |\log(|W|/\delta)|$.

Now, for $n\geqslant n_2$, 
\begin{align*}
\sum\limits_{W_i \in \mathcal{I}_n^\delta(W)} |e^{S_n g}|_{C^0(W_i)}
&\leqslant \sum\limits_{W' \in \cG_{n_2}^\delta(W)} |e^{S_{n_2} g}|_{C^0(W')} \sum\limits_{W_i \in \cI_{n-n_2}^\delta(W')} |e^{S_{n-n_2} g}|_{C^0(W_i)}\\
&\leqslant K(Kn_2 +1) e^{n_2 \sup g} \times  2(Kn_1+1)^{\frac{n-n_2}{n_1}} e^{(n-n_2)\sup g}
\end{align*}
and by hyperbolicity and Lemma~\ref{lemma:sup_less_poly_inf},
\begin{align*}
\sum\limits_{W_i \in \cG^\delta_n(W)} \!\!\!\! |e^{S_n g}|_{C^0(W_i)} 
&\geqslant C^{-1} |e^{S_{n_2} g}|_{C^0(W')} \!\!\!\!\!\! \sum\limits_{W_i \in \cG^\delta_{n-n_2}(W')} \!\!\!\!\!\! e^{(n-n_2) \inf g} 
\geqslant \frac{1}{3} C_1 C^{-1} e^{n_2 \inf g} e^{(n-n_2)(\inf g + \log \Lambda)}
\end{align*}
where $W' \in \cG_{n_2}^\delta(W)$ is such that $|W'|> \delta/3$. Therefore,
\begin{align*}
\frac{ \sum\limits_{W_i \in \mathcal{I}_n^\delta(W)} |e^{S_n g}|_{C^0(W_i)} }{ \sum\limits_{W_i \in \cG^\delta_n(W)} |e^{S_n g}|_{C^0(W_i)} }
&\leqslant 6 C_1^{-1} C e^{n_2(\sup g - \inf g)} K(Kn_2+1) (Kn_1+1)^{\frac{n-n_2}{n_1}} e^{(n-n_2)(\sup g - \inf g - \log \Lambda)} \\
&\leqslant 2c_0^{-1}C^2 e^{n_2(\sup g - \inf g)} K(Kn_2+1) \varepsilon^{n/n_1}.
\end{align*}
Since $n_2 \leqslant \overline{C}_2 |\log(|W|/\delta)|$, we can bound this expression by $\varepsilon$ by choosing some $C_2>0$ and $n$ large enough so that $n/n_1 \geqslant C_2 \frac{\log(|W|/\delta)}{\log \varepsilon}$. For such $n$, the left hand side of \eqref{eq:Sn_over_Gn} is bounded by $\varepsilon + \frac{\varepsilon}{1-\varepsilon} \leqslant \frac{2\varepsilon}{1-\varepsilon}$, which completes the proof of the corollary.

As usual, the time reversal of \eqref{SSP.2} is obtained by performing the same proof, but with the time reversal counterpart of $\cG_n^{\delta}(W)$, for unstable curves $W$. 
\end{proof}

\subsection{Exact Exponential Growth of Thermodynamic Sums -- Cantor Rectangles}\label{subsect:supermultiplicativity}
It follows from the submultiplicativity in the characterisation of $P_*(T,g)$ (Theorem~\ref{thm:pressure}) that \[ e^{nP_*(T,g)} \leqslant e^{- \inf g} \sum_{A \in \cM_{0}^n} \sup_{x \in A} e^{(S_n g)(x)} \, , \quad \forall n \geqslant 1. \]
In this subsection, we shall prove a supermultiplicativity statement (Lemma~\ref{lemma:supermultiplicativity}) from which we deduce the upper bound for $\sum_{A \in \cM_{0}^n} \sup_{x \in A} e^{(S_n g)(x)}$ in Proposition~\ref{prop:almost_exponential_growth} giving the upper bound in Proposition~\ref{prop:upper_bounds_norms}, and ultimately the upper bound on the spectral radius of $\cL_g$ on $\cB$.

The following key estimate gives the reverse inequality of \eqref{eq:G_n_leqs_M_0^n}, for stable curves that are not to short, thus linking thermodynamic sums over curves and partitions. The proof will crucially use the fact that the SRB measure is mixing in order to bootstrap from SSP.1.

\begin{proposition}\label{prop:GnW_geq_M0n}
Let $g$ be a $(\cM_0^1,\alpha)$-H\"older potential with $P_*(T,g) - \sup g > s_0 \log 2$ and which has SSP.1. Let $\delta_1$ be the value of $\delta$ from the condition SSP.1. Then there exists $c_0 > 0$ such that for all $W \in \hW^s$ with $|W| \geqslant \delta_1 /3$ and $n \geqslant 1$, we have 
\[ \sum\limits_{W_i \in \cG_n^{\delta_0}(W)} |e^{S_n g}|_{C^0(A)} \geqslant c_0 \sum\limits_{A \in \cM_{-n}^0} |e^{S^{-1}_n g}|_{C^0(A)}. \]
The constant $c_0$ depends on $\delta_1$.
\end{proposition}

The proof relies crucially on the notion of \emph{Cantor rectangles}. We introduce this notion as in \cite[Definition~5.7]{BD2020MME}. Let $W^s(x)$ and $W^u(x)$ denote the maximal smooth components of the local stable and unstable manifolds of $x \in M$.

\begin{definition}\label{def:Cantor_rectangle+proper_crossing}
A solid rectangle $D$ in $M$ is a closed connected set whose boundary comprises precisely four nontrivial curves: two (segments of) stable manifolds and two (segments of) unstable manifolds. Given a solid rectangle $D$, the (locally maximal) Cantor rectangle $R$ in $D$ is formed by taking the points in $D$ whose local stable and unstable manifolds completely cross $D$. Cantor rectangles have a natural product structure: for any $x,\, y \in R$, then $W^s(x) \cap W^u(y) \in R$. In \cite[Section~7.11]{chernov2006chaotic}, Cantor rectangles are proved to be closed, and thus contain their outer boundaries, which are contained in the boundary of $D$. With a slight abuse, we will call these pairs of stable and unstable manifolds the stable and unstable boundaries of $R$. In this case, we denote $D$ by $D(R)$ to emphasize that it is the smallest solid rectangle containing $R$.
\end{definition}

\begin{proof}[Proof of Proposition~\ref{prop:GnW_geq_M0n}]
Using \cite[Lemma 7.87]{chernov2006chaotic}, we may find finitely many Cantor rectangles $R_1,...,R_k$ satisfying
\begin{align}\label{eq:fat_Cantor_rectangle}
\inf\limits_{x \in R_i} \frac{m_{W^u}(W^u(x) \cap R_i)}{m_{W^u}(W^u(x) \cap D(R_i))} \geqslant 0.9 \, , \quad \forall 1 \leqslant i \leqslant k,
\end{align}
whose stable and unstable boundaries have lengths at most $\tfrac{1}{10}\delta_1$, ``\emph{covering $M$}'' in the sense that any stable curve of length at least $\delta_1/3$ \emph{properly crosses} at least one of them. A stable curve $W \in \hW^s$ is said to properly cross $R$ if $W$ crosses both unstable sides of $R$, $W$ does not cross any stable manifolds $W^s(x) \cap D(R)$ for $x \in R$, and the point $W \cap W^u(x)$ subdivides the curve $W^u(x) \cap D(R)$ in a ratio between $0.1$ and $0.9$ (i.e. $W$ does not come to close to either stable boundary of $R$). The cardinality $k$ is fixed, depending only on $\delta_1$.

Recall that $L_u^{\delta_1}(\cM_{-n}^0)$ denotes the elements of $\cM_{-n}^0$ whose unstable diameter is longer than $\delta_1/3$. We claim that for all $n \in \mathbbm{N}$, at least one $R_i$ is fully crossed in the unstable direction by each element in a subset $\tilde{L}$ of $\cM_{-n}^0$ such that 
\begin{align}\label{eq:large_proportion_fully_cross}
\sum\limits_{A \in \tilde{L}} |e^{S_n^{-1} g}|_{C^0(A)} \geqslant \frac{1}{k} \sum\limits_{A \in L_u^{\delta_1}(\cM_{-n}^0)} |e^{S_n^{-1} g}|_{C^0(A)}.
\end{align}
Notice that if $A \in \cM_{-n}^0$, then $\partial A$ is comprised of unstable curves belonging to $\cup_{i=1}^n T^i \cS_0$, and possibly $\cS_0$. By definition of unstable manifolds, $T^i \cS_0$ cannot intersect the unstable boundaries of the $R_i$; thus if $A\cap R_i \neq \emptyset$, then either $\partial A$ terminates inside $R_i$ or $A$ fully crosses $R_i$. Thus elements of $L_u^{\delta_1}(\cM_{-n}^0)$ fully cross at least one $R_i$ and so at least one $R_i$ must be fully crossed by \emph{a large fraction} $\tilde{L}$ of $L_u^{\delta_1}(\cM_{-n}^0)$ in the sense of (\ref{eq:large_proportion_fully_cross}), proving the claim.

For each $n \in \mathbbm{N}$, denote by $i_n$ the index of a rectangle $R_{i_n}$ which is fully crossed by a large enough subset $\tilde{L}_n$ of $L_u(\cM_{-n}^0)$, in the sense of (\ref{eq:large_proportion_fully_cross}).

Fix $\delta_* \in (0,\delta_1/10)$ and for $i=1,...k$, choose a ``high density" subset $R_i^* \subset R_i$ satisfying the following conditions: $R_i^*$ has a non-zero Lebesgue measure, and for any unstable manifold $W^u$ such that $W^u \cap R_i^* \neq \emptyset$ and $|W^u| < \delta_*$, we have $\tfrac{m_{W^u}(W^u \cap R_i^*)}{|W^u|} \geqslant 0.9$. (Such a $\delta_*$ and $R_i^*$ exist due to the fact that $m_{W^u}$-almost every $y \in R_i$ is a Lebesgue density point of the set $W^u(y) \cap R_i$ and the unstable foliation is absolutely continuous with respect to $\musrb$ or, equivalently, Lebesgue.)

Due to the mixing property of $\musrb$ and the finiteness of the number of rectangles $R_i$, there exist $\varepsilon >0$ and $n_3 \in \mathbbm{N}$ such that for all $1 \leqslant i,j \leqslant k$ and all $n \geqslant n_3$, $\musrb(R_i^* \cap T^{-n}R_j) \geqslant \varepsilon$. If necessary, we increase $n_3$ so that the unstable diameter of the set $T^{-n}R_i$ is less than $\delta_*$ for each $i$, and $n \geqslant n_3$.

Now let $W \in \hW^s$ with $|W| \geqslant \delta_1 /3$ be arbitrary. Let $R_j$ be a Cantor rectangle that is properly crossed by $W$. Let $n \in \mathbbm{N}$ and let $i_n$ be as above. By mixing, $\musrb(R_{i_n}^* \cap T^{-n_3}R_j) \geqslant \varepsilon$. By \cite[Lemma 7.90]{chernov2006chaotic}, there is a component of $T^{-n_3}W$ that fully crosses $R_{i_n}^*$ in the stable direction. Call this component $V \in \cG_{n_3}^{\delta_0}(W)$. Thus
\begin{align*}
\sum\limits_{W_i \in \cG_n^{\delta_0}(V)} |e^{S_n g}|_{C^0(W_i)} &= \sum\limits_{W_i \in \cG_n^{\delta_0}(V)} |e^{S_n^{-1} g}|_{C^0(T^{n}W_i)} \geqslant \sum\limits_{A \in \tilde{L}_n} \inf\limits_{A} |e^{S_n^{-1} g}| \geqslant \frac{1}{C_g} \sum\limits_{A \in \tilde{L}_n} \sup\limits_{A} |e^{S_n^{-1} g}| \\
&\geqslant \frac{1}{k C_g} \sum\limits_{A \in L_u^{\delta_1}(\cM_{-n}^0)} |e^{S_n^{-1} g}|_{C^0(A)}.
\end{align*}
We now have to relate the lhs to the analogous quantity where $V$ is replace by $W$.
\begin{align*}
&\sum\limits_{W_i \in \cG_n^{\delta_0}(W)}|e^{S_n g}|_{C^0(W_i)} = \sum_{V_j \in \cG_{n+n_3}^{\delta_0}(W)} \sum_{\substack{W_i \in \cG_n^{\delta_0}(W) \\ T^{n_3}V_j \subset W_i}} \frac{|e^{S_n g}|_{C^0(W_i)}}{\#\{ V_j \in \cG_{n+n_3}^{\delta_0}(W) \mid T^{n_3}V_j \subset W_i \}} \\
&\geqslant \sum_{V_j \in \cG_{n+n_3}^{\delta_0}(W)} |e^{S_n g \circ T^{n_3}}|_{C^0(V_j)} \sum_{\substack{W_i \in \cG_n^{\delta_0}(W) \\ T^{n_3}V_j \subset W_i}} \frac{1}{\#\{ V_j \in \cG_{n+n_3}^{\delta_0}(W) \mid T^{n_3}V_j \subset W_i \}} \\
&\geqslant \frac{C \delta_0}{\# \cM_0^{n_3}} e^{-n_3 \sup g} \sum_{V_j \in \cG_{n+n_3}^{\delta_0}(W)} |e^{S_{n+n_3} g}|_{C^0(V_j)} \geqslant \frac{C \delta_0}{\# \cM_0^{n_3}} e^{-n_3 \sup g} \sum_{V_j \in \cG_{n}^{\delta_0}(V)} |e^{S_{n+n_3} g}|_{C^0(V_j)} \\
&\geqslant \frac{C \delta_0}{\# \cM_0^{n_3}} e^{-n_3 (\sup g - \inf g)} \!\!\!\! \sum_{V_j \in \cG_{n}^{\delta_0}(V)} \!\!\!\! |e^{S_{n} g}|_{C^0(V_j)} 
\geqslant \frac{1}{kC_g} \frac{C \delta_0}{\# \cM_0^{n_3}} e^{-n_3 (\sup g - \inf g)} \!\!\!\! \sum\limits_{A \in L_u^{\delta_1}(\cM_{-n}^0)} \!\!\!\! |e^{S_n^{-1} g}|_{C^0(A)} \\
&\geqslant C_{n_1} \delta_1 \frac{1}{kC_g} \frac{C \delta_0}{\# \cM_0^{n_3}} e^{-n_3 (\sup g - \inf g)} \sum\limits_{A \in \cM_{-n}^0} |e^{S_n^{-1} g}|_{C^0(A)},
\end{align*}
for all $n \geqslant \max\{n_2, n_3\}$, where we used Lemma~\ref{lemma:Lu_geq_M-n0} for the last inequality. Thus the proposition holds for all $n \geqslant \max\{n_2, n_3\}$. It extends to all $n\in \mathbbm{N}$ since there are finitely many values of $n$ to correct for.
\end{proof}

\begin{lemma}[Supermultiplicativity]\label{lemma:supermultiplicativity}
There exists a constant $c_1$ such that for all $n \in \mathbbm{N}$, and all $0 < j < n$, we have
\[ \sum\limits_{A \in \cM_0^n} |e^{S_n g}|_{C^0(A)} \geqslant c_1 \sum\limits_{A \in \cM_0^{n-j}} |e^{S_{n-j} g}|_{C^0(A)} \sum\limits_{A \in \cM_{0}^j} |e^{S_j g}|_{C^0(A)}.\]
\end{lemma}

\begin{proof}
Fix $n,j \in \mathbbm{N}$ with $j<n$. First, notice that
\begin{align*}
&\sum\limits_{A \in \cM_0^n} |e^{S_n g}|_{C^0(A)} \geqslant \sum\limits_{A \in \cM_0^n} \sup\limits_{A} e^{(S_{n-j} g + S_j^{-1} g )\circ T^j} \geqslant  \sum\limits_{A \in \cM_{-j}^{n-j}} \sup\limits_{A} e^{S_{n-j}g} \inf\limits_{A} e^{S_j^{-1} g}\\
&\qquad\qquad \geqslant  \sum\limits_{A \in \cM_{0}^{n-j}} \sup\limits_{A} e^{S_{n-j}g} \sum\limits_{\substack{B \in \cM_{-j}^0 \\ B \cap A \neq \emptyset}} \inf\limits_{B} e^{S_j^{-1} g} 
\geqslant C_g \sum\limits_{A \in \cM_{0}^{n-j}} |e^{S_{n-j}g}|_{C^0(A)} \sum\limits_{\substack{B \in \cM_{-j}^0 \\ B \cap A \neq \emptyset}} |e^{S_j^{-1} g}|_{C^0(B)} \\
&\qquad\qquad \geqslant C_g \sum\limits_{A \in \cM_{0}^{n-j}} |e^{S_{n-j}g}|_{C^0(A)} \sum\limits_{\substack{B \in \cM_{-j}^0 \\ B \cap A \neq \emptyset}} |e^{S_j^{-1} g}|_{C^0(B)},
\end{align*}
where we used Lemma~\ref{lemma:sup_less_poly_inf} for the forth inequality.

Recall that $L_u^{\delta_1}(\cM_{-j}^{0})$ denotes the elements of $\cM_{-j}^{0}$ whose unstable diameter is longer than $\delta_1/3$. Similarly, $L_s^{\delta_1}(\cM_0^{n-j})$ denotes those elements of $\cM_0^{n-j}$ whose stable diameter is larger than $\delta_1/3$. By Lemma~\ref{lemma:Lu_geq_M-n0}
\[\sum\limits_{A \in L_s^{\delta_1}(\cM_0^{n-j})} |e^{S_{n-j} g}|_{C^0(A)} \geqslant C_{n_1} \delta_1 \sum\limits_{A \in \cM_0^{n-j}} |e^{S_{n-j} g}|_{C^0(A)} , \quad \text{for $n-j \geqslant n_2$.} \]

Let $A \in L_s^{\delta_1}(\cM_0^{n-j})$ and let $V_A \in \hW^s$ be a stable curve in $A$ with length at least $\delta_1/3$. By Proposition~\ref{prop:GnW_geq_M0n}, \[ \sum\limits_{W_i \in \cG_{j}^{\delta_0}(V_A)} |e^{S_{j} g}|_{C^0(W_i)} \geqslant c_0 \sum\limits_{B \in \cM_{-j}^0} |e^{S_{j}^{-1} g}|_{C^0(B)}.\]
Each component of $\cG_j^{\delta_0}(V_A)$ corresponds to one component of $V_A \smallsetminus \cS_{-j}$ (up to subdivision of long pieces in $\cG_j^{\delta_0}(V_A)$). Thus
\begin{align*}
&\sum\limits_{A \in \cM_{0}^{n-j}} |e^{S_{n-j}g}|_{C^0(A)} \sum\limits_{\substack{B \in \cM_{-j}^0 \\ B \cap A \neq \emptyset}} |e^{S_j^{-1} g}|_{C^0(B)}
\geqslant \sum\limits_{A \in L_s^{\delta_1}(\cM_0^{n-j})} |e^{S_{n-j}g}|_{C^0(A)} \sum\limits_{W_i \in \cG_j^{\delta_0}(V_A)} |e^{S_j^{-1} g}|_{C^0(T^j W_i)} \\
&\qquad \geqslant \!\!\! \sum\limits_{A \in L_s^{\delta_1}(\cM_0^{n-j})} \!\!\! |e^{S_{n-j}g}|_{C^0(A)} \!\!\! \sum\limits_{W_i \in \cG_j^{\delta_0}(V_A)} \!\!\! |e^{S_j g}|_{C^0(W_i)} 
\geqslant C \!\! \sum\limits_{A \in \cM_0^{n-j}} \!\! |e^{S_{n-j}g}|_{C^0(A)} \sum\limits_{B \in \cM_{-j}^0} |e^{S_j^{-1} g}|_{C^0(B)},
\end{align*}
proving the lemma with $c_1= c_0 C_{n_1} C^2 \delta_1  $ when $n-j \geqslant n_2$. For $n-j \leqslant n_2$, since 
\[ \sum\limits_{A \in \cM_0^{n-j}} |e^{S_{n-j-1} g}|_{C^0(A)} \leqslant \left( \sum\limits_{A \in \cM_0^{1}} |e^{ g}|_{C^0(A)} \right)^{n-j}\]
the lemma holds by decreasing $c_1$ since there are only finitely many values to correct for.
\end{proof}

\begin{proposition}[Exact Exponential Growth]\label{prop:almost_exponential_growth}
Let $g$ be a $(\cM_0^1,\alpha)$-H\"older continuous potential such that $P_*(T,g) - \sup g >0 $ and which has SSP.1. Let $c_1$ be the constant given by Lemma~\ref{lemma:supermultiplicativity}. Then for all $n \in \mathbbm{N}$, we have
\[ \sum\limits_{A \in \cM_0^n} |e^{S_n g}|_{C^0(A)} \leqslant \frac{2}{c_1} e^{n P_*(T,g)}.\]
\end{proposition}

\begin{proof}
Let $\psi(n) \coloneqq e^{-nP_*(T,g)} \sum_{A \in \cM_0^n} |e^{S_n g}|_{C^0(A)}$. 
Suppose there exists $n_1 \in \mathbbm{N}$ such that $\psi(n_1) \geqslant 2/c_1$, where $c_1$ is the constant from Lemma~\ref{lemma:supermultiplicativity}. Then \begin{align*}
\psi(2n_1) \geqslant c_1 \psi(n_1)^2 = \frac{1}{c_1}(c_1 \psi(n_1))^2.
\end{align*}
Iterating this bound, we obtain for any $k \geqslant 1$,
\begin{align*}
\psi(2^{k}n_1) \geqslant \frac{1}{c_1}(c_1 \psi(n_1))^{2^k}.
\end{align*}
This implies that $\lim_{k\to +\infty} \tfrac{1}{2^kn_1}\log \psi(2^kn_1) \geqslant \tfrac{1}{n_1}\log 2 >0$, which contradicts the definition of $\psi(n)$ (since by construction $\lim_{n \to +\infty} \tfrac{1}{n}\log \psi(n) = 0$). We conclude that $\psi(n) \leqslant 2/c_1$ for all $n \geqslant 1$.
\end{proof}

\begin{remark}\label{remark:neighbourhood_zero_potential}
Notice that for $g=0$, the condition $P_*(T,g) - \sup g > s_0 \log 2$ becomes $h_* > s_0 \log 2$, where $h_*$ is the topological entropy of $T$ defined in \cite{BD2020MME}. This is precisely the condition of sparse recurrence to singularities from \cite{BD2020MME}, and as discussed there, we don't know any example of billiard table not satisfying this condition. Notice that by continuity, if $h_* > s_0 \log 2$ holds, then $P_*(T,g) - \sup g > s_0 \log 2$ holds for all $g$ in a neighbourhood of the zero potential. Up to taking a smaller neighbourhood,$\log \Lambda > \sup g - \inf g$ also holds. Therefore, by Lemmas~\ref{lemma:short_curves_rare}, \ref{lemma:Lu_geq_M-n0} and Corollary~\ref{corol:long_over_all}, there exists a neighbourhood of $g=0$ (in the $(\cM_0^1,\alpha)$-H\"older topology) in which every potential has SSP.1 and SSP.2.
In particular, for any $t \in \mathbbm{R}$ with $|t|$ close enough to zero, the potential $-t \tau$ has SSP.1 and SSP.2.
\end{remark}

\subsection{Estimates on norms of the potential}

In Section~\ref{sect:measure_mu_g}, we will need similar estimates as in the present section but with the $C^0$ norm replaced by the $C^\beta$ norm, $0<\beta< 1/3$. The following lemma shows that previous estimates are still valid up to a multiplicative constant.

\begin{lemma}\label{lemma:norm_C0_exp_birkhoff}
For every bounded $(\cM_0^1,\alpha)$-H\"older continuous potential $g$, there exists $C>0$ such that for all $W \in \cW^s$, all $n \geqslant 0$ and all $W_i \in \cG_n^{\delta}(W)$, $|e^{S_ng}|_{C^{\alpha}(W_i)} \leqslant C |e^{S_ng}|_{C^{0}(W_i)}$, where $\delta \in (0,\delta_0]$.
\end{lemma}

\begin{proof}
Let $g$ be such a potential. Let $c$ be such that $g \geqslant c$. Let $W_i \in \cG_n^{\delta}(W)$. Then
\begin{align*}
H^{\alpha}_{W_i}(e^{S_n g}) &\leqslant \sum\limits_{k=0}^{n-1} |e^{ - g \circ T^k + S_n g}|_{C^0(W_i)} H^{\alpha}_{W_i}(g \circ T^k) 
\leqslant |e^{S_n g}|_{C^0(W_i)} \sum\limits_{k=0}^{n-1} e^{-c} C \Lambda^{-\alpha k} |g|_{C^{\alpha}(M)} \\
&\leqslant |e^{S_n g}|_{C^0(W_i)} C \frac{1}{1 - \Lambda^{\alpha}} e^{-c} |g|_{C^{\alpha}(M)},
\end{align*}
where for the second inequality we adapted the argument from \cite[eq (6.2)]{BD2020MME}, so that \begin{align*}
\frac{g(T^k x) - g(T^k y)}{d_W(T^k x,T^k y)^{\alpha}} \frac{d_W(T^k x,T^k y)^{\alpha}}{d_W(x,y)^{\alpha}} \leqslant C H^{\alpha}_{T^k W_i}(g) |J^s T^k|^{\alpha}_{C^0(W_i)} \leqslant C \Lambda^{-\alpha k} |g|_{C^{\alpha}(M)}.
\end{align*}
\end{proof}

\section{The Banach Spaces $\cB$ and $\cB_w$ and the Transfer Operators $\cL_g$}\label{sect:banach_spaces_and_transfer_operators}

In Section~\ref{sect:measure_mu_g}, we construct the equilibrium state $\mu_g$ for $T$ under the potential $g$ out of left and right eigenvectors, $\tilde \nu$ and $\nu$, of a transfer operator $\cL_g$ associated with the billiard map and the potential $g$, acting on suitable Banach spaces $\cB$ and $\cB_w$ of anisotropic distributions. In this section, we define these Banach spaces $\cB$ and $\cB_w$ as well as the transfer operator $\cL_g$.

\subsection{Motivation and heuristics}

The spaces $\cB$ and $\cB_w$ are the same as in \cite{BD2020MME}, but we recall their construction not only for completeness, but also to introduce notations. The norms we introduce below are defined by integrating along stable manifolds in $\cW^s$. We define precisely the notion of distance $d_{\cW^s}(\cdot,\cdot)$ between such curves as well as a distance $d(\cdot,\cdot)$ defined among functions supported on these curves.

In the setup of uniform hyperbolic dynamic, the relevant transfer operator to study equilibrium states associated to a potential $g$ -- see for example \cite{baladi18book} -- can be defined on measurable function $f$ by
\[ \cL_g f = \left(e^g \frac{f}{J^s T} \right) \circ T^{-1} \]
where $J^s T$ is the stable Jacobian of $T$. Ignoring first the low regularity of $J^s T$, we see from the hyperbolicity of $T$ that the composition with $T^{-1}$ should increase the regularity of $f$ in the unstable direction, while decreasing the regularity in the stable direction. By integrating along stable manifold against the arclength measure, we hope to recover some regularity along the stable manifold -- notice that by a change of variable, $J^s T$ does disappear. Morally, the weak norm $|\cdot |_w$ and the strong stable norm $||\cdot ||_s$ measure the regularity of the averaged action of $\cL_g$. On the other hand, the strong unstable norm $||\cdot ||_u$ captures the regularity when passing from a stable manifold to another one. Here, this regularity should be thought of as a $\log$-scaled H\"older regularity.

\subsection{Definition of the Banach spaces and embeddings into distribution}\label{subsect:banach_spaces}

Recall that $\cW^s$ denote the set of all nontrivial connected subsets $W$ of length at most $\delta_0$ of stable manifolds for $T$. Such curves have bounded curvature above by fixed constant \cite[Prop.~4.29]{chernov2006chaotic}. Thus $T^{-1} \cW^s = \cW^s$, up to subdivision of curves. Obviously, $\cW^s \subset \hW^s$. We define $\cW^u$ similarly from unstable manifolds of $T$.

Given a curve $W \in \cW^s$, we denote by $m_W$ the unnormalized Lebesgue (arclength) measure on $W$, so that $|W| = m_W(W)$. Since the stable cone $C^s$ \eqref{eq:stable_cones} is bounded away from the vertical, we may view each stable manifolds $W \in \cW^s$ as the graph of a function $\varphi_W(r)$ of the arclength coordinate $r$ ranging over some interval $I_W$, that is
\begin{align*}
W = \{ G_W(r) \coloneqq (r,\varphi_W(r)) \mid r \in I_W \}.
\end{align*}
Given two curves $W_1, \, W_2 \in \cW^s$, we may use this representation to define a ``distance"\footnote{Actually, $d_{\cW^s}$ is not a metric since it does not satisfies the triangle inequality. It is nonetheless sufficient for our purpose to produce a usable notion of a distance between stable manifolds.} between them. Define
\begin{align*}
d_{\cW^s}(W_1,W_2) = |I_{W_1} \,\triangle\, I_{W_2}| + |\varphi_{W_1} - \varphi_{W_2}|_{C^1(I_{W_1}\cap I_{W_2})} 
\end{align*}
if $I_{W_1}\cap I_{W_2} \neq \emptyset$, and $d_{\cW^s}(W_1,W_2) = +\infty$ otherwise.

Similarly, given two test functions $\psi_1$ on $W_1$, and $\psi_2$ on $W_2$, we define a distance between them by
\begin{align*}
d(\psi_1,\psi_2)=|\psi_1 \circ G_{W_1} - \psi_2 \circ G_{W_2}|_{C^0(I_{W_1}\cap I_{W_2})} \, ,
\end{align*}
whenever $d_{\cW^s}(W_1,W_2)$ is finite, and $d(\psi_1,\psi_2) = +\infty$ otherwise.

We can now introduce the norms used to define the spaces $\cB$ and $\cB_w$. These norms will depend on the constants $\epsilon_0>0$ and $\delta_0 \in (0,1)$, as well as on four positive real numbers $\alpha$, $\beta$, $\gamma$ and $\zeta$ so that
\begin{align*}
0<\beta<\alpha \leqslant \min\{ 1/3,\alpha_g \}, \quad 1<2^{s_0 \gamma} < e^{P_*(T,g) - \sup g}, \quad 0<\zeta <\gamma
\end{align*}
where $g$ is a given, bounded $(\cM_0^1,\alpha_g)$-H\"older potential such that $P_*(T,g) - \sup g >s_0 \log 2$.
\begin{remark}
The condition $\alpha \leqslant 1/3$ is needed for \cite[Lemma~4.4]{BD2020MME}, which is used to prove the embedding into distributions. The number $1/3$ comes from the regularity of the density function of the conditional measures in the disintegration of $\musrb$ against the stable foliation. The bound $\alpha \leqslant \alpha_g$ will be needed to verify that some functions involving $g$ are $C^\alpha$. The upper bound on $\gamma$ arises from the use of the growth lemma~\ref{lemma:Growth_lemma}. The dependence on $\delta_0$ comes from the definition of $\cW^s$.
\end{remark}

For $f \in C^1(M)$, define the weak norm of $f$ by
\[
| f |_w = \sup_{W \in \cW^s} \sup_{\substack{\psi \in C^\alpha(W) \\ |\psi|_{C^\alpha(W)} \leqslant 1}}
\int_W f \, \psi \, \mathrm{d}m_W \, .
\]
Similarly, define the strong stable norm of $f$ by\footnote{The logarithmic modulus of continuity 
 in $\|f\|_s$ is used to obtain a finite spectral radius.}
\[
\| f \|_s = \sup_{W \in \cW^s} \sup_{\substack{\psi \in C^\beta(W) \\ |\psi|_{C^\beta(W)} \leqslant |\log |W||^\gamma}}
\int_W f \, \psi \, \mathrm{d}m_W \, ,
\]
(note that $| f |_w \leqslant \max \{1, |\log \delta_0|^{-\gamma}\} \| f \|_s$).
Finally, for $\varsigma \in (0,  \gamma)$, define
the strong unstable norm\footnote{The logarithmic modulus of continuity appears in $\|f\|_u$ because
of the logarithmic modulus of continuity in $\|f\|_s$. Its presence in $\|f\|_u$
causes  the loss of the
spectral gap.} of $f$ by
\[
\| f \|_u = \sup_{\ve \leqslant \ve_0} \sup_{\substack{W_1, W_2 \in \cW^s \\ d_{\cW^s}(W_1, W_2) \leqslant \ve}}
\sup_{\substack{\psi_i \in C^\alpha(W_i) \\ |\psi_i|_{C^\alpha(W_i)} \leqslant 1 \\ d(\psi_1, \psi_2) = 0}} 
|\log \ve|^\varsigma \left| \int_{W_1} f \, \psi_1 \, \mathrm{d}m_{W_1} - \int_{W_2} f \, \psi_2 \, \mathrm{d}m_{W_2} \right| \,  .
\]

In order to use functional analysis results, we need to work with complete spaces. Since $C^1(M)$ is not complete for the norms\footnote{For example, the sequence $\left( (r,\varphi) \mapsto \frac{1}{n} \sin 2\pi n^2 \frac{r}{|\Gamma_i|} \right)_n$ is a Cauchy sequence of $C^1(M)$ functions with respect to $|\cdot |_w$, but diverges in the $C^1$-norm.} $|\cdot|_w$ and $\| \cdot \|_s +  \| \cdot \|_u$, we will use the corresponding completed spaces.
\begin{definition}[The Banach spaces]
The space $\cB_w$ is  the completion of $C^1(M)$ with respect to the weak norm
$| \cdot |_w$, while $\cB$ is the completion of $C^1(M)$ with respect to the strong norm,
$\| \cdot \|_{\cB} = \| \cdot \|_s +  \| \cdot \|_u$.
Notice that since $|\cdot |_w \leqslant \| \cdot \|_{\cB}$, there is a canonical map $\cB \to \cB_w$.
\end{definition}

Since the main purpose of the spaces $\cB$ and $\cB_w$ is to contain left and right eigenvectors of a transfer operator acting on those spaces, a crucial feature of $\cB$ and $\cB_w$ is that we can see them as subspaces of the distributional space $(C^1(M))^*$. From this property, we will be able to construct a positive distribution by pairing the left and right eigenvectors, and to extend it into the desired equilibrium measure. First, we need to introduce some other spaces, on which the transfer operator will be naturally defined (and then extended to $\cB$ and $\cB_w$). 

Define the usual homogeneity strips
\[ \mathbbm{H}_k \coloneqq \left\lbrace (r,\varphi) \in M_i \mid \frac{\pi}{2} - \frac{1}{k^2} \leqslant \varphi \leqslant \frac{\pi}{2} - \frac{1}{(k+1)^2} \right\rbrace, \quad k \geqslant k_0 \, ,\]
and analogously for $k \leqslant -k_0$. Define $\cW^s_\mathbbm{H} \subset \cW^s$ as the set of stable manifolds $W \in \cW^s$ such that $T^n W$ lies in a single homogeneity strip for all $n \geqslant 0$. We write $\psi \in C^\alpha(\cW^s_{\mathbbm{H}})$ if $\psi \in C^\alpha(W)$ for all $W \in \cW^s_{\mathbbm{H}}$ with uniformly bounded H\"older norm. The norm of $\psi$ in $C^\alpha(\cW^s_{\mathbbm{H}})$ is defined to be the sup over all the $C^\alpha(W)$ norms, with $W$ ranging in $\cW^s_{\mathbbm{H}}$. Similarly, define the space $C^\alpha_{\cos}(\cW^s_{\mathbbm{H}})$ containing the functions $\psi$ such that $\psi \cos \varphi \in C^\alpha(\cW^s_{\mathbbm{H}})$. The norm of $\psi$ in $C^\alpha_{\cos}(\cW^s_{\mathbbm{H}})$ is defined to be the norm of $\psi \cos \varphi$ in $C^\alpha(\cW^s_{\mathbbm{H}})$. Clearly, $C^\alpha(\cW^s_{\mathbbm{H}}) \subset C^\alpha_{\cos}(\cW^s_{\mathbbm{H}})$.

The canonical map $\cB_w \to (\cF)^*$ (for $\cF = C^1(M)$, or $\cF = C^\alpha(\cW^s_{\mathbbm{H}})$) is understood in the following sense: for $f \in \cB_w$, there exists $C_f < \infty$ such that letting $f_n \in C^1(M)$ be a sequence converging to $f$ in the $\cB_w$ norm, for every $f \in \cF$ the following limit exists
\[ f(\psi) \coloneqq \lim\limits_{n \to +\infty} \int f_n \psi \, \mathrm{d}\musrb \]
and satisfies $|f(\psi)| \leqslant C_f ||\psi||_{\cF}$. 

We summarize the properties of these Banach spaces obtained in \cite{BD2020MME} in the next proposition.

\begin{proposition}The spaces $\cB_w$ and $\cB$ are such that: \\
\noindent (i) The following canonical maps are all continuous
\[ C^1(M) \to \cB \to \cB_w \to (C^\alpha(\cW^s_{\mathbbm{H}}))^* \to (C^1(M) )^*, \]
and the first two maps are injective. In particular, we also have the two injective and continuous maps
\[ (\cB_w)^* \to \cB^* \to (C^1(M))^* .\]
\noindent (ii) The inclusion map $\cB \hookrightarrow \cB_w$ is compact.
\end{proposition}

\begin{proof}
The point (i) is the content of \cite[Proposition~4.2]{BD2020MME}. We detail the proof of the injectivity of the map $\cB \to \cB_w$. To do so, we prove that the formula defining $|\cdot |_w$ (respectively $||\cdot ||_s$ and $||\cdot ||_u$) can be extended when $f \in \cB_w$ (respectively $f \in \cB$), and that it coincides with the norm of $f$. 

First, notice that when $f \in C^1(M)$, then for given $W \in \cW^s$ and $\psi \in C^\alpha(W)$ we have $\int_W f \psi \,\mathrm{d}m_W \leqslant |f|_w |\psi|_{C^\alpha(W)}$. Thus $f \mapsto \int_W f \psi \,\mathrm{d}m_W$ can be extended uniquely to $\cB_w$. 

Now, let $f \in \cB_w$, $\varepsilon >0$ and $f_n$ be a Cauchy sequence of $C^1(M)$ functions converging to $f$ in $\cB_w$. Thus, there exists some $n_\varepsilon$ such that for all $n \geqslant n_\varepsilon$, $| f - f_n |_w \leqslant \varepsilon$. Let $W \in \cW^s$ and $\psi \in C^\alpha(W)$ with $|\psi|_{C^\alpha(W)} \leqslant 1$. By definition of $|f_n|_w$, for all $n$, there exist $W_n$ and $\psi_n \in C^\alpha(W_n)$ with $|\psi_n|_{C^\alpha(W_n)}\leqslant 1$ such that 
\begin{align*}
\left| \int_{W_n} f_n \psi_n \,\mathrm{d}m_{W_n} - |f_n|_w \right| \leqslant \varepsilon. 
\end{align*}
Thus, we have 
\begin{align*}
\left| \int_{W_n} f \psi_n \,\mathrm{d}m_{W_n} - \int_{W_n} f_n \psi_n \,\mathrm{d}m_{W_n} \right| \leqslant |f-f_n|_w |\psi_n|_{C^\alpha(W_n)} \leqslant \varepsilon, \quad \forall n \geqslant n_\varepsilon,
\end{align*}
and so $\left| |f_n|_w - \int_{W_n} f \psi_n \,\mathrm{d}m_{W_n} \right| \leqslant 2\varepsilon$. In particular, we get 
\begin{align*}
\sup_{W \in \cW^s} \sup_{\substack{\psi \in C^\alpha(W) \\ |\psi|_{C^\alpha(W)} \leqslant 1}} \int_W f \, \psi \, \mathrm{d}m_W \geqslant |f|_w.
\end{align*}
We now prove the reverse inequality. Using the same notations as above, there exist $V \in \cW^s$ and $\varphi \in C^\alpha(V)$ with $|\varphi|_{C^\alpha(V)} \leqslant 1$ such that 
\begin{align*}
\left| \int_{V} f \varphi \,\mathrm{d}m_{V} - \sup_{W \in \cW^s} \sup_{\substack{\psi \in C^\alpha(W) \\ |\psi|_{C^\alpha(W)} \leqslant 1}} \int_W f \, \psi \, \mathrm{d}m_W \right| \leqslant \varepsilon.
\end{align*} 
Now, since
\begin{align*}
\left| \int_{V} f_n \varphi \,\mathrm{d}m_{V} - \int_{V} f \varphi \,\mathrm{d}m_{V} \right| \leqslant |f-f_n|_w \leqslant \varepsilon, \quad \forall n \geqslant n_\varepsilon,
\end{align*}
we have that $| \sup_{W \in \cW^s} \sup_{\substack{\psi \in C^\alpha(W) \\ |\psi|_{C^\alpha(W)} \leqslant 1}} \int_W f \, \psi \, \mathrm{d}m_W - \int_{V} f_n \varphi \,\mathrm{d}m_{V} | \leqslant 2\varepsilon$ for all large enough $n$. In particular
\begin{align*}
\sup_{W \in \cW^s} \sup_{\substack{\psi \in C^\alpha(W) \\ |\psi|_{C^\alpha(W)} \leqslant 1}} \int_W f \, \psi \, \mathrm{d}m_W \leqslant |f_n|_w + 2\varepsilon.
\end{align*}
Taking the limit in $n$, we get the claimed inequality.

The corresponding results for $f \in \cB$ and norms $||\cdot||_s$ and $||\cdot||_u$ are obtained similarly, noticing that for all $f \in C^1(M)$,
\begin{align*}
\int_W f \psi \,\mathrm{d}m_W \leqslant ||f||_s |\psi|_{C^\beta(W)} |\log |W| |^{-\gamma} \leqslant ||f||_\cB |\psi|_{C^\beta(W)} |\log |W| |^{-\gamma} \, , \,\,\, \forall W \in \cW^s , \, \forall \psi \in C^\beta(W)
\end{align*}
Thus the integrals against $C^\beta(W)$ functions in the definition of $||\cdot ||_s$ makes sense even when $f \in \cB$. On the other hand, since $|\cdot |_w \leqslant ||\cdot ||_{\cB}$, the integrals in the definition of $||\cdot ||_u$ can be extended to $f \in \cB$ as in the above case where $f \in \cB_w$.

We can now show the injectivity of the canonical map $\cB \to \cB_w$. Let $f \in \cB$ with $||f||_\cB \neq 0$. If $||f||_s \neq 0$, then the fact that $|f|_w \neq 0$ follows from the definition of $C^\beta(W)$ as the closure of $C^1(W)$ in the $C^\beta$ norm, so that $C^\alpha(W)$ is dense in $C^\beta(W)$. Now, if $||f||_u \neq 0$, then by definition of $||\cdot ||_u$, we can find some $W \in \cW^s$ and $\psi \in C^\alpha(W)$ so that $\int_w f \psi \, \mathrm{d}m_W >0$. Thus $|f|_w \neq 0$.

The point (ii) is precisely the content of \cite[Proposition~6.1]{BD2020MME}.
\end{proof}

\subsection{The transfer operators}

We may define the transfer operator $\cL_g : (C^\alpha_{\cos}(\cW^s_{\bH}))^* \to (C^\alpha(\cW^s))^*$, for a given weight function $g$ by
\[
\cL_g f (\psi) = f \big(e^g \tfrac{\psi \circ T}{J^sT} \big), \quad \psi \in C^\alpha(\cW^s) \, .
\]

This operator is well defined because, if $\psi \in C^\alpha(\cW^s)$ then $e^g \, \psi \circ T \in C^\alpha(\cW^s)$. Furthermore, since $J^sT$ and $\cos \varphi$ are $1/3$-log-H\"older on homogeneous stable manifolds, and $\cos \varphi / J^sT$ is bounded  away from $0$ and $+\infty$ also on homogeneous stable manifolds, we get that $1/J^sT \in C^\alpha_{\cos}(\cW^s_{\bH})$. Thus $e^g \tfrac{\psi \circ T}{J^sT} \in C^\alpha_{\cos}(\cW^s_{\bH})$.

When $f \in C^1(M)$, we identify $f$ with the measure\footnote{To show the claimed inclusion just use that $\mathrm{d}\musrb=(2  |\partial Q|)^{-1} \cos \vf \, \mathrm{d}r\mathrm{d}\vf$.} 
\begin{equation}\label{iddent}
f \musrb \in (C^\alpha_{\cos}(\cW^s_{\bH}))^* \, . 
\end{equation}
The measure above 
is (abusively) still denoted by $f$.  For $f\in C^1(M)$, we have 
\begin{align*}
\cL_g (f\musrb) (\psi) \!
= \!\! \int \!\! f \,e^g \frac{\psi \circ T}{J^sT} \, \mathrm{d}\musrb 
= \!\! \int \!\! \left(e^g \frac{f}{J^s T} \right) \! \circ T^{-1} \psi \, \mathrm{d}\musrb 
= \left( \! \left(e^g \frac{f}{J^s T} \right) \! \circ T^{-1}  \, \musrb\right) \! (\psi).
\end{align*}
Thus, due to the identification \eqref{iddent} we have $\cL_g f = (e^g f /J^sT )\circ T^{-1}$, as claimed above.

\begin{proposition} For any fixed $(\cM_0^1,\alpha_g)$-H\"older potential $g$ and associated spaces $\cB$ and $\cB_w$:
\begin{enumerate}
\item[(i)] If $f \in C^1(M)$, then $\cL_g(f) \in \cB$.
\item[(ii)] The operators $\cL_g : (C^1(M), |\cdot|_w) \to \cB_w$ and $\cL_g : (C^1(M), ||\cdot||_\cB) \to \cB$ are continuous.
\end{enumerate}
In particular, $\cL_g$ extends uniquely into operators on both $\cB_w$ and $\cB$.
\end{proposition}

\begin{proof}
The proof of (i) proceeds similarly as in its analogous result \cite[Lemma~4.3]{BD2} by introducing a mollification $f_\eta \in C^1(M)$ of $\cL_g(f)$, where $f \in C^1(M)$ and $\eta >0$. As noticed in \cite[Remark~4.11]{BD2}, the proof of \cite[Lemma~4.3]{BD2} can be adapted to the case $g=0$. The modification mainly relies on giving a nonhomogeneous version of \cite[Lemma~4.9]{BD2}, which can easily be done for a weaker -- but sufficiently tight -- upperbound. The corresponding bounds on $||\cL_g(f) - f_\eta||_s$ and $||\cL_g(f) - f_\eta||_u$ decrease to $0$ as $\eta$ goes to zero. Now, for any $(\cM_0^1,\alpha_g)$-H\"older potential $g$, one can use same techniques and decompositions to get the bound $||\cL_g(f) - f_\eta||_s \leqslant C_{f,g} \eta^{\alpha_g/6}$. Concerning the unstable norm, we follow the modifications described in \cite[Remark~4.11]{BD2}, but distinguishing between the cases $\eta^{\alpha_g/6} < |\log \ve|^{-2\zeta}$ and $\eta^{\alpha_g/6} \geqslant |\log \ve|^{-2\zeta}$ (instead of $\eta^{1/6}$). In the first case, we obtain a bound $||\cL_g(f) - f_\eta||_u \leqslant C_{f,g} \eta^{\alpha_g/12}$, whereas in the second case (in which $\ve$ is superexponentially small in $\eta$) we get an upperbound of the form $C_{f,g} \eta^{-\frac{18+\alpha_g}{12}} \exp \Big( -\min(\alpha-\beta, \frac{1}{2}) \eta^{-\frac{\alpha_g}{12\zeta}} \Big)$. Therefore, both $||\cL_g(f) - f_\eta||_s$ and $||\cL_g(f) - f_\eta||_u$ converge to $0$ as $\eta$ goes to zero.

Point (ii) follows from Proposition~\ref{prop:upper_bounds_norms}, in the case $n=1$.
\end{proof}

\section{Norm Estimates and Spectral Radius}\label{sect:norm_estimates}

The purpose of this section is to state and prove sharp upper and lower bounds on the norm of the iterated operator $\cL_g^n$, both in $\cB_w$ and $\cB$. These estimates are instrumental for Section~\ref{sect:measure_mu_g}. Upper bounds are the content of Proposition~\ref{prop:upper_bounds_norms}, while lower bounds are the content of Theorem~\ref{thm:spectral_radius}. We conclude this section with a discussion (Remark~\ref{remark:discussion_assumption}) on the assumptions made in the specific case $g=-h_{\rm top}(\phi_1) \tau$. In particular, we provide billiard tables for which the inequality $P_*(T,g) - \sup g > s_0 \log 2$ holds. The SSP.1 and SSP.2 conditions are also being discussed. 

\begin{proposition}\label{prop:upper_bounds_norms}
Let $g$ be a $(\cM_0^1,\alpha_g)$-H\"older continuous potential. Assume that $P_*(T,g) - \sup g > s_0 \log 2$ and that SSP.1 holds. Then there exist $\delta_0$ and $C > 0$ such that for all $f \in \cB$,
\begin{align}
| \cL_g^n f |_w & \leqslant \frac{C}{\delta_0} e^{n P_*(T,g) } | f |_w \, , \quad \forall n\geqslant 0  \; ; \label{eq:weak ly poly} \\
\|\cL_g^n f\|_s & \leqslant \frac{C}{\delta_0} e^{n P_*(T,g)}  \|f\|_s   \, , \quad \forall n\geqslant 0 \; ; \label{cheapeq:stable ly poly} \\
\| \cL_g^n f \|_u & \leqslant \frac{C}{\delta_0} (\| f \|_u + \| f \|_s)  e^{n P_*(T,g)}  \, , \quad \forall n \geqslant 0 \; . \label{eq:unstable ly poly}
\end{align}
It follows that the spectral radius of $\cL_g$ on $\cB$ and $\cB_w$ is at most $e^{P_*(T,g)}$.
\end{proposition}

\begin{remark}
It is possible to obtain similar estimates without the assumption SSP.1, however an additional factor $e^{n\varepsilon}$ appears on the right hand sides, for any arbitrary $\varepsilon >0$. We indicate places in the proof where it happens and how to correct for it. The conclusion about the upper bound of the spectral radius still holds. Nonetheless, in order to construct nontrivial maximal eigenvectors, we will need the estimates from Proposition~\ref{prop:upper_bounds_norms}.
\end{remark}

\begin{theorem}\label{thm:spectral_radius}
Let $g$ be a $(\cM_0^1,\alpha_g)$-H\"older continuous potential. Assume that $P_*(T,g) - \sup g > s_0 \log 2$ and that SSP.1 holds. Then there exists $C$ such that 
\[ ||\cL_g^n 1 ||_s \geqslant |\cL_g^n 1|_w \geqslant C \frac{\delta_1}{2} e^{n P_*(T,g)}. \]
\end{theorem}

\begin{proof}[Proof of Proposition~\ref{prop:upper_bounds_norms}]
Let $\delta_0$ be the scale associated to $g$ as in the beginning of Section~\ref{subsect:Frag_lemmas--round_1}. The  set $\cW^s$ is defined with respect to the scale $\delta_0$.

We start with the weak norm estimate \eqref{eq:weak ly poly}. Let $f \in C^1(M)$, $W \in \cW^s$ and $\psi \in C^{\alpha}(W)$ be such that $|\psi|_{C^{\alpha}(W)} \leqslant 1$. For $n\geqslant 0$ we use the definition of the weak norm on each $W_i \in \cG_n^{\delta_0}(W)$ to estimate
\[ \int_W \cL_g^n f \psi \, \mathrm{d}m_W = \sum\limits_{W_i \in \cG_n^{\delta_0}(W)} \int_{W_i} f e^{S_n g} \psi\circ T^n \mathrm{d}m_{W_i} \leqslant |f|_w \sum\limits_{W_i \in \cG_n^{\delta_0}(W)} | e^{S_n g} |_{C^{\alpha}(W_i)} | \psi \circ T^n |_{C^{\alpha}(W_i)}. \]

Clearly, $\sup |\psi\circ T^n|_{W_i}\leqslant \sup_W |\psi|$.
For $x,y \in W_i$, we have,
\begin{align}\label{eq:Calpha_norm_observable_Tn}
\frac{|\psi (T^nx) - \psi (T^ny)|}{d_W(T^nx,T^ny)^\alpha}
\cdot \frac{d_W(T^nx,T^ny)^\alpha}{d_W(x,y)^\alpha} &\leqslant C |\psi|_{C^\alpha(W)} 
|J^sT^n|^\alpha_{C^0(W_i)} 
\leqslant C\Lambda^{-\alpha n} |\psi|_{C^\alpha(W)}\, ,
\end{align}
so that  $H_{W_i}^\alpha(\psi\circ T^n)
\leqslant C\Lambda^{-\alpha n} H_{W}^\alpha(\psi)$ and thus $|\psi \circ T^n|_{C^\alpha(W_i)} \leqslant C |\psi|_{C^\alpha(W)}$.
By Lemma~\ref{lemma:norm_C0_exp_birkhoff}, we get 
\begin{align*}
\int_W \cL_g^n f \psi \, \mathrm{d}m_W &\leqslant C |f|_w |\psi|_{C^\alpha(W)} \sum\limits_{W_i \in \cG_n^{\delta_0}(W)} | e^{S_n g} |_{C^{0}(W_i)} 
\leqslant \frac{2C}{\delta_0} |f|_w |\psi|_{C^\alpha(W)} \sum\limits_{A \in \cM_0^n} | e^{S_n g} |_{C^{0}(A)}, \\
&\leqslant \frac{2C}{c_1 \delta_0} |f|_w |\psi|_{C^\alpha(W)} e^{nP_*(T,g)},
\end{align*}
where the second inequality uses that there are no more than $2\delta_0^{-1}$ curves $W_i$ of $\cG_n^{\delta_0}(W)$ per element of $\cM_0^n$, and the third inequality uses the Exact Exponential Growth from Proposition~\ref{prop:almost_exponential_growth}\footnote{without the assumption SSP.1, Proposition~\ref{prop:almost_exponential_growth} might not hold. Still, for $\varepsilon >0$ and all $n \geqslant 1$, $\sum_{A \in \cM_0^n} |e^{S_n g}|_{C^0(A)} \leqslant C_\varepsilon \, e^{n(P_*(T,g) + \varepsilon)}$ because of the subadditivity from Theorem~\ref{thm:pressure}. }. Taking the appropriate $\sup$ over $\psi$ and $W$ we obtain \eqref{eq:weak ly poly}.

\vspace{1em}

Now we prove the strong stable norm estimate \eqref{cheapeq:stable ly poly}. We can choose $m$ so large that $2^{s_0\gamma}(Km+1)^{1/m} < e^{P_*(T,g) - \sup g}$. Let $W \in \cW^s$, $\psi \in C^{\beta}(W)$ such that $|\psi|_{C^{\beta}(W)} \leqslant |\log |W| |^{\gamma}$. Then, by definition of the strong norm
\begin{align*}
\int_W  \!\! \cL_g^n f \psi \mathrm{d}m_W 
&= \!\!\!\!\!\!\! \sum\limits_{W_i \in \cG_n^{\delta_0}(W)} \!\!\!\!\! \int_{W_i} \!\! f \psi \circ T^n e^{S_n g} \mathrm{d}m_{W_i} 
\leqslant \!\!\!\!\!\!\! \sum\limits_{W_i \in \cG_n^{\delta_0}(W)} \!\!\!\!\!\!\! ||f||_s |\psi \circ T^n|_{C^{\beta}(W_i)} |e^{S_n g}|_{C^{\beta}(W_i)} |\log |W_i| |^{-\gamma} \\
&\leqslant C ||f||_s \sum\limits_{W_i \in \cG_n^{\delta_0}(W)} \left( \frac{\log |W|}{ \log |W_i|} \right)^{\gamma} |e^{S_n g}|_{C^{\beta}(W_i)}\\
&\leqslant C ||f||_s 2^{2\gamma +1} \delta_0^{-1} \sum\limits_{j=1}^n 2^{js_0 \gamma}(Km+1)^{j/m} e^{j \sup g} \sum\limits_{A \in \cM_0^{n-j}} |e^{S_{n-j} g}|_{C^{0}(A)}
\end{align*}
where for the last line we used Lemma~\ref{lemma:Growth_lemma}(b) and Lemma~\ref{lemma:norm_C0_exp_birkhoff}. Let 
\[
 D_n \coloneqq C 2^{2\gamma +1} \delta_0^{-1} \sum\limits_{j=1}^n 2^{js_0 \gamma}(Km+1)^{j/m} e^{j \sup g} \sum\limits_{A \in \cM_0^{n-j}} |e^{S_{n-j} g}|_{C^{0}(A)}. 
\]
Let $\varepsilon_1 = P_*(T,g) - \sup g - \log (2^{s_0 \gamma} (Km+1)^{1/m}) > 0$. Using Proposition~\ref{prop:almost_exponential_growth}, we obtain\footnote{Here, again, conclusion from Proposition~\ref{prop:almost_exponential_growth} can be replaced.}

\begin{align*}
D_n \leqslant 2^{2\gamma +1} \frac{C}{c_1 \delta_0} \sum\limits_{j=1}^n e^{(P_*(T,g) - \varepsilon_1)j} e^{(n-j)P_*(T,g)} 
\leqslant 2^{2\gamma +1} \frac{1}{1 - e^{-\varepsilon_1}}  \frac{C}{c_1 \delta_0} e^{n P_*(T,g)}.
\end{align*} 
Combining this with the previous estimate and taking appropriate $\sup$ yield \eqref{cheapeq:stable ly poly}.

\vspace{1em}

Finally, we now prove the strong unstable norm estimate (\ref{eq:unstable ly poly}). Fix $\tilde{\varepsilon} < \varepsilon_0$, and consider two curves $W^1$, $W^2 \in \cW^s$ with $d_{\cW^s}(W^1,W^2) < \tilde{\varepsilon}$. For $n \geqslant 1$, we recall how, as described in \cite[\S~6.2]{BD2020MME}, $T^{-n}W^\ell$ is partitioned into ``matched'' pieces $U^\ell_j$ and ``unmatched'' pieces $V^\ell_i$, $\ell=1,2$.

More precisely, to each $x \in T^{-n}\omega$, where $\omega$ is a connected component of $W^1 \smallsetminus \cS_{-n}$, we associate a vertical line segment $\gamma_x$ of length most $C \Lambda^{-n} \tilde \ve$ (so that its image $T^n \gamma_x$, if not cut by a singularity, has length at most $C \tilde \ve$). By \cite[\S~4.4]{chernov2006chaotic}, we get that the $T^i \gamma_x$ are made of cone-unstable curves, and thus enjoy minimum expansion given by $\Lambda$.

Doing this for all such $x$, we partition of $W^1 \smallsetminus \cS_{-n}$ into countably many subintervals of points for which $T^n \gamma_x$ intersects $W^2 \smallsetminus \cS_{-n}$ (but not $\cS_{-n}$), and subintervals of points for which this is not the case. This induces a corresponding partition of $W^2 \smallsetminus \cS_{-n}$.

Denote by $V_i^\ell \subset T^{-n} W^\ell$ the subintervals that are not matched. Note that the $T^n V_i^\ell$ occur either at the endpoints of $W^\ell$ or because the vertical segment $\gamma_x$ cuts $\cS_n$. In both cases, due to the uniform transversality of $\cS_{-n}$ and $C^u$ with $C^s$, one must have that $|T^n V_i^\ell| \leqslant C \tilde \ve$.

In the remaining subintervals the foliation $\{ T^n\gamma_x \}_{x \in T^{-n}W^1}$ provides a one-to-one correspondence between points in $W^1$ and $W^2$. We cut these subintervals into pieces such that the lengths of their images under $T^{-i}$ are less than $\delta_0$ for each $0 \leqslant i \leqslant n$ and the pieces are pairwise matched by the foliation $\{\gamma_x\}$. We call these matched pieces $U^\ell_j \subset T^{-n} W^\ell$, such that 
\begin{equation}
\label{eq:match}
\begin{split}
\mbox{If } \; & U^1_j = \{ G_{U_j^1}(r) \coloneqq (r, \vf_{U^1_j}(r)) \mid r \in I_j \}\, \, \mbox{then } U^2_j = \{ G_{U_j^2}(r) \coloneqq (r, \vf_{U^2_j}(r)) \mid r \in I_j \} \, ,
\end{split}
\end{equation}
so that the point $x = (r, \vf_{U^1_j}(r))$ is associated with the point $\bar x = (r, \vf_{U^2_j}(r))$ by the vertical segment $\gamma_x \subset \{(r,s)\}_{s\in[-\pi/2, \pi/2]}$, for each $r \in I_j$. Furthermore, since the stable cone is bounded away from the vertical direction, we can adjust the elements of $\cG_n^{\delta_0}(W^\ell)$ created by artificial subdivisions (the ones due to length) so that $U^\ell_j \subset W^\ell_i$ and $V^\ell_k \subset W^\ell_{i'}$ for some $W^\ell_i, W^\ell_{i'} \in \cG_n^{\delta_0}(W^\ell)$ for all $j,k \geqslant 1$ and $\ell = 1,2$, without changing the bounds on sums over $\cG_n^{\delta_0}(W^\ell)$. There is at most one $U^\ell_j$ and two $V^\ell_j$ per $W^\ell_i \in \cG_n^{\delta_0}(W^\ell)$.

Thus we have the decompositions $W^\ell = (\cup_j T^nU^\ell_j) \cup (\cup_i T^nV^\ell_i)$, $\ell = 1, 2$.

Given $\psi_\ell$ on $W^\ell$ with $|\psi_\ell|_{C^\alpha(W^\ell)} \leqslant 1$ and $d(\psi_1, \psi_2) \leqslant \tilde{\ve}$, we must estimate

\begin{align}\label{eq:strong_unstable_norm_matched+unmatched}
\begin{split}
\left| \int_{W^1} \cL_g^n f \psi_1 \mathrm{d}m_{W_1} - \int_{W^2} \cL_g^n f \psi_2 \mathrm{d}m_{W_2} \right| &\leqslant \sum\limits_{l,i} \left| \int_{V_i^l} f \psi_l\circ T^n e^{S_n g} \mathrm{d}m \right|  \\ &\, + \sum\limits_{j} \left| \int_{U_j^1} f \psi_1 \circ T^n e^{S_n g} \mathrm{d}m - \int_{U_j^2} f \psi_2 \circ T^n e^{S_n g} \mathrm{d}m    \right|.
\end{split}
\end{align}
First, we estimate the differences of matched pieces $U_j^l$. The function $\phi_j= ( \psi_1 \circ T^n e^{S_n g}) \circ G_{U_j^1} \circ G_{U_j^2}^{-1}$ is well defined on $U_j^2$, and we can estimate each difference by
\begin{align}\label{eq:matched_pieces}
\begin{split}
\left| \int_{U_j^1} f \psi_1 \circ T^n e^{S_n g} \mathrm{d}m - \int_{U_j^2} f \psi_2 \circ T^n e^{S_n g} \mathrm{d}m    \right| &\leqslant \left| \int_{U_j^1} f \psi_1 \circ T^n e^{S_n g} \mathrm{d}m - \int_{U_j^2} f \phi_j \mathrm{d}m \right| \\ &+ \left| \int_{U_j^2} f ( \phi_j - \psi_2 \circ T^n e^{S_n g} ) \mathrm{d}m \right|.
\end{split}
\end{align}
We bound the first term in equation \eqref{eq:matched_pieces} using the strong unstable norm. We have that $|G_{U^1_j} \circ G_{U^2_j}^{-1}|_{C^1} \leqslant C_g$, for some $C_g > 0$ due to the fact that each curve $U_j^l$ has uniformly bounded curvature and slopes bounded away from infinity. Thus $|\phi_j|_{C^{\alpha}(U_j^2)} \leqslant C C_g |\psi_1|_{C^\alpha(W^1)} |e^{S_n g}|_{C^\alpha(W^1)}$. Moreover, $d(\psi_1 \circ T^n e^{S_n g}, \phi_j) = |\psi_1 \circ T^n e^{S_n g} \circ G_{U^1_j}  - \phi_j \circ G_{U^2_j}| = 0$ by definition of $\phi_j$. To complete the bound on the first term, we use the following estimate from \cite[Lemma 4.2]{DZ1}: There exists $C > 0$, independent of $W^1$ and $W^2$, such that 
\begin{align}\label{eq:closeness_matched_pieces}
d_{\cW^s}(U_j^1,U_j^2) \leqslant C \Lambda^{-n} n \tilde \varepsilon \eqqcolon \varepsilon_1, \quad \forall j.
\end{align}
Then, applying the definition of the strong unstable norm with $\varepsilon_1$ instead of $\tilde{\varepsilon}$
\begin{align}\label{eq:bound_diff_int_matched_pieces}
\sum\limits_{j} \left| \int_{U_j^1} f \psi_1 \circ T^n e^{S_n g} \mathrm{d}m - \int_{U_j^2} f \phi_j \mathrm{d}m \right|  \leqslant 2 \delta_0^{-1} C C_g^2 |\log \varepsilon_1|^{-\zeta} ||f||_u \sum\limits_{A \in \cM_0^n } |e^{S_n g}|_{C^0(A)},
\end{align}
where we used Lemmas~\ref{lemma:norm_C0_exp_birkhoff} and~\ref{lemma:Growth_lemma}(b) with $\gamma =0$ since there is at most one matched piece $U^1_j$ corresponding to each component $W^1_i \in \cG_n^{\delta_0}(W^1)$ of $T^{-n} W^1$.

It remains to estimate the second term of \eqref{eq:matched_pieces} using the strong stable norm.
\[ \left| \int_{U_j^2} f ( \phi_j - \psi_2 \circ T^n e^{S_n g} ) \mathrm{d}m \right| \leqslant ||f||_s |\log |U^2_j| |^{-\gamma} |\phi_j - \psi_2 \circ T^n e^{S_n g}|_{C^{\beta}(U_j^2)}. \]
In order to estimate this last $C^{\beta}$-norm, we use that $|G_{U_j^2}|_{C^1} \leqslant C_g$ and $|G_{U_j^2}^{-1}|_{C^1} \leqslant C_g$. Hence
\begin{align}\label{eq:diff_phi_psi_matched_pieces}
\begin{split}
|\phi_j - \psi_2 \circ T^n e^{S_n g}|_{C^{\beta}(U_j^2)} &\leqslant C | (\psi_1 \circ T^n e^{S_n g}) \circ G_{U_j^1} - (\psi_2 \circ T^n e^{S_n g}) \circ G_{U_j^2} |_{C^{\beta}(I_j)} \\
&\leqslant C | (\psi_1 \circ T^n \circ G_{U_j^1} - \psi_2 \circ T^n \circ G_{U_j^2}) |_{C^{\beta}(I_j)} \, | e^{S_n g} |_{C^{0}(U_j^1)} \\
&\qquad\quad + C |\psi_2|_{C^{0}(U_j^2)} \, |e^{S_n g} \circ G_{U_j^1} - e^{S_n g} \circ G_{U_j^2} |_{C^{\beta}(I_j)}.
\end{split}
\end{align}
It follows from \cite[Lemma 4.4]{DZ1} that 
\[ | \psi_1 \circ T^n \circ G_{U_j^1} - \psi_2 \circ T^n \circ G_{U_j^2} |_{C^{\beta}(I_j)} \leqslant C \tilde{\varepsilon}^{\alpha - \beta}.\]
Now, we need to estimate $|e^{S_n g} \circ G_{U_j^1} - e^{S_n g} \circ G_{U_j^2} \, |_{C^{\beta}(I_j)}$. Since $d(T^i (G_{U_j^1}(r)),T^i (G_{U_j^2}(r))) \leqslant C \Lambda^{-(n-i)} \tilde{\ve}$ for all $r \in I_j$ and $0 \leqslant i \leqslant n$, we get
\begin{align}\label{eq:C0_norm_matched_pieces}
\begin{split}
&|e^{S_n g} \circ G_{U_j^1}(r) - e^{S_n g} \circ G_{U_j^2}(r) | 
= e^{S_n g(G_{U_j^1}(r))} \, \left| 1 - e^{S_n g(G_{U_j^2}(r))-S_n g(G_{U_j^1}(r))} \right| \\
&\,\,\,\, \leqslant 2 |e^{S_n g}|_{C^0(U_j^1)} |S_n g(G_{U_j^2}(r))-S_n g(G_{U_j^1}(r))| 
\leqslant 2C \frac{\Lambda^{\alpha_g}}{\Lambda^{\alpha_g} -1} |g|_{C^{\alpha_g}} (C\tilde{\ve})^{\alpha_g} |e^{S_n g}|_{C^0(U_j^1)}
\end{split}
\end{align} 
We estimate the $\beta$-H\"older constant in two ways. First, using \eqref{eq:C0_norm_matched_pieces} twice, we have for all $r,\, s \in I_j$ that
\begin{align*}
|e^{S_n g} \circ G_{U_j^1}(r) - e^{S_n g} \circ G_{U_j^2}(r) - e^{S_n g} \circ G_{U_j^1}(s) + e^{S_n g} \circ G_{U_j^2}(s)| \leqslant C \tilde \ve^{\alpha_g} |e^{S_n g}|_{C^0(U_j^1)}.
\end{align*}
On the other hand, using that $G_{U_j^\ell}(r)$ and $G_{U_j^\ell}(s)$ lie on the same stable curve,
\begin{align*}
|e^{S_n g} &\circ G_{U_j^1}(r) - e^{S_n g} \circ G_{U_j^2}(r) - e^{S_n g} \circ G_{U_j^1}(s) + e^{S_n g} \circ G_{U_j^2}(s)| \\
&\leqslant |e^{S_n g} \circ G_{U_j^1}(r) - e^{S_n g} \circ G_{U_j^1}(s) + e^{S_n g} \circ G_{U_j^2}(s) - e^{S_n g} \circ G_{U_j^2}(r)| \\
&\leqslant |e^{S_n g}|_{C^{\alpha_g}(U_j^1)} d(G_{U_j^1}(r),G_{U_j^1}(s))^{\alpha_g} + |e^{S_n g}|_{C^{\alpha_g}(U_j^2)} d(G_{U_j^2}(r),G_{U_j^2}(s))^{\alpha_g} \\
&\leqslant C |e^{S_n g}|_{C^{0}(U_j^1)} |r-s|^{\alpha_g}.
\end{align*}
Thus, this quantity constant is bounded by the $\min$ of the two estimates. This $\min$ is maximal when the two upperbounds are equal, that is when $\tilde \ve = C |r-s|$. Therefore, dividing by $|r-s|^{\beta}$, the $\beta$-H\"older constant satisfies 
\[ H^\beta_{I_j} (  e^{S_n g} \circ G_{U_j^1} - e^{S_n g} \circ G_{U_j^2} ) \leqslant C \tilde \ve ^{\alpha_g - \beta} |e^{S_n g}|_{C^{0}(U_j^1)} \, . \]
Therefore, we have proved that 
\[ |e^{S_n g} \circ G_{U_j^1} - e^{S_n g} \circ G_{U_j^2} |_{C^{\beta}(I_j)} \leqslant C \tilde \ve ^{\alpha_g - \beta} |e^{S_n g}|_{C^{0}(U_j^1)} \, . \]
Combining the above estimates inside \eqref{eq:diff_phi_psi_matched_pieces}, we finally have
\[ |\phi_j - \psi_2 \circ T^n e^{S_n g}|_{C^{\beta}(U_j^2)} \leqslant C \tilde \ve ^{\alpha - \beta} |e^{S_n g}|_{C^{0}(U_j^1)} \, . \]
Summing over $j$ yields
\[ \sum\limits_{j} \left| \int_{U_j^2} f ( \phi_j - \psi_2 \circ T^n e^{S_n g} ) \mathrm{d}m \right| \leqslant C |\log \delta_0 |^{-\gamma} ||f||_s \tilde{\varepsilon}^{\alpha - \beta} 2 \delta_0^{-1} \sum\limits_{A \in \cM_0^n} |e^{S_n g}|_{C^0(A)},
\]
where we used Lemma~\ref{lemma:Growth_lemma}(b) with $\gamma =0$ since there is at most one matched piece $U_j^l$ corresponding to each component $W^l_i \in \cG_n^{\delta_0}(W^l)$ of $T^{-n}W^l$. Since $\delta_0 < 1$ is fixed, this completes the estimate on the second term of \eqref{eq:matched_pieces}. Hence the contribution of matched pieces in \eqref{eq:strong_unstable_norm_matched+unmatched} is controlled. 

\medskip
We now turn to the estimate of the first sum in \eqref{eq:strong_unstable_norm_matched+unmatched} concerning the unmatched pieces.

We say an unmatched curve $V^1_i$ is created at time $j$, $1 \leqslant j \leqslant n$, if  $j$ is the first time that $T^{n-j}V^1_i$ is not part of a matched element of $\cG_j^{\delta_0}(W^1)$.  Indeed, there may be several curves $V^1_i$ (in principle exponentially many in $n-j$) such that $T^{n-j}V^1_i$ belongs to the same unmatched element of $\cG_j^{\delta_0}(W^1)$. Define
\begin{align*}
A_{j,k} = \{ i \mid V^1_i &\mbox{ is created at time $j$} 
\\ &
\mbox{and $T^{n-j}V^1_i$ belongs to the
unmatched curve $W^1_k \subset T^{-j}W^1$} \} \, .
\end{align*}
Due to the uniform hyperbolicity of $T$,  and, again, uniform transversality of $\cS_{-n}$ with the stable cone and the one of $C^s(x)$ with $C^u(x)$, we have $|W^1_k| \leqslant C \Lambda^{-j} \tilde \ve$.

Recall that from Lemma~\ref{lemma:Growth_lemma}(a) for $\bar \gamma = 0$, if for a certain time $q$, every element of $\cG_q^{\delta_0}(W_k^1)$ have length less than $\delta_0/3$ -- that is, if $\cG_q^{\delta_0}(W_k^1) = \cI_q^{\delta_0}(W_k^1)$ -- then we have the subexponential growth
\begin{align}\label{eq:slow grow}
\sum_{V \in \cG_q^{\delta_0}(W^1_k)} |e^{S_q g}|_{C^0(V)} \leqslant 2 (Km +1)^{q/m} e^{q \sup g} \, .
\end{align}
We would like to establish a lower bound on the value of $q$ as a function of $j$.

More precisely, we want to find $q(j)$, as large as possible, so that 

	(a) $\cG_{q(j)}^{\delta_0}(W_k^1) = \cI_{q(j)}^{\delta_0}(W_k^1)$; \qquad
	
	(b) $\displaystyle \frac{|\log |V||^{-\gamma}}{|\log \tilde \ve|^{-\varsigma}} \leqslant 1$, for all $V \in \cI_{q(j)}^{\delta_0}(W_k^1)$.\\
This is the content of the next two lemmas.

\begin{lemma}\label{lemma:very_short_curves}
If $W \in \hW^s$ is such that $\tilde C^2 |W|^{2^{-kn_0s_0}} < \delta_0/3$ for some $k \geqslant 1$, where $\tilde C$ is the constant from \eqref{eq:sqrt_bound}. Then $\cG_{kn_0}^{\delta_0}(W) = \cI_{kn_0}^{\delta_0}(W)$, and for all $1 \leqslant l \leqslant k$, and all $W_i \in \cG_{ln_0}^{\delta_0}(W)$, $|W_i| \leqslant \tilde C^2 |W|^{2^{-ln_0s_0}}$.
\end{lemma}

\begin{proof}
We prove the lemma by induction on $k$. We start with the case $k=1$. Let $1 \leqslant l \leqslant n_0$ and $W_i \in \cG_{l}^{\delta_0}(W)$. Denote $V = T^l W_i \subset W$. Then, for all $0 \leqslant j \leqslant l$, $|T^j W_i| \leqslant \delta_0$. Decomposing $T^{-l}V = W_i$ as in the beginning of the proof of Lemma~\ref{lemma:Growth_lemma}, we get that $|W_i| \leqslant \tilde C |W|^{2^{-n_0s_0}}$, which is less than $\delta_0/3$ by assumption. Thus, $\cG_{l}^{\delta_0}(W) = S_{l}^{\delta_0}(W)$ for each $0 \leqslant l \leqslant n_0$. Therefore $\cG_{n_0}^{\delta_0}(W) = \cI_{n_0}^{\delta_0}(W)$, with the claimed estimate.

Consider now the case $k>1$. Notice that, by construction, we have 
\[
\cG_{(k+1)n_0}^{\delta_0}(W) = \bigcup_{W_i \in \cG_{kn_0}^{\delta_0}(W)} \cG_{n_0}^{\delta_0}(W).
\]
Thus, we can apply the same method to estimate the length of an element $W_i \in \cG_{kn_0+l}^{\delta_0}(W)$ from the length of its \emph{parent} in $\cG_{kn_0}^{\delta_0}(W)$, iterating the estimates as for \eqref{eq:sqrt_bound}.
\end{proof}

\begin{lemma}
The above conditions (a) and (b) are satisfied for $q(j) \coloneqq \frac{(\gamma-\zeta) \log(j-j_0)}{\gamma s_0 \log2} -1$, for all $j \geqslant j_1$, where $j_1 > j_0 \geqslant 0$ are constants (uniform in $\tilde \ve$ and $W^1_k$). For $j < j_1$, set $q(j)=0$.
\end{lemma}
\begin{proof}
Since $|W_k^1| \leqslant C \tilde \ve \Lambda^{-j}$ and using Lemma~\ref{lemma:very_short_curves}, the condition (a) will be satisfied whenever $\tilde C^2 (C \tilde \ve \Lambda^{-j})^{2^{-q s_0}} \leqslant \delta_0/3$. Let $j_0$ be such that $C \Lambda^{-j_0} < 1$. Then (a) is satisfied whenever $\tilde C^2 \Lambda^{-(j-j_0)2^{-q}} \leqslant \delta_0/3$, that is
\begin{align}\label{cond_1_q(j)}
q \leqslant \frac{\log(j-j_0)}{s_0 \log 2} - C_2\, , \text{ with } C_2 \coloneqq \frac{1}{s_0 \log 2} \log \frac{\log \frac{3 \tilde C^2}{\delta_0}}{\log \Lambda} .
\end{align}
Note that $C_2$ is uniform, and that the right-hand-side of \eqref{cond_1_q(j)} is larger than $q(j)$ for all $j$ large enough, say $j \geqslant j_1$.

Using the estimate from Lemma~\ref{lemma:very_short_curves}, condition (b) is satisfied whenever $|\log \tilde C^2 (C\tilde \ve \Lambda^{-j})^{2^{-q}} |^{\gamma} > |\log \tilde \ve |^{\zeta}$. Now, we have that
\[
|\log \tilde C^2 (C\tilde \ve \Lambda^{-j})^{2^{-q}} | = |\log \tilde C^2  + 2^{-q} \log (C\tilde \ve \Lambda^{-j}) | > \frac{1}{2} |2^{-q} \log (C\tilde \ve \Lambda^{-j}) |,
\]
whenever
\begin{align}\label{cond_2_q(j)}
q +1 \leqslant \frac{\log(j-j_0)}{s_0 \log 2} + C_3 \, , \text{ with } C_3 = \frac{1}{s_0 \log 2} \log \frac{ \log \Lambda}{\log \tilde C^2}
\end{align}
Note that $C_3$ is uniform, and that the right-hand-side of \eqref{cond_2_q(j)} is larger than $q(j)$ for all $j$ large enough, say $j \geqslant j_1$ (up to increasing the value of $j_1$).

We thus have to prove that $|\log C \tilde \ve \Lambda^{-j}|^\gamma > 2^{(q+1)\gamma}|\log \tilde \ve|^\zeta$ (which implies (b)). Notice that, from the definition of $q(j)$, we have $2^{(q(j)+1)\gamma} \leqslant (j-j_0)^{\gamma - \zeta}$. We distinguish two cases.

Assume first that $(j-j_0)\log \Lambda \geqslant | \log \tilde \ve |$. Therefore
\begin{align*}
2^{(q(j)+1)\gamma}&|\log \tilde \ve|^\zeta \leqslant (j-j_0)^{\gamma - \zeta} |\log \tilde \ve|^\zeta \leqslant (j-j_0)^\gamma (\log \Lambda)^\zeta \leqslant ((j-j_0)\log \Lambda)^\gamma \\
&\leqslant ((j-j_0)\log \Lambda + |\log \tilde \ve| + |\log C\Lambda^{-j_0}| )^\gamma 
\leqslant |-(j-j_0)\log \Lambda + \log \tilde \ve + \log C\Lambda^{-j_0} |^\gamma \\
&\leqslant |\log C \tilde \ve \Lambda^{-j}|^\gamma.
\end{align*}
On the other hand, if $(j-j_0)\log \Lambda \leqslant | \log \tilde \ve |$, then
\begin{align*}
2^{(q(j)+1)\gamma}|\log \tilde \ve|^\zeta &\leqslant (j-j_0)^{\gamma - \zeta} |\log \tilde \ve|^\zeta \leqslant \frac{|\log \tilde \ve|^{\gamma-\zeta}}{(\log \Lambda)^{\gamma-\zeta} } |\log \tilde \ve |^\zeta \leqslant |\log \tilde \ve|^\gamma \\
&\leqslant |-(j-j_0)\log \Lambda + \log \tilde \ve + \log C\Lambda^{-j_0} |^\gamma 
\leqslant |\log C \tilde \ve \Lambda^{-j}|^\gamma.
\end{align*}
Thus, the choice $q(j)$ satisfies (a) and (b) for all $j \geqslant j_1$.
\end{proof}

We now resume the proof of Proposition~\ref{prop:upper_bounds_norms}. 

We next estimate\footnote{For the $4^{\rm th}$ and $6^{\rm th}$ inequalities, we use Proposition~\ref{prop:almost_exponential_growth}. Here again, $P_*(T,g)$ can be replaced by $P_*(T,g)+\varepsilon$ up to a larger multiplicative constant.} over the unmatched pieces $V_i^l$ in (\ref{eq:strong_unstable_norm_matched+unmatched}), using the strong stable norm. Since cases $l=1$ and $l=2$ are similar here, we only deal with the case $l=1$.
\begin{align*}
& \sum\limits_{V^1_i} \left| \int_{V^1_i} f \psi_1 \circ T^n e^{S_n g} \mathrm{d}m_{V_i^1} \right| 
= \sum\limits_{j=1}^n \sum\limits_{k} \sum\limits_{i \in A_{j,k}} \left| \int_{T^{n-j}V_i^1} (\cL_g^{n-j} f) \psi_1 \circ T^j e^{S_j g} \right| \\
&\leqslant \sum\limits_{j=1}^n \sum\limits_{k} \sum\limits_{V_l \in \cG_{q(j)}^{\delta_0}(W_k^1)} \left| \int_{V_l} (\cL_g^{n-j-q(j)} f ) \psi_1 \circ T^{j + q(j)} e^{S_{j+q(j)} g} \right| \\
&\leqslant \sum\limits_{j=1}^n \sum\limits_{k} \sum\limits_{V_l \in \cG_{q(j)}^{\delta_0}(W_k^1)} ||\cL_g^{n-j-q(j)} f ||_s C |\log |V_l| |^{-\gamma} |\psi_1 \circ T^{j + q(j)}|_{C^{\beta}(V_l)} | e^{S_{j+q(j)} g} |_{C^{\beta}(V_l) } \\
&\leqslant C ||f||_s \sum\limits_{j=1}^n  \frac{C}{c_1 \delta_0} e^{(n -j-q(j))P_*(T,g)} |\log \tilde \ve |^{-\zeta} \sum\limits_{k} \sum\limits_{V_l \in \cG_{q(j)}^{\delta_0}(W_k^1)}  | e^{S_{j+q(j)} g} |_{C^{\beta}(V_l) } \\
&\leqslant \frac{C}{c_1 \delta_0} ||f||_s \sum\limits_{j=1}^n e^{(n -j-q(j))P_*(T,g)} |\log \tilde \ve |^{-\zeta} \sum\limits_{W^1_k \subset T^{-j}W^1 } \sum\limits_{ V_l \in \cG_{q(j)}^{\delta_0}(W^1_k)} |e^{S_j g \circ T^{q(j)} + S_{q(j)} g} |_{C^{\beta}(V_l)} \\
&\leqslant \frac{C}{c_1 \delta_0} ||f||_s \sum\limits_{j=1}^n e^{(n -j-q(j))P_*(T,g)} |\log \tilde \ve |^{-\zeta} \sum\limits_{W^1_k \subset T^{-j}W^1 } |e^{S_j g} |_{C^{0}(W^1_k)}  \sum\limits_{ V_l \in \cG_{q(j)}^{\delta_0}(W^1_k)} |e^{S_{q(j)} g} |_{C^{0}(V_l)} \\
&\leqslant \frac{C}{c_1 \delta_0} ||f||_s \sum\limits_{j=1}^n e^{(n -j-q(j))P_*(T,g)} |\log \tilde \ve |^{-\zeta} \frac{4C}{c_1 \delta_0} e^{j P_*(T,g)}  e^{q(j) \sup g} (Km+1)^{q(j)/m} \\
&\leqslant \frac{C}{(c_1 \delta_0)^2} ||f||_s |\log \tilde{\varepsilon}|^{-\zeta} e^{n P_*(T,g)} \sum\limits_{j=1}^n e^{-q(j)( P_*(T,g) - \sup g - \frac{1}{m}\log(Km+1))}.
\end{align*}
Now, for $\tilde \varepsilon > 0$, fixed, since we assume that $ P_*(T,g) - \sup g > s_0 \log 2 $, we can chose $m$ large enough and $\zeta$ small enough such that $ \varepsilon_1 \coloneqq P_*(T,g) - \sup g - \frac{1}{m}\log(Km+1) - \frac{\gamma}{\gamma - \zeta}s_0 \log 2 > 0$. By definition of $q(j)$, we obtain that
\begin{align*}
\sum\limits_{j=j_1}^n \!\! e^{-q(j)( P_*(T,g) - \sup g - \frac{1}{m}\log(Km+1))} 
\! &= \!\! \sum\limits_{j=j_1}^n \!\! e^{-  \frac{(\gamma - \zeta) \log (j-j_0)}{\gamma s_0 \log 2} ( \varepsilon_1 +  \frac{\gamma}{\gamma - \zeta}s_0 \log 2)} 
\! = \!\! \sum\limits_{j=j_1}^n (j-j_0)^{-1 -  \frac{\gamma - \zeta}{\gamma s_0 \log 2}  \varepsilon_1 },
\end{align*}
is bounded. The bound (\ref{eq:unstable ly poly}) then follows by combining all the above estimates into (\ref{eq:strong_unstable_norm_matched+unmatched}) and taking the appropriate suprema.
\end{proof}

We now deduce the bounds of Theorem~\ref{thm:spectral_radius} from the growth rate of stable curves proved in Proposition~\ref{prop:GnW_geq_M0n}.

\begin{proof}[Proof of Theorem~\ref{thm:spectral_radius}]
To prove this lower bound on $|\cL_g^n 1|_w$, recall the choice of $\delta_1 > 0$ from Lemma~\ref{lemma:short_curves_rare} for $\varepsilon = 1/4$. Let $W \in \cW^s$ with $|W| \geqslant \delta_1/3$ and set the test function $\psi \equiv 1$. For $n\geqslant n_1$,
\begin{align*}
\int_W \cL_g^n 1 \mathrm{d}m_W = \sum\limits_{W_i \in \cG_n^{\delta_1}(W)} \int_{W_i} e^{S_n g} \mathrm{d}m_{W_i} \geqslant \sum\limits_{W_i \in \cG_n^{\delta_1}(W)} \frac{\delta_1}{2} \inf\limits_{W_i} e^{S_n g} \geqslant \frac{\delta_1}{2} C^{-1} \sum\limits_{W_i \in \cG_n^{\delta_1}(W)} \sup\limits_{W_i} e^{S_n g},
\end{align*} 
where we used Lemma~\ref{lemma:sup_less_poly_inf} for the second inequality, since for each $W_i \in \cG_n^{\delta_1}(W)$ there exists $A \in \cM_0^n$ such that $W_i \subset A$ and $\sup_{W_i} e^{S_n g} \leqslant \sup_{A} e^{S_n g} \leqslant C \inf_{A} e^{S_n g} \leqslant C \inf_{W_i} e^{S_n g}$.
We can now use Proposition~\ref{prop:GnW_geq_M0n} to get
\begin{align}\label{eq:seed_grows_exp}
\int_W \cL_g^n 1 \mathrm{d}m_W \geqslant \frac{\delta_1}{2C} c_0 \sum\limits_{A \in \cM_{-n}^0} |e^{S_n^{-1} g}|_{C^0(A)} \geqslant \frac{\delta_1}{2C} c_0 e^{n P_*(T,g)}.
\end{align}
Thus \[ ||\cL_g^n 1 ||_s \geqslant |\cL_g^n 1|_w \geqslant \frac{\delta_1}{2} c_0 e^{n P_*(T,g)}.\]
Letting $n$ tend to infinity, one obtains $\lim\limits_{n \to \infty} ||\cL_g^n 1 ||_{\cB}^{1/n} \geqslant e^{P_*(T,g)}$.
\end{proof}

\begin{remark}\label{remark:discussion_assumption}
In the case $g = - h_{\rm top}(\phi_1)\tau$, the assumption $P_*(T,g)-\sup g > s_0 \log 2$ in Proposition~\ref{prop:upper_bounds_norms} is implied by the condition $h_{\rm top}(\phi_1)\tau_{\min} > s_0 \log 2$, which is itself implied by $\tau_{\min} h_{\musrb}(T)/\musrb(\tau) > s_0 \log 2$ thanks to the Abramov formula. This latter condition appears to be satisfied for billiards studied by Baras and Gaspard \cite{Baras95} and by Garrido \cite{Garrido97}, as long as $\tau_{\min}$ is \emph{not too small}.

\begin{figure}[ht]
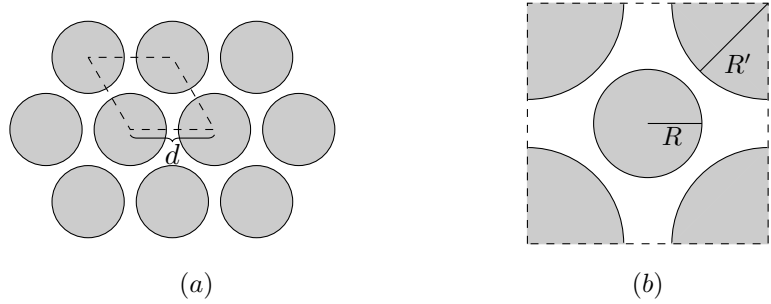

\tikz[x=8mm,y=8mm]
{
 \filldraw[fill=black!20!white, draw=black] (4.7,.7) circle (.6);
\filldraw[fill=black!20!white, draw=black] (6.1,.7) circle (.6);
\filldraw[fill=black!20!white, draw=black] (7.5,.7) circle (.6);
\filldraw[fill=black!20!white, draw=black] (4,1.9) circle (.6);
\filldraw[fill=black!20!white, draw=black] (5.4,1.9) circle (.6);
 \filldraw[fill=black!20!white, draw=black] (6.8,1.9) circle (.6);
  \filldraw[fill=black!20!white, draw=black] (8.2,1.9) circle (.6); 
\filldraw[fill=black!20!white, draw=black] (4.7,3.1) circle (.6);
\filldraw[fill=black!20!white, draw=black] (6.1,3.1) circle (.6);
\filldraw[fill=black!20!white, draw=black] (7.5,3.1) circle (.6);

\draw[dashed] (4.7,3.1) -- (5.4,1.9) -- (6.8,1.9) -- (6.1,3.1) -- cycle;

\draw[decoration={brace,mirror,raise=2pt},decorate]
  (5.4,1.9) -- node[below=1.6pt] {$d$} (6.8,1.9);

\node at (6.5,-.7){\small$(a)$};

\filldraw[fill=black!20!white, draw=black]  (13.6,0) arc (0:90:1.6);
\filldraw[fill=black!20!white, draw=black!20!white]  (13.6,0) -- (12,1.6) -- (12,0) -- cycle;
\filldraw[fill=black!20!white, draw=black]  (16,1.6) arc (90:180:1.6);
\filldraw[fill=black!20!white, draw=black!20!white]  (16,1.6) -- (14.4,0) -- (16,0) -- cycle;

\filldraw[fill=black!20!white, draw=black]  (14.4,4) arc (180:270:1.6);
\filldraw[fill=black!20!white, draw=black!20!white]  (14.4,4) -- (16,2.4) -- (16,4) -- cycle;

\filldraw[fill=black!20!white, draw=black]  (12,2.4) arc (270:360:1.6);
\filldraw[fill=black!20!white, draw=black!20!white]  (12,2.4) -- (13.6,4) -- (12,4) -- cycle;

\draw[dashed] (12,0) rectangle (16,4) ;

\filldraw[fill=black!20!white, draw=black]  (14,2) circle (.9);
\draw (14,2) -- (14.9,2);
\node at (14.4,1.8){\small $R$};

\draw (16,4) -- (14.87,2.87);
\node at (15.5, 3) {\small$R'$};

\node at (14,-.7){\small$(b)$}; 

}
\caption{(a) The Sinai billiard on a triangular lattice studied in \cite{Baras95} with angle $\pi/3$, scatterer of radius 1, and distance $d$ between the centers of adjacent scatterers. (b) The Sinai billiard on a square lattice with scatterers of radius $R < R'$ studied in \cite{Garrido97}. The boundary of a single cell is indicated by dashed lines in both tables. (Figure taken from \cite[Figure~2]{BD2020MME}.)}
\label{fig:tables}
\end{figure}

Indeed, Garrido \cite{Garrido97} studied the Sinai billiard corresponding to the periodic Lorentz gas with two scatterers of radius $R<R'$ on the unit square lattice (Figure~\ref{fig:tables}(b)). Setting $R'=0.4$, Garrido computed $h_{\musrb}(T)$ and $\musrb(\tau)$ for about 20 values of $R$ ranging from $R=0.1$ (when the horizon becomes infinite) to $R=\tfrac{\sqrt{2}}{2}-0.4$ (when the scatterers touch: $\tau_{\min}=0$). According to \cite[\S~2.4]{BD2020MME}, in those examples we can always find $\varphi_0$ and $n_0$ such that $s_0 \leqslant \tfrac{1}{2}$. Furthermore, $\tau_{\min}=\tfrac{\sqrt{2}}{2} - 0.4 - R$. Now, for $R=0.1^+$, we find that \[\tau_{\min} h_{\musrb}(T)/\musrb(\tau) \geqslant (\tfrac{\sqrt{2}}{2} - 0.5) \tfrac{1.7}{0.5} \geqslant 0.7 > \tfrac{1}{2} \log 2 \geqslant s_0 \log 2 \, , \]
and for $R=0.2$, we find that \[\tau_{\min} h_{\musrb}(T)/\musrb(\tau) \geqslant (\tfrac{\sqrt{2}}{2} - 0.6) \tfrac{1.4}{0.3} \geqslant 0.48 > \tfrac{1}{2} \log 2 \geqslant s_0 \log 2 \, . \] 
Since for $R \in (0.1,\, 0.2]$, $R \mapsto \tau_{\min}(R)$ is a linear function, and according to Garrido Figures~6 and 8, $R \mapsto \musrb(\tau)(R)$ is well approximated by an affine function and $R \mapsto h_{\musrb}(T)(R)$ is lower bounded by an affine function joining the values at $R=0.1$ and $0.2$, it appears that the condition $\tau_{\min} h_{\musrb}(T)/\musrb(\tau) > s_0 \log 2$ is satisfied for all $R \in (0.1 , \, 0.2]$.

Baras and Gaspard studied the Sinai billiard corresponding to the Lorentz gas with disks of radius $1$ centered in a triangular lattice (Figure~\ref{fig:tables}(a)). The distance $d$ between points on the lattice is varied from $d=2$ (when the scatterers touch: $\tau_{\min} = 0$) to $d=4/\sqrt{3}$ (when the horizon becomes infinite). We have that $\tau_{\min} = d-2$ and, still according to \cite[\S~2.4]{BD2020MME}, in those examples we can always find $\varphi_0$ and $n_0$ such that $s_0 \leqslant \tfrac{1}{2}$. The computed values are the average Lyapunov exponent of the billiard flows given in \cite{Baras95}, providing a lower bound directly on $h_{\musrb}(T)/\musrb(\tau)$. For $d=0.2$, we find
\[ \tau_{\min} h_{\musrb}(T)/\musrb(\tau) \geqslant (\tfrac{4}{\sqrt{3}} -2) \, 1.8 \geqslant 0.55 > \tfrac{1}{2} \log 2 \geqslant s_0 \log 2 \, . \]

The condition $h_{\rm top}(\phi_1)\tau_{\min} > s_0 \log 2$ is a little bit more restrictive than the one used by Baladi and Demers in \cite{BD2020MME} since, by the Abramov formula, $h_* \geqslant h_{\rm top}(\phi_1)\tau_{\min}$. (Also, we do not know any  example of billiard for which the condition $h_* > s_0 \log 2$ is not satisfied.)

We now turn to the condition SSP.1. Unfortunately, we don't know any billiard table such that the potential $ g = -h_{\rm top}(\phi_1) \tau $ satisfies a \emph{simple} condition implying SSP.1. By \emph{simple}, we mean a sufficient condition that does not involve topological entropies, since they are notoriously hard to estimate numerically. First, recall from Lemmas~\ref{lemma:short_curves_rare} and \ref{lemma:Lu_geq_M-n0} that $\log \Lambda > h_{\rm top}(\phi_1)(\tau_{\max} - \tau_{\min})$ implies SSP.1. Remark that since $g$ and $\tfrac{1}{n} S_n g$ are cohomologous, they would give rise to the same equilibrium states. It is then advantageous to work with the Birkhoff average instead of $g$ because $\max \tfrac{1}{n} S_n g \leqslant \tau_{\max}$ and $\min \tfrac{1}{n} S_n g = \tau_{\min}$ (notice that $\tau_{\min}$ is achieved on an orbit of period $2$). Now, taking advantage of the Abramov formula and of the variational principle, we get that $ \max \tfrac{1}{n} S_n g < 2 \tau_{\min} $ implies $h_* > h_{\rm top}(\phi_1)(\max \tfrac{1}{n} S_n g - \tau_{\min})$ (recall that $h_* > \log \Lambda$ is the topological entropy of $T$, as defined in \cite{BD2020MME}). The condition $ \max \tfrac{1}{n} S_n g < 2 \tau_{\min} $ involves quantities that are easy to estimate numerically, however, we don't know any billiard table satisfying this condition.
\end{remark}

\section{The measure $\mu_g$}\label{sect:measure_mu_g}

This section is devoted to the construction, the properties and the uniqueness of an equilibrium state $\mu_g$ for $T$, associated to a potential $g$. 

We will assume throughout that $g$ is a $(\cM_0^1,\alpha_g)$-H\"older potential such that $P_*(T,g) - \sup g > s_0 \log 2$ and that the conditions SSP.1 and SSP.2 are satisfied.

\subsection{Construction of the measure $\mu_g$ -- Measure of Singular Sets}

In this section, we construct left and right maximal eigenvectors for $\cL_g$ (Proposition~\ref{prop:exist}) as well as a $T$-invariant measure $\mu_g$ by pairing them as in \eqref{eq:def_mu_g}. From this particular structure, we deduce an estimate on the measure of the neighbourhood of singular curves (Lemma~\ref{lemma:measure_Sk}) which yields that $\mu_g$ is $T$-adapted and has no atoms (Corollary~\ref{corol:mug-has_no_atoms-is_adapted-sees_manifolds}).

\begin{proposition}\label{prop:exist}
If $g$ is a $(\cM_0^1,\alpha_g)$-H\"older continuous potential such that $P_*(T,g) - \sup g >s_0 \log 2$ and $\log \Lambda > \sup g - \inf g$, then there exist $\nu \in \cB_w$ and $\tilde{\nu} \in \cB^*_w$ such that $\cL_g \nu = e^{P_*(T,g)} \nu$ and $\cL_g^* \tilde{\nu} = e^{P_*(T,g)} \tilde{\nu}$. In addition, $\nu$ and $\tilde\nu$ take nonnegative values on nonnegative $C^1$ functions on $M$ and are thus nonnegative Radon measures. Finally, $\tilde\nu (\nu) \neq 0$ and $||\nu||_u$ is finite.
\end{proposition}

From the definition of $|\cdot |_w$, we can see that since for $\varphi, f  \in C^1(M)$, $| \varphi f |_w \leqslant |\varphi|_{C^1(M)} |f|_w$, the multiplication by $\varphi$ can be extended to $f \in \cB_w$. Therefore, if $P_*(T,g) - \sup g > s_0 \log 2$ and both SSP conditions are satisfied, a bounded linear map  $\mu_g$ from $C^1(M)$ to  $\mathbb{C}$ can be defined by taking $\nu$ and $\tilde \nu$ from Proposition~\ref{prop:exist} and setting
\begin{equation}\label{eq:def_mu_g}
\mu_g(\varphi)= \frac{ \tilde \nu (\varphi \nu)}{ \tilde \nu (\nu)}\, .
\end{equation}
This quantity is nonnegative for all nonnegative $\varphi \in C^1(M)$ and thus $\mu_g$ can be extended into a nonnegative Radon measure $\mu_g \in (C^0(M))^*$, with  $\mu_g(1)=1$. Clearly, $\mu_g$ is a $T$-invariant probability measure since for every $\varphi \in C^1(M)$ we have
$$
\tilde \nu(\varphi \nu)=e^{-P_*(T,g)}\tilde \nu(\varphi \cL_g(\nu))=e^{-P_*(T,g)} \tilde \nu(\cL_g((\varphi \circ T)\nu))=
\tilde \nu ((\varphi \circ T)\nu )=\tilde \nu(\nu)\mu_g(\varphi \circ T) \, .
$$

\begin{proof}
Let $1$ denote the constant function equal to $1$ on $M$. We will construct $\nu$ out of this seed. By Theorem~\ref{thm:spectral_radius}, recall that $\| \cL_g^n 1 \|_{\cB} \geqslant \| \cL_g^n 1 \|_s \geqslant | \cL_g^n 1 |_w  \geqslant C  e^{nP_*(T,g)}$. Now consider 
\begin{align}\label{eq:def_nu_n}
\nu_n \coloneqq \frac{1}{n} \sum\limits_{k=0}^{n-1}  e^{- kP_*(T,g)} \cL_g^k 1 \, , \quad n\geqslant 1.
\end{align}
By construction, the $\nu_n$ are nonnegative, and thus Radon measures. By Proposition~\ref{prop:upper_bounds_norms}, they satisfy $||\nu_n||_{\cB} \leqslant \overline C$, so using the relative compactness of $\cB$ in $\cB_w$ (\cite[Proposition~6.1]{BD2020MME}), we extract a subsequence $(n_j)$ such that $\lim_j \nu_{n_j} = \nu$ is a nonnegative Radon measure, and the convergence is in $\cB_w$. Since $\cL_g$ is continuous on $\cB_w$, we may write,
\begin{align*}
\cL_g \nu &= \lim\limits_{j \to \infty} \frac{1}{n_j} \sum\limits_{k=0}^{n_j-1}  e^{- kP_*(T,g)} \cL_g^{k+1} 1 \\
&=\lim\limits_{j \to \infty} \frac{e^{P_*(T,g)}}{n_j} \sum\limits_{k=0}^{n_j-1} e^{- kP_*(T,g)} \cL_g^{k} 1 - \frac{1}{n_j} e^{P_*(T,g)}1 + \frac{1}{n_j} e^{(n_j-1)P_*(T,g)} \cL_g^{n_j}1 
= e^{P_*(T,g)}\nu \, ,
\end{align*}
where we used that the second and third terms go to $0$ (in the $\cB$-norm). We thus obtain a nonnegative measure $\nu \in \cB_w$ such that $\cL_g \nu = e^{P_*(T,g)} \nu$.

Although $\nu$ is not a priori an element of $\cB$, it does inherit bounds on the unstable norm from the sequence $\nu_n$. The convergence of $(\nu_{n_j})$ to $\nu$ in $\cB_w$ implies that
\[ \lim\limits_{j \to \infty} \sup\limits_{W\in \cW^s} \sup\limits_{\substack{\psi \in C^\alpha(W) \\ |\psi|_{C^\alpha(W)} \leqslant 1 }} \left( \int_W \nu \psi \, \mathrm{d}m_W - \int_W \nu_{n_j} \psi \, \mathrm{d}m_W \right) = 0 \, .\]
Since $(\nu_n)$ is bounded in $\cB$, there is a constant $\overline{C}>0$ such that $||\nu_{n_j}||_u \leqslant \overline{C}$. It follows that $||\nu||_u \leqslant \overline{C}$, as claimed.

Next, recalling the bound $|\int f \, \mathrm{d}\musrb| \leqslant \hat{C} |f|_w$ from \cite[Proposition 4.2]{BD2020MME}, setting $\mathrm{d}\musrb \in (\cB_w)^*$ to be the functional defined on $C^1(M) \subset \cB_w$ by $\mathrm{d}\musrb(f) = \int f \, \mathrm{d}\musrb$ and extended by density, we define
\begin{align}\label{eq:def_nu_tilde_n}
\tilde \nu_n \coloneqq \frac{1}{n} \sum\limits_{k=0}^{n-1} e^{-kP_*(T,g)} (\cL_g^*)^k(\mathrm{d}\musrb).
\end{align}
Then, we have $|\tilde{\nu}_n(f)| \leqslant C |f|_w$ for all $n$ and all $f \in \cB_w$. So $\tilde \nu_n$ is bounded in $(\cB_w)^* \subset \cB^*$. By compactness of the embedding (\cite[Proposition 6.1]{BD2020MME}), we can find a subsequence $\tilde \nu_{\tilde n_j}$ converging to $\tilde \nu \in \cB^*$. By the argument above, we have $\cL_g^* \tilde \nu = e^{P_*(T,g)} \tilde \nu$.

We next check that $\tilde \nu$, which in principle lies in the dual of $\cB$, is in fact an element of $(\cB_w)^*$. For this, it suffices to find $\tilde C < \infty$ so that for any $f \in \cB$ we have \begin{align}\label{eq:nu_tilde_in_weak_dual}
\tilde{\nu}(f) \leqslant \tilde C |f|_w.
\end{align} 
Now, for $f \in \cB$ and any $\tilde n_j \geqslant 1$, we have
\[ |\tilde \nu(f)| \leqslant |(\tilde \nu_{\tilde n_j} - \tilde \nu)(f)| + |\tilde \nu_{\tilde n_j}(f)| \leqslant |(\tilde \nu_{\tilde n_j} - \tilde \nu)(f)| + |f|_w. \] 
Since $\tilde \nu_{\tilde n_j} \to \tilde \nu$ in $\cB^*$, we conclude $|\tilde \nu(f)| \leqslant |f|_w$ for all $f \in \cB$. Since $\cB$ is dense in $\cB_w$, by \cite[Thm~I.7]{Reed80methods} $\tilde \nu$ extends uniquely to a bounded linear functional on $\cB_w$ satisfying (\ref{eq:nu_tilde_in_weak_dual}). It only remains to prove that $\tilde \nu(\nu) >0$.

Since $\tilde \nu$ is continuous on $\cB_w$, we have on the one hand
\begin{align*}
\tilde{\nu}(\nu) = \lim\limits_{j \to \infty} \tilde \nu (\nu_{n_j}) = \lim\limits_{j \to \infty} \frac{1}{n_j} \sum\limits_{k=0}^{n_j-1} e^{-kP_*(T,g)} \tilde \nu(\cL_g^k 1) = \lim\limits_{j \to \infty} \frac{1}{n_j} \sum\limits_{k=0}^{n_j-1} \tilde{\nu}(1)=\tilde{\nu}(1),
\end{align*}
where we have used that $\tilde \nu$ is an eigenvector of $\cL_g^*$. On the other hand,
\begin{align*}
\tilde{\nu}(1) = \lim\limits_{j \to \infty} \frac{1}{\tilde n_j} \sum\limits_{k=0}^{\tilde n_j-1} e^{-kP_*(T,g)} ((\cL_g^*)^k \mathrm{d}\musrb)(1) = \lim\limits_{j \to \infty} \frac{1}{\tilde n_j} \sum\limits_{k=0}^{\tilde n_j-1} e^{-kP_*(T,g)} \int \cL_g^k 1 \,\mathrm{d}\musrb.
\end{align*}
Next, we disintegrate $\musrb$ as in the proof of \cite[Lemma 4.4]{BD2020MME} into conditional measure $\musrb^{W_\xi}$ on maximal homogeneous stable manifolds $W_\xi \in \cW^s_{\mathbbm{H}}$ and a factor measure $\mathrm{d} \hatmusrb(\xi)$ on the index set $\Xi$ of stable manifolds. Recall that $\musrb^{W_\xi} = |W_\xi|^{-1} \rho_\xi \mathrm{d}m_W$, where $\rho_\xi$ is uniformly log-H\"older continuous so that 
\begin{align}\label{eq:bounds_density}
0 < c_\rho \leqslant \inf\limits_{\xi \in \Xi} \inf\limits_{W_\xi} \rho_\xi \leqslant \sup\limits_{\xi \in \Xi} |\rho_\xi|_{C^\alpha(W_\xi)} \leqslant C_\rho < \infty .
\end{align}
Let $\Xi^{\delta_1}$ denote those $\xi \in \Xi$ such that $|W_\xi| \geqslant \delta_1/3$ and note that $ \hatmusrb(\Xi^{\delta_1}) > 0$. Then, disintegrating as usual, we get by (\ref{eq:seed_grows_exp}) for $k \geqslant n_1$,
\begin{align*}
\int \cL_g^k 1 \, \mathrm{d}\musrb &= \int_{\Xi} \int_{W_\xi} \cL_g^k 1 \rho_\xi |W_\xi|^{-1} \, \mathrm{d}m_{W_\xi} \mathrm{d}\hatmusrb(\xi) \\
&\geqslant \int_{\Xi^{\delta_1}} \int_{W_\xi} \cL_g^k 1 \, \mathrm{d}m_{W_\xi} c_\rho 3 \delta_1^{-1} \, \mathrm{d}\hatmusrb(\xi) \geqslant c_\rho \frac{2c_0}{3} e^{kP_*(T,g)} \hatmusrb(\Xi^{\delta_1}) >0.
\end{align*}
Thus $\tilde \nu(\nu) = \tilde \nu(1) \geqslant c_\rho \frac{2c_0}{3} \hatmusrb(\Xi^{\delta_1}) > 0$ as required.
\end{proof}

\begin{lemma}\label{lemma:measure_Sk}
For any $\gamma >0$ such that $2^{s_0\gamma}< e^{P_*(T,g) - \sup g}$ and any $k \in \mathbbm{Z}$ there exists $C_k>0$ such that 
\begin{align}\label{eq:estimate_mug_neighbourhood_Sk}
\mu_g(\cN_\ve(\cS_k)) \leqslant C_k |\log \varepsilon|^{-\gamma}, \quad \forall \varepsilon>0.
\end{align}
In particular, for any $p > 1/\gamma$ (one can choose $p<1$ for $\gamma >1$), $\eta >0$, and $k \in \mathbbm{Z}$, for $\mu_g$-almost every $x \in M$, there exists $C>0$ such that
\begin{align}\label{eq:Borel_Cantelli_seed}
d(T^n x,\cS_k) \geqslant C e^{-\eta n^p}, \quad \forall n\geqslant 0.
\end{align}
\end{lemma}

\begin{proof}
The proof is the same as the one of \cite[Lemma~7.3]{BD2020MME} where $\mu_*$ should be replaced by $\mu_g$, and the $\nu$ from there by the $\nu$ from here. We sketch the proof: \eqref{eq:estimate_mug_neighbourhood_Sk} is obtained by showing \begin{align}\label{eq:estimate_nu_neighbourhood_Sk}
|\nu(\cN_\varepsilon(\cS_k))| \leqslant C |1_{k,\varepsilon} \nu|_w \leqslant C_k |\log \varepsilon|^{-\gamma} \, ,
\end{align}
where $1_{k,\ve} = \mathbbm{1}_{\cN_\ve(\cS_k)}$, and then by applying \eqref{eq:nu_tilde_in_weak_dual}. Then, \eqref{eq:Borel_Cantelli_seed} is deduced from \eqref{eq:estimate_mug_neighbourhood_Sk} using the Borel--Cantelli Lemma and $\ve$ of the form $e^{- \eta n^p}$.
\end{proof}

\begin{corollary}\label{corol:mug-has_no_atoms-is_adapted-sees_manifolds}
\begin{enumerate}
\item[a)] For any $\gamma>0$ so that $P_*(T,g) - \sup g > \gamma s_0 \log 2$ and any $C^1$ curve $S$ uniformly transverse to the stable cone, there exists $C>0$ such that $\nu(\cN_\ve(S)) \leqslant C |\log \ve|^{-\gamma}$ and $\mu_g(\cN_\ve(S)) \leqslant C |\log \ve|^{-\gamma}$ for all $\ve >0$.
\item[b)] The measures $\nu$ and $\mu_g$ have no atoms, and $\mu_g(W)=0$ for all $W \in \cW^s$ and $W \in \cW^u$.
\item[c)] The measure $\mu_g$ is adapted: $\int |\log d(x,\cS_{\pm 1})| \,\mathrm{d}\mu_g < \infty$.
\item[d)] $\mu_g$-almost every point in $M$ has a stable and unstable manifold of positive length.
\end{enumerate}
\end{corollary}

\begin{proof}
The proof is identical to the one of \cite[Corollary~7.4]{BD2020MME}, replacing $\mu_*$ by $\mu_g$.
\end{proof}

\subsection{$\nu$-Almost Everywhere Positive Length of Unstable Manifolds}\label{sect:positive_length_unstable_manifolds}

In this section, we establish almost everywhere positive length of unstable manifolds in the sense of the measure $\nu$ -- the maximal eigenvector of $\cL_g$ in $\cB_w$, extended into a measure since it is a nonnegative distribution. To do so, we will view elements of $\cB_w$ as \emph{leafwise measure} (Definition~\ref{def:leafwise_measure}). Indeed, in Lemma~\ref{lemma:disintegration_nu}, we make a connection between the disintegration of $\nu$ as a measure, and the family of leafwise measures on the set of stable manifolds $\cW^s$.

\begin{definition}[Leafwise distribution and leafwise measure]\label{def:leafwise_measure}
For $f \in C^1(M)$ and $W \in \cW^s$, the map defined on $C^\alpha(W)$ by
\[ \psi \mapsto \int_W f \psi \,\mathrm{d}m_W, \]
can be viewed as a distribution of order $\alpha$ on $W$. Since $|\int_W f \psi \,\mathrm{d}m_W| \leqslant |f|_w |\psi|_{C^\alpha(W)}$, we can extend the map sending $f \in C^1(M)$ to this distribution of order $\alpha$, to $f \in \cB_w$. We denote this extension by $\int_W f \psi \,\mathrm{d}m_W$ or $\int_W \psi f$, and we call the corresponding family of distributions $(f,W)_{W \in \cW^s}$ the leafwise distribution associated to $f \in \cB_w$. \\
Note that if $\int_W f \psi \, \mathrm{d}m_W \geqslant 0$ for all $\psi \geqslant 0$, then the leafwise distribution on $W$ can be extended into a bounded linear functional on $C^0(W)$, or in other words, a Radon measure. If this holds for all $W \in \cW^s$, the leafwise distribution is called a leafwise measure.
\end{definition}

\begin{lemma}[Almost Everywhere Positive Length of Unstable Manifolds, for $\nu$]\label{lemma:positive_length_unstable_nu}
For $\nu$-almost every $x \in M$ the stable and unstable manifolds have positive length. Moreover, viewing $\nu$ as a leafwise measure, for every $W \in \cW^s$, $\nu$-almost every $x \in W$ has an unstable manifold of positive length.
\end{lemma}

\begin{lemma}\label{lemma:disintegration_nu}
Let $\nu^{W_\xi}$ and $\hat{\nu}$ denote the conditional measures and factor measure obtained by disintegrating $\nu$ on the set of homogeneous stable manifolds $W_\xi \in \cW^s_{\mathbbm{H}}$, $\xi \in \Xi$. Then for any $\psi \in C^\alpha(M)$,
\[ \int_{W_\xi} \psi \,\mathrm{d}\nu^{W_\xi} = \frac{\int_{W_\xi} \psi \rho_\xi \,\nu }{\int_{W_\xi} \rho_\xi \,\nu} \quad \forall \xi \in \Xi, \text{ and } \mathrm{d}\hat{\nu}(\xi) = |W_\xi|^{-1} \left( \int_{W_\xi} \rho_\xi \nu \right) \mathrm{d} \hatmusrb (\xi). \]
Moreover, viewed as a leafwise measure, $\nu(W)>0$ for all $W \in \cW^s$.
\end{lemma}

\begin{proof}
The proof is formally the same the one for \cite[Lemma~6.6]{BD2020MME} replacing $\mu_*$ and their $\nu$ by $\mu_g$ and the $\nu$ from here, as well as their $n_2$ by the one from Corollary~\ref{corol:long_over_all}, and \cite[eq (7.13)]{BD2020MME} by its weighted counterpart: there exists $\overline{C} > 0$ such that for all $W \in \cW^s$,
\begin{align}\label{eq:estimate_nu_on_leaf__lower_bound}
\int_W \nu \geqslant \overline{C}|W|^{(P_*(T,g) - \sup g)\overline{C}_2}.
\end{align}
This estimate holds since, recalling \eqref{eq:def_nu_n}, we have for $\overline{C} = \frac{c_0}{2C}\delta_1^{1-(P_*(T,g)-\inf g)\overline{C}_2}$,
\begin{align*}
\int_W \nu &= \lim\limits_{n_j} \frac{1}{n_j} \sum\limits_{k=0}^{n_j-1} e^{-kP_*(T,g)} \int_W \cL_g^k 1 \,\mathrm{d}m_W \\
&\geqslant \lim\limits_{n_j} \frac{1}{n_j} \sum\limits_{k=n_2}^{n_j-1} e^{-kP_*(T,g)} \sum\limits_{W_i \in \cG_{n_2}^{\delta_1}(W)} \int_{W_i} e^{S_{n_2} g} \cL_g^{k-n_2} 1 \,\mathrm{d}m_{W_i} \\
&\geqslant \lim\limits_{n_j} \frac{1}{n_j} \sum\limits_{k=n_2}^{n_j-1} e^{-kP_*(T,g)} e^{n_2 \inf g} \frac{C\delta_1}{2}c_0 e^{P_*(T,g)(k-n_2)} \\
&\geqslant \frac{C\delta_1}{2}c_0 e^{-n_2 (P_*(T,g)-\inf g) } \geqslant \overline{C} |W|^{(P_*(T,g) - \sup g)\overline{C}_2},
\end{align*}
where we used Theorem~\ref{thm:spectral_radius} for the second inequality.
\end{proof}

\begin{proof}[Proof of Lemma~\ref{lemma:positive_length_unstable_nu}.]
The statement about stable manifolds of positive length follows from the characterization of $\hat \nu$ in Lemma~\ref{lemma:disintegration_nu}, since the set of points with stable manifolds of zero length has zero $\hatmusrb$-measure \cite{chernov2006chaotic}.

The rest of the proof follows closely the one of the analogous result \cite[Lemma~7.6]{BD2020MME} (corresponding to $g \equiv 0)$, but with more general computations.

We fix $W \in \cW^s$ and prove the statement about $\nu$ as a leafwise measure. This will imply the statement regarding unstable manifolds for the measure $\nu$ by Lemma~\ref{lemma:disintegration_nu}.

Fix $\varepsilon>0$ and $\hat \Lambda \in (1,\Lambda)$, and define $O = \cup_{n\geqslant 1} O_n$, where
\[ O_n \coloneqq \{ x \in W \mid n= \min\{ j\geqslant 1 \mid d_u(T^{-j}x, \cS_1) < \varepsilon C_e \hat \Lambda^{-j} \} \}, \]
and $d_u$ denotes distance restricted to the unstable cone. By \cite[Lemma 4.67]{chernov2006chaotic}, any $x \in W\smallsetminus O$ has an unstable manifold of length at least $2\varepsilon$. We now estimate $\nu(O) = \sum_{n\geqslant 1} \nu(O_n)$, where equality holds since the $O_n$ are disjoint. Since each $O_n$ is a finite union of open subcurves of $W$, we have
\begin{align}\label{eq:O_n}
\int_W \mathbbm{1}_{O_n} \nu = \lim\limits_{j \to \infty} \int_W \mathbbm{1}_{O_n} \nu_{n_j} = \lim\limits_{j \to \infty} \frac{1}{n_j} \sum\limits_{k=0}^{n_j-1} e^{-kP_*(T,g)} \int_W  \mathbbm{1}_{O_n} \cL_g^k 1 \,\mathrm{d}m_W.
\end{align}
We give estimates in two cases.

\smallskip
\noindent
{\em Case I: $k < n$.}
Write $\int_{W\cap O_n} \cL_g^k 1 \,\mathrm{d}m_W = \sum_{W_i \in \cG_k^{\delta_0}(W)} \int_{W_i \cap T^{-k}O_n} e^{S_k g} \,\mathrm{d}m_{W_i}$.

If $x \in T^{-k}O_n$, then $y = T^{-n+k}x$ satisfies $d_u(y, \cS_1) < \ve C_e \hLambda^{-n}$ and thus we have $d_u(Ty, \cS_{-1}) \leqslant C \ve^{1/2} \hLambda^{-n/2}$.  Due to the uniform transversality of stable and unstable cones, as well as the fact that elements of $\cS_{-1}$ are uniformly transverse to the stable cone, we have $d_s(Ty, \cS_{-1}) \leqslant C \ve^{1/2} \hLambda^{-n/2}$ as well, with possibly a larger constant $C$. 

Let $r^s_{-j}(x)$ denote the distance from $T^{-j}x$ to the nearest endpoint of $W^s(T^{-j}x)$, where $W^s(T^{-j}x)$ is the maximal local stable manifold containing $T^{-j}x$.  From the above analysis, we see that $W_i \cap T^{-k}O_n \subseteq \{ x \in W_i : r^s_{-n+k+1}(x) \leqslant C \ve^{1/2} \hLambda^{-n/2} \}$. The time reversal of the growth lemma \cite[Thm~5.52]{chernov2006chaotic} gives $m_{W_i}(r^s_{-n+k+1}(x) \leqslant C \ve^{1/2} \hLambda^{-n/2} ) \leqslant C' \ve^{1/2} \hLambda^{-n/2}$ for a constant $C'$ that is uniform in $n$ and $k$.  Thus, using Proposition~\ref{prop:almost_exponential_growth}, we find
\[
\int_{W \cap O_n} \cL_g^k 1 \, \mathrm{d}m_W \leqslant C' \ve^{1/2} \hLambda^{-n/2} \sum\limits_{W_i \in \cG_k^{\delta_0}(W)} |e^{S_k g}|_{C^0(W_i)} \leqslant C e^{k P_*(T,g)} \ve^{1/2} \hLambda^{-n/2}\,  .
\]

\smallskip
\noindent
{\em Case II: $k \geqslant n$.}
Using the same observation as in Case I, if $x \in T^{-n+1}O_n$, then $x$ satisfies $d_s(x, \cS_{-1}) \leqslant C \ve^{1/2} \hLambda^{-n/2}$.  We change variables to estimate the integral precisely at time $-n+1$, and then use Propositions~\ref{prop:upper_bounds_norms} and~\ref{prop:almost_exponential_growth}, and Lemma~\ref{lemma:norm_C0_exp_birkhoff},
\begin{align*}
&\int_{W \cap O_n} \cL^k1 \, \mathrm{d}m_W  = \sum_{W_i \in \cG_{n-1}^{\delta_0}(W)} \int_{W_i \cap T^{-n+1}O_n} e^{S_{n-1}g} \cL_g^{k-n+1} 1 \, \mathrm{d}m_{W_i}  \\
&\quad \leqslant \sum_{W_i \in \cG_{n-1}^{\delta_0}(W)} \int_{W_i \cap ( r^s_1 \leqslant C \varepsilon^{1/2} \hLambda^{-n/2})} e^{S_{n-1}g} \cL_g^{k-n+1} 1 \, \mathrm{d}m_{W_i}  \\
&\quad\leqslant \sum_{W_i \in \cG_{n-1}^{\delta_0}(W)} |\log |W_i \cap ( r^s_1 \leqslant C \varepsilon^{1/2} \hLambda^{-n/2})||^{-\gamma} |e^{S_{n-1} g}|_{C^\beta(W_i)}  \| \cL_g^{k-n+1} 1 \|_s \\
& \quad\leqslant \sum_{W_i \in \cG_{n-1}^{\delta_0}(W)} |\log (C \ve^{1/2} \hLambda^{-n/2})|^{-\gamma} C |e^{S_{n-1}g}|_{C^0(W_i)} e^{(k-n+1)P_*(T,g)} \\
&\quad \leqslant |\log (C \ve^{1/2} \hLambda^{-n/2})|^{-\gamma} C e^{k P_*(T,g)} \, .
\end{align*}
Using the estimates of Cases I and II in \eqref{eq:O_n} and using the weaker bound, we see that,
\[
\int_W \mathbbm{1}_{O_n} \, \nu_{n_j} \leqslant C |\log (C \ve^{1/2} \hLambda^{-n/2})|^{-\gamma}\, .
\]
Summing over $n$, we have, $\int_W 1_O \, \nu_{n_j} \leqslant C' |\log \ve|^{1-\gamma}$, uniformly in $j$.  Since $\nu_{n_j}$ converges to $\nu$ in the weak norm, this bound carries over to $\nu$.  Since $\ve>0$ was arbitrary and $\gamma>1$, this implies $\nu(O) = 0$, completing the proof of the lemma.
\end{proof}

\subsection{Absolute Continuity of $\mu_g$ -- Full Support}\label{sect:absolute_continuity}

In this subsection, we will assume that $\gamma > 1$, which is possible since $P_*(T,g) - \sup g > s_0 \log 2$. In the next subsection, we prove that $\mu_g$ is Bernoulli. This proof relies on showing first that $\mu_g$ is K-mixing. As a first step, we will prove that $\mu_g$ is ergodic, using a Hopf-type argument. This will require the absolute continuity of the stable and the unstable foliations for $\mu_g$, which will be deduce from SSP.2 and the following absolute continuity for $\nu$:

\begin{proposition}\label{prop:abs_c0_holonomy_wrt_nu}
Let $R$ be a Cantor rectangle. Fix $W^0 \in \cW^s(R)$ and for $W \in \cW^s(R)$, let $\Theta_W$ denote the holonomy map from $W^0 \cap R$ to $W \cap R$ along unstable manifolds in $\cW^u(R)$. Then for any $(\cM_0^1,\alpha_g)$-H\"older potential with $P_*(T,g)-\sup g > s_0 \log 2$ and having SSP.1, $\Theta_W$ is absolutely continuous with respect to the leafwise measure $\nu$.
\end{proposition}

\begin{proof}
Since by Lemma~\ref{lemma:positive_length_unstable_nu} unstable manifolds comprise a set of full $\nu$-measure, it suffices to fix a set $E \subset W^0 \cap R$ with $\nu$-measure zero, and prove that the $\nu$-measure of $\Theta_W(E) \subset W$ is also zero.

Here again, the proof follows closely the one of the analogous result \cite[Proposition~7.8]{BD2020MME} (corresponding to $g \equiv 0)$, but with more general computations.

Since $\nu$ is a regular measure on $W^0$, for $\ve > 0$, there exists an open set $O_\ve \subset W^0$, $O_\ve \supset E$, such that $\nu(O_\ve) \leqslant \ve$. Indeed, since $W^0$ is compact, we may choose $O_\ve$ to be a finite union of intervals. Let $\psi_\ve$ be a smooth function which is 1 on $O_\ve$ and 0 outside of an $\ve$-neighbourhood of $O_\ve$.  We may choose $\psi_\ve$ so that $\int_{W^0} \psi_\varepsilon \, \nu < 2 \varepsilon$.

Using \eqref{eq:Calpha_norm_observable_Tn}, we choose $n = n(\ve)$ such that $| \psi_\ve \circ T^n |_{C^1(T^{-n}W^0)} \leqslant 1$ and $\Lambda^{-n} \leqslant \ve$. Following the procedure described in the proof of the estimate on the unstable norm in Proposition~\ref{prop:upper_bounds_norms}, we subdivide $T^{-n}W^0$ and $T^{-n}W$ into matched pieces $U^0_j$, $U_j$ and unmatched pieces $V^0_i$, $V_i$.  With this construction, none of the unmatched pieces $T^nV^0_i$ intersect an unstable manifold
in $\cW^u(R)$ since unstable manifolds are not cut under $T^{-n}$.

Indeed, on matched pieces, we may choose a foliation 
$\Gamma_j = \{ \gamma_x \}_{x \in U^0_j}$ such that:

i) $T^n\Gamma_j$ contains all unstable manifolds in $\cW^u(R)$ that intersect $T^nU^0_j$;

ii) between unstable manifolds in $\Gamma_j \cap T^{-n}(\cW^u(R))$, we interpolate via unstable curves;

iii) the resulting holonomy $\Theta_j$ from $T^nU^0_j$ to $T^nU_j$ has uniformly bounded Jacobian\footnote{Indeed,
  \cite{BDL2018ExpDecay} shows the Jacobian is H\"older continuous, but we shall not need this here.} with
  respect to arc-length, with bound depending on the
  unstable diameter of $D(R)$, by \cite[Lemmas~6.6, 6.8]{BDL2018ExpDecay};  

iv)  pushing forward $\Gamma_j$ to $T^n\Gamma_j$ in $D(R)$, we interpolate in the gaps using unstable curves;
  call $\bGamma$ the resulting foliation of $D(R)$;

v)  the associated holonomy map $\bTheta_W$ extends $\Theta_W$ and has uniformly bounded Jacobian,
  again by \cite[Lemmas~6.6 and 6.8]{BDL2018ExpDecay}.

Using the map $\bTheta_W$, we define $\tpsi_\ve = \psi_\ve \circ \bTheta_W^{-1}$, and note that
$| \tpsi_\ve|_{C^1(W)} \leqslant C |\psi_\ve |_{C^1(W^0)}$, where we write $C^1(W)$ for the set of Lipschitz functions on
$W$, i.e., $C^\alpha$ with $\alpha =1$.

Next, we modify $\psi_\ve$ and $\tpsi_\ve$ as follows:  We set them
equal to $0$ on the images of unmatched pieces, $T^nV^0_i$ and $T^nV_i$, respectively.  Since these curves do not intersect
unstable manifolds in $\cW^u(R)$, we still have $\psi_\ve = 1$ on $E$ and $\tpsi_\ve = 1$ on $\Theta_W(E)$.  Moreover, the
set of points on which $\psi_\ve > 0$ (resp. $\tpsi_\ve > 0$) is a finite union of open intervals that cover $E$ (resp. $\Theta_W(E)$).

Since $\int_{W^0} \psi_\varepsilon \, \nu < 2\varepsilon$, in order to estimate $\int_{W} \tpsi_\varepsilon \, \nu$, we estimate the following difference, using matched pieces 
\begin{equation}
\label{eq:match split}
\begin{split}
\int_{W^0}  \psi_\ve \, \nu &- \int_W \tpsi_\ve \, \nu  = e^{-n P_*(T,g)} \left( \int_{W^0} \psi_\ve \, \cL^n \nu - \int_W \tpsi_\ve \, \cL^n \nu \right)  \\
& = e^{-n P_*(T,g)} \sum_j \int_{U^0_j} \psi_\ve \circ T^n \, e^{S_n g} \, \nu - \int_{U_j} \phi_j \, \nu + \int_{U_j} (\phi_j - \tpsi_\ve \circ T^n \, e^{S_n g}) \, \nu \,  , 
\end{split}
\end{equation}
where $\phi_j = (\psi_\ve \circ T^n \, e^{S_n g}) \circ G_{U^0_j} \circ G_{U_j}^{-1}$, and $G_{U^0_j}$ and $G_{U_j}$ represent the functions defining $U^0_j$ and $U_j$, respectively, defined as in \eqref{eq:match}.  
Next, since $d(\psi_\ve \circ T^n \, e^{S_n g}, \phi_j) = 0$ by construction, and using \eqref{eq:closeness_matched_pieces} and the assumption that $\Lambda^{-n} \leqslant \ve$, we have by \eqref{eq:bound_diff_int_matched_pieces},
\begin{equation}
\label{eq:match one}
e^{-n P_*(T,g)} \left| \sum_j \int_{U^0_j} \psi_\ve \circ T^n \, \nu - \int_{U_j} \phi_j \, \nu \right| \leqslant C 
|\log \ve|^{-\varsigma} \| \nu \|_u \, .
\end{equation}

It remains to estimate the last term in \eqref{eq:match split}. This we do using the weak norm,
\begin{equation}
\label{eq:match two}
\int_{U_j} (\phi_j - \tpsi_\ve \circ T^n \, e^{S_n g}) \, \nu \leqslant |\phi_j - \tpsi_\ve \circ T^n \, e^{S_n g}|_{C^\alpha(U_j)} \, |\nu|_w \, .
\end{equation}
By \eqref{eq:diff_phi_psi_matched_pieces}, we have
\[
|\phi_j - \tpsi_\ve \circ T^n \, e^{S_n g}|_{C^\alpha(U_j)} \leqslant C |(\psi_\ve \circ T^n \, e^{S_n g}) \circ G_{U^0_j} - (\tpsi_\ve \circ T^n \, e^{S_n g}) \circ G_{U_j} |_{C^\alpha(I_j)} \, ,
\]
where $I_j$ is the common $r$-interval on which $G_{U^0_j}$ an $G_{U_j}$ are defined.

Fix $r \in I_j$, and let $x = G_{U^0_j}(r) \in U_j$ and $\bx = G_{U_j}(r)$.  Since $U^0_j$ and $U_j$ are matched, there
exist $y \in U^0_j$ and an unstable curve $\gamma_y \in \Gamma_j$ such that $\gamma_y \cap U_j = \bx$.  By definition
of $\tpsi_\ve$, we have $\tpsi_\ve \circ T^n(\bx) = \psi_\ve \circ T^n(y)$.  Thus,
\[
\begin{split}
|(\psi_\ve \circ T^n \, e^{S_n g}) &\circ G_{U^0_j}(r) - (\tpsi_\ve \circ T^n \, e^{S_n g}) \circ G_{U_j} (r)|\\
&\leqslant |\psi_\ve \circ T^n(x) - \tpsi_\ve \circ T^n(\bx)| | e^{S_n g (x)}| + |\tpsi_\varepsilon \circ T^n (\bx)| |e^{S_n g (x)} - e^{S_n g (\bx)}| \\
& \leqslant (|\psi_\ve \circ T^n (x) - \psi_\ve \circ T^n(y)| + |\psi_\ve \circ T^n(y) - \tpsi_\ve \circ T^n(\bx)|)e^{n \sup g} + |e^{S_n g (x)} - e^{S_n g (\bx)}| \\
& \leqslant \left( |\psi_\ve \circ T^n|_{C^1(U^0_j)} d(x,y) + |g|_{C^{\alpha_g}} \frac{\Lambda^{\alpha_g}}{\Lambda^{\alpha_g} - 1} (C \ve)^{\alpha_g} \right) e^{n\sup g} \\
& \leqslant (C \Lambda^{-n} + C\varepsilon^{\alpha_g})e^{n\sup g}  \leqslant C (\ve + \ve^{\alpha_g})e^{n\sup g} \, ,
\end{split}
\] 
where we have used the fact that $d(x,y) \leqslant C\Lambda^{-n}$ due to the uniform transversality of stable and unstable curves. We also used the fact that, by definition, the vertical segment $\gamma_x$ connecting $x$ to $\bx$ is such that $|T^n \gamma_x|< C \ve$. Since each $T^i \gamma_x$ lies in the extended unstable cone, for all $0 \leqslant i \leqslant n$, we get that $d(T^i(x), T^i(\bx)) \leqslant C \Lambda^{-(n-i)} \ve$, hence the bound 
\begin{align*}
|e^{S_n g (x)} - e^{S_n g (\bx)}| &\leqslant |e^{S_n g (x)}| \cdot | 1 - e^{S_n g(\bx) - S_n g (x)}| \leqslant 2 e^{n \sup g} |S_n g(\bx) - S_n g (x)| \\
&\leqslant \frac{\Lambda^{\alpha_g}}{\Lambda^{\alpha_g} - 1} (C \ve)^{\bar\alpha} |g|_{C^{\alpha_g}} e^{n\sup g}
\end{align*}
where we used that $|1-e^x| \leqslant 2|x|$ when $x$ is near $0$.

Now given $r, s \in I_j$, we have on the one hand,
{\small
\begin{align}\label{eq:holder_const_part1}
\begin{split}
|(\psi_\ve \circ T^n \, e^{S_n g}) \circ G_{U^0_j}(r) &- (\tpsi_\ve \circ T^n \, e^{S_n g}) \circ G_{U_j} (r) \\
&- (\psi_\ve \circ T^n \, e^{S_n g}) \circ G_{U^0_j}(s) + (\tpsi_\ve \circ T^n \, e^{S_n g}) \circ G_{U_j} (s)| 
 \leqslant 2C\ve^{\bar \alpha} e^{n\sup g} \, ,
\end{split}
\end{align}
}
\hspace{-1ex}while on the other hand,
{\small
\begin{align}\label{eq:holder_const_part2}
\begin{split}
&|(\psi_\ve \circ T^n \, e^{S_n g}) \circ G_{U^0_j}(r) - (\tpsi_\ve \circ T^n \, e^{S_n g}) \circ G_{U_j} (r) -
(\psi_\ve \circ T^n \, e^{S_n g}) \circ G_{U^0_j}(s) + (\tpsi_\ve \circ T^n \, e^{S_n g}) \circ G_{U_j} (s)| \\
&= |(\psi_\ve \circ T^n \, e^{S_n g}) \circ G_{U^0_j}(r) -
(\psi_\ve \circ T^n \, e^{S_n g}) \circ G_{U^0_j}(s) -( (\tpsi_\ve \circ T^n \, e^{S_n g}) \circ G_{U_j} (r) - (\tpsi_\ve \circ T^n \, e^{S_n g}) \circ G_{U_j} (s))| \\
&\leqslant |\psi_\ve|_{C^0(W^0)} |e^{S_n g}|_{C^{\alpha_g}} d(G_{U^0_j}(r),G_{U^0_j}(s))^{\alpha_g} + |\psi_\ve \circ T^n |_{C^1(W^0)} d(G_{U^0_j}(r),G_{U^0_j}(s)) |e^{S_n g}|_{C^0} \\
&\quad\quad + |\tpsi_\ve|_{C^0(W)} |e^{S_n g}|_{C^{\alpha_g}} d(G_{U^0_j}(r),G_{U^0_j}(s))^{\alpha_g} + |\tpsi_\ve \circ T^n |_{C^1(W)} d(G_{U^0_j}(r),G_{U^0_j}(s)) |e^{S_n g}|_{C^0} \\
&\leqslant (C |r-s| + C' |r-s|^{\alpha_g}) e^{n \sup g} \leqslant C |r-s|^{\alpha_g} e^{n \sup g} \, ,
\end{split}
\end{align}
}
\hspace{-1ex}where we have used Lemma~\ref{lemma:norm_C0_exp_birkhoff} and the fact that $G_{U^0_j}^{-1}$ and $G_{U_j}^{-1}$ have bounded derivatives since the
stable cone is bounded away from the vertical.

The difference between evaluation at $r$ and $s$, divided by $|r-s|^\alpha$, is bounded by the minimum of \eqref{eq:holder_const_part1} and \eqref{eq:holder_const_part2}, both divided by $|r-s|^\alpha$. This is greatest when the two are equal, i.e.,
when $|r-s| = C \ve$.  Thus $H^\alpha((\psi_\ve \circ T^n \, e^{S_n g}) \circ G_{U^0_j} - (\tpsi_\ve \circ T^n \, e^{S_n g}) \circ G_{U_j}) \leqslant C \ve^{\alpha_g -\alpha} e^{n \sup g}$,
and so
$
|\phi_j - \tpsi_\ve \circ T^n \, e^{S_n g}|_{C^\alpha(U_j)} \leqslant C \ve^{\alpha_g -\alpha} e^{n \sup g}
$.
Putting this estimate together with \eqref{eq:match one} and \eqref{eq:match two} in \eqref{eq:match split}, we conclude,
\begin{equation}\label{eq:weak small2}
\left| \int_{W^0}  \psi_\ve \, \nu - \int_W \tpsi_\ve \, \nu \right| \leqslant C |\log \ve|^{-\varsigma} \| \nu \|_u + C \ve^{\alpha_g-\alpha} |\nu|_w e^{-n(P_*(T,g)- \sup g)} \, .
\end{equation}
Now since $\int_{W^0} \psi_\ve \, \nu \leqslant 2\ve$, we have
\begin{equation}
\label{eq:weak small}
\int_W \tpsi_\ve \, \nu \leqslant C' |\log \ve|^{-\varsigma}\,  ,
\end{equation}
where $C'$ depends on $\nu$.  Since $\tpsi_\ve = 1$ on $\Theta_W(E)$ and $\tpsi_\ve > 0$ on an open set containing 
$\Theta_W(E)$ for every $\ve > 0$, we have $\nu(\Theta_W(E)) = 0$, as required.
\end{proof}

\begin{corollary}[Absolute Continuity of $\mu_g$ with Respect to Unstable Foliations]\label{corol:abs_c0_mug_unst_fol}
Let $R$ be a Cantor rectangle with $\mu_g(R)>0$. Fix $W^0 \in \cW^s(R)$ and for $W \in \cW^s(R)$, let $\Theta_W$ denote the holonomy map from $W^0\cap R$ to $W\cap R$ along unstable manifolds in $\cW^u(R)$. Then $\Theta_W$ is absolutely continuous with respect to the measure $\mu_g$.
\end{corollary}

In order to deduce the corollary from the Proposition~\ref{prop:abs_c0_holonomy_wrt_nu}, we introduce the set $M^{\mathrm{reg}}$, as in \cite{BD2020MME}, of regular points and a countable cover of this set by Cantor rectangles. The set $M^{\mathrm{reg}}$ is defined by
\begin{align*}
M^{\mathrm{reg}} = \{ x \in M \mid d(x, \partial W^s(x)) > 0 \, , \,\, d(x, \partial W^u(x)) > 0 \}.
\end{align*}
At each $x \in M^{\mathrm{reg}}$, we can apply \cite[Prop~7.81]{chernov2006chaotic} and construct a closed locally maximal Cantor rectangle $R_x$ containing $x$, which is the direct product of local stable and unstable manifolds. Furthermore, by trimming the sides, we may arrange it so that $\tfrac{1}{2} \mathrm{diam}^s(R_x) \leqslant \diam^u(R_x) \leqslant 2\,\diam^s(R_x)$.

\begin{lemma}[Countable Cover of $M^{\mathrm{reg}}$ by Cantor Rectangle]\label{lemma:countable_cover_with_Cantor_rectangles}
There exists a countable set $\{x_j\}_{j \in \mathbbm{N}} \subset M^{\mathrm{reg}}$ such that $\cup_j R_{x_j} = M^{\mathrm{reg}}$ and each $R_j \coloneqq R_{x_j}$ satisfies \eqref{eq:fat_Cantor_rectangle}.
\end{lemma}
\begin{proof}
This lemma is exactly the content of \cite[Lemma~7.10]{BD2020MME}.
\end{proof}

Let $\{R_j \mid j \in \mathbbm{N} \}$ be the family of Cantor rectangles constructed in Lemma~\ref{lemma:countable_cover_with_Cantor_rectangles}, discarding the ones with zero $\mu_g$-measure. Then $\mu_g (\cup_j R_j) = \mu_g(M^{\rm reg}) =1$, by Corollary~\ref{corol:mug-has_no_atoms-is_adapted-sees_manifolds}(d). In the rest of the paper, we shall work with this countable collection of rectangles.

Given a Cantor rectangle $R$, define $\cW^s(R)$ to be the set of stable manifolds that completely cross $D(R)$, and similarly for $\cW^u(R)$. 

\begin{proof}[Proof of Corollary~\ref{corol:abs_c0_mug_unst_fol}]
In order to prove absolute continuity of the unstable foliation with respect to $\mu_g$, we will show that the conditional measures $\mu_g^W$ of $\mu_g$ are equivalent to $\nu$ on $\mu_g$-almost every $W \in \cW^s(R)$. 

The proof here follows closely the one of \cite[Corollary~7.9]{BD2020MME} (corresponding to $g \equiv 0)$, but with more general computations.

Fix a  Cantor rectangle $R$ satisfying \eqref{eq:fat_Cantor_rectangle} with $\mu_g(R)>0$, and $W^0$ as in the statement of Corollary~\ref{corol:abs_c0_mug_unst_fol}. Let $E \subset W^0 \cap R$ satisfy $\nu(E) = 0$, for the leafwise measure $\nu$.

For any $W \in \cW^s(R)$, we have the holonomy map $\Theta_W: W^0 \cap R \to W \cap R$ as in the proof of Proposition~\ref{prop:abs_c0_holonomy_wrt_nu}. For $\ve > 0$, we approximate $E$, choose $n$ and construct a foliation $\bGamma$ of the solid rectangle $D(R)$ as before. Define $\psi_\ve$ and use the foliation $\bGamma$ to define $\tpsi_\ve$ on $D(R)$.  We have $\tpsi_\ve = 1$ on $\bar E = \cup_{x \in E} \bar\gamma_x$, where $\bar\gamma_x$ is the element of $\bGamma$ containing $x$. We extend $\tpsi_\ve$ to $M$ by setting it equal to $0$ on $M \setminus D(R)$.

It follows from the proof of Proposition~\ref{prop:abs_c0_holonomy_wrt_nu}, in particular \eqref{eq:weak small}, that $\tpsi_\ve \nu \in \cB_w$, and
$
|\tpsi_\ve \nu|_w \leqslant C' |\log \ve|^{-\varsigma} 
$.
Now,
\begin{equation}
\label{eq:limit mu}
\begin{split}
\tilde\nu (\nu) \, \mu_g(\tpsi_\ve) & = \tilde\nu(\tpsi_\ve\nu) = \lim_{j \to \infty} \frac{1}{n_j} \sum_{k=0}^{n_j-1} e^{-k P_*(T,g)} (\cL_g^*)^k \, \mathrm{d}\musrb(\tpsi_\ve \nu) \\
& = \lim_{j \to \infty} \frac{1}{n_j} \sum_{k=0}^{n_j-1} e^{-k P_*(T,g)} \musrb(\cL_g^k(\tpsi_\ve \nu)) \, .
\end{split}
\end{equation}
For each $k$, using the disintegration of $\musrb$ as in the proof of Lemma~\ref{lemma:disintegration_nu} with the same notation as there, and \eqref{eq:bounds_density}, we estimate, 
\begin{align*}
\musrb(\cL_g^k(\tpsi_\ve \nu)) & = \int_{\Xi} \int_{W_\xi}  \cL_g^k(\tpsi_\ve \nu) \, \rho_\xi \, \mathrm{d}m_{W_\xi} \, |W_\xi|^{-1} \, \mathrm{d}{\hatmusrb}(\xi) \\
& \leqslant C \int_{\Xi} |\cL_g^k(\tpsi_\ve \nu) |_w \, |W_\xi|^{-1} \, \mathrm{d}{\hatmusrb}(\xi)
\leqslant C_\rho e^{k P_*(T,g)} |\tpsi_\ve \nu|_w \leqslant C_\rho e^{k P_*(T,g)} |\log \ve|^{-\varsigma} \, ,
\end{align*}
where we have used \eqref{eq:weak ly poly} in the last line, as well as the $\hatmusrb$-integrability of $|W_\xi|^{-1}$ from \cite[Exercise~7.22]{chernov2006chaotic}. Thus $\mu_g(\tpsi_\ve) \leqslant C |\log \ve|^{-\varsigma}$, for each $\ve >0$, so that $\mu_g(\bar E) = 0$.  

Disintegrating $\mu_g$ into conditional measures $\mu_g^{W_\xi}$ on $W_\xi \in \cW^s$ and a factor measure $d\hat\mu_g(\xi)$ on the index set $\Xi_R$ of stable manifolds in $\cW^s(R)$, it follows that $\mu_g^{W_\xi}(\bar E) = 0$ for $\hat\mu_g$-almost every $\xi \in \Xi_R$. Since $E$ was arbitrary, the conditional measures of $\mu_g$ on $\cW^s(R)$ are absolutely continuous with respect to the leafwise measure $\nu$.

\medskip
To show that in fact $\mu_g^W$ is equivalent to $\nu$, suppose now that $E \subset W^0$ has $\nu(E)>0$. For any $\ve >0$ such that $C'|\log \ve|^{-\varsigma} < \nu(E)/2$, where $C'$ is from \eqref{eq:weak small}, choose $\psi_\ve \in C^1(W^0)$ such that $\nu(|\psi_\ve - 1_E|) < \ve$, where $1_E$ is the indicator function of the set $E$. 
As above, we extend $\psi_\ve$ to a function $\tpsi_\ve$ on $D(R)$ via the foliation $\bGamma$, and then to $M$ by setting $\tpsi_\ve = 0$ on $M \setminus D(R)$.

We have $\tpsi_\ve \nu \in \cB_w$ and by \eqref{eq:weak small2}
\begin{equation}
\label{eq:lower psi}
\nu(\tpsi_\ve \, 1_W) \geqslant \nu(\psi_\ve \, 1_{W^0}) - C'|\log \ve|^{-\varsigma} ,\qquad \mbox{for all } W \in \cW^s(R) \, .
\end{equation}

Now following \eqref{eq:limit mu} and disintegrating $\musrb$ as usual, we obtain,
\begin{equation}
\label{eq:counting}
\begin{split}
\mu_g(\tpsi_\ve) 
& = \lim_n \frac 1n \sum_{k=0}^{n-1} e^{-k P_*(T,g)} \int_{\Xi} \int_{W_\xi} \cL_g^k(\tpsi_\ve \nu) \, \rho_\xi \, \mathrm{d}m_{W_\xi} \, \mathrm{d}\hatmusrb(\xi) \\
& = \lim_n \frac 1n \sum_{k=0}^{n-1} e^{-k P_*(T,g)} \int_\Xi \left( \sum_{W_{\xi, i} \in \cG_k^{\delta_1}(W_\xi)}  \int_{W_{\xi, i}} \tpsi_\ve \, \rho_\xi \circ T^k \, e^{S_k g} \, \nu \right) \, \mathrm{d}\hatmusrb(\xi) \,  .
\end{split}
\end{equation}
To estimate this last expression, we estimate the thermodynamic sum over the curves $W_{\xi, i}$ which properly cross the rectangle $R$.

By SSP.2 and the choice of $\delta_1$ in \eqref{eq:delta1}, there exists $k_0$, depending only on the minimum length of $W \in \cW^s(R)$, such that 

\[ \sum\limits_{W_i \in L_k^{\delta_1}(W_\xi)} |e^{S_k g}|_{C^0(W_i)} \geqslant \frac 13 \sum\limits_{W_i \in \cG_k^{\delta_1}(W_\xi)} |e^{S_k g}|_{C^0(W_i)} \, , \quad \text{for all} \, k \geqslant k_0. \]

By choice of our covering $\{ R_i \}$ from Lemma~\ref{lemma:countable_cover_with_Cantor_rectangles}, all $W_{\xi, j} \in L_k^{\delta_1}(W_\xi)$ properly cross one of finitely many $R_i$.  By the topological mixing property of $T$, there exists $n_0$, depending only on the length scale $\delta_1$, such that some smooth component of $T^{-n_0}W_{\xi, j}$ properly crosses $R$.  Thus, letting $\cC_k(W_\xi)$ denote those $W_{\xi, i} \in \cG_k^{\delta_1}(W_\xi)$ which properly cross $R$, we have
\begin{align*}
&\sum\limits_{W_i \in \cC_{k+n_0}(W_\xi)} |e^{S_k g}|_{C^0(W_i)} \geqslant \sum\limits_{W_{\xi,i} \in L^{\delta_1}_k(W_\xi) } \sum\limits_{\tilde W \subset \cG_{n_0}^{\delta_1}(W_{\xi,i})\cap \cC_{k+n_0}(W_\xi) } e^{n_0 \inf g} |e^{S_k g}|_{C^0(W_{\xi,i})} \\
&\qquad \geqslant e^{n_0 \inf g} \!\!\!\!\!\! \sum\limits_{W_{\xi,i} \in L^{\delta_1}_k(W_\xi) } \!\!\!\!\!\! |e^{S_k g}|_{C^0(W_{\xi,i})} 
\geqslant \frac 13 e^{n_0 \inf g} \!\!\!\! \sum\limits_{W_{\xi,i} \in \cG^{\delta_1}_k(W_\xi) } \!\!\!\! |e^{S_k g}|_{C^0(W_{\xi,i})}   \geqslant  \frac 13 c \,  e^{n_0 \inf g} \, e^{k P_*(T,g)} \, ,
\end{align*}
for all $k \geqslant k_0$, where $c>0$ depends on $c_0$ from Proposition~\ref{prop:GnW_geq_M0n} as well as the minimum length of $W \in \cW^s(R)$.

Using this lower bound on the sum together with \eqref{eq:lower psi} yields,
\[
\mu_g(\tpsi_\ve) \geqslant \tfrac 13 c e^{-n_0 P_*(T,g)} \big( \nu(\psi_\ve) - C' |\log \ve|^{-\varsigma}\big)
\geqslant C'' \big( \nu(E) - |\log \ve|^{-\varsigma}\big)\,  .
\]
Taking $\ve \to 0$, we have 
\begin{align}\label{eq:mug_geq_nu}
\mu_g(\bar E) \geqslant C'' \nu(E) \, ,
\end{align}
and so $\mu_g^W(\bar E) > 0$ for almost every $W \in \cW^s(R)$.
\end{proof}

\begin{proposition}[Full Support]\label{prop:full_support}
We have $\mu_g(O)>0$ for any open set $O$.
\end{proposition}
\begin{proof}
The proof is the same as the one of \cite[Proposition~7.11]{BD2020MME}, replacing $\mu_*$ by $\mu_g$, as well as results used from \cite[\S~7]{BD2020MME} by their counterparts from the present section.
\end{proof}

\subsection{Bernoulli property of $\mu_g$ and Variational Principle}\label{sect:bernoulli_var_principle}

In this section, we use the absolute continuity results on the holonomy map from Section~\ref{sect:absolute_continuity} to establish that $\mu_g$ is K-mixing. We also prove an upper bound on the $\mu_g$-measure of weighted dynamical Bowen balls. Using these estimates, we are able to prove that $\mu_g$ is an equilibrium state for $T$ under the potential $g$ -- that is, $\mu_g$ realizes the $\sup$ in the definition of $P(T,g)$ -- and $\mu_g$ satisfies the variational principle: $P_*(T,g) = P(T,g)$. Finally, using again the absolute continuity along side with Cantor rectangles and the bound~\eqref{eq:estimate_mug_neighbourhood_Sk} on the neighbourhoods of the singular sets, we can bootstrap from the K-mixing to prove that $\mu_g$ is Bernoulli.

\begin{lemma}[Single Ergodic Component]
If $R$ is a Cantor rectangle with $\mu_g(R)>0$, then all the stable manifolds $\cW^s(R)$ are contained in a single ergodic component of $\mu_g$.
\end{lemma}

\begin{proof}
Replacing $\mu_*$ by $\mu_g$, the proof of the analogous result \cite[Lemma~7.15]{BD2020MME} can be applied verbatim. The proof there follows the Hopf strategy.
\end{proof}

\begin{proposition}\label{prop:mug_is_ergodic}
For all $(\cM_0^1, \alpha_g)$-H\"older potential $g$ such that $P_*(T,g) - \sup g > s_0 \log 2$ and having SSP.1 and SSP.2, $(T,\mu_g)$ is K-mixing.
\end{proposition}

\begin{proof}
Replacing $\mu_*$ by $\mu_g$, the proof of the analogous result \cite[Proposition~7.16]{BD2020MME} can be applied verbatim. We outline the steps of the proof. 

First, Baladi and Demers show that $(T^n,\mu_*)$ is ergodic for all $n \geqslant 1$. To do so, they use the topological mixing of $T$ to prove that any two Cantor rectangles belong to the same ergodic component of $T^n$. 

Then, they prove that $T$ is K-mixing. To do so, they construct a measurable partition out of the stable and unstable manifolds, that is finer than the Pinsker partition $\pi(T)$. Using the covering of $M^{\rm reg}$ by Cantor rectangles $\{R_i\}$, and the absolute continuity of the holonony map, they prove that each $R_i$ belongs to a single component of $\pi(T)$. From this, they deduce that $\pi(T)$ contains finitely many elements on which $T$ acts by permutation. Since $\pi(T)$ is $T$-invariant and $(T^n,\mu_*)$ is ergodic for all $n \geqslant 1$, $\pi(T)$ must be trivial.
\end{proof}

\begin{proposition}[Upper Bounds on Weighted Dynamical Balls]\label{prop:measure_bowen_balls}
Assume that $P_*(T,g)- \sup g > s_0 \log 2$ and that SSP.1 holds. There exists $A < \infty$ such that for all $\varepsilon >0$ sufficiently small, $x \in M$, and $n \geqslant 1$, the measure $\mu_g$ constructed in \eqref{eq:def_mu_g} satisfies
\[ \mu_g( e^{-S_n^{-1} g} \mathbbm{1}_{B_n^{-1}(x,\varepsilon)}) \leqslant A e^{-n P_*(T,g)} \, ,\]
where $B_n^{-1}(x,\ve)$ is the Bowen ball at $x$ of length $n$ for $T^{-1}$.
\end{proposition}

\begin{proof}
The inequality follows from the beginning of the proof of \cite[Proposition~7.12]{BD2020MME}, where $\mu_*, \cL$ and $h_*$ should be replaced by respectively $\mu_g, \cL_g$ and $P_*(T,g)$.
\end{proof}

\begin{corollary}\label{thm:equilibrium states}
For all $(\cM_0^1, \alpha_g)$-H\"older potential $g$ such that $P_*(T,g) - \sup g > s_0 \log 2$ and having SSP.1 and SSP.2, the measure $\mu_g$ is an equilibrium state of $T$ under the potential $g$: we have $P_*(T,g) = h_{\mu_g}(T) + \int g \, \mathrm{d}\mu_g$.
\end{corollary}

\begin{proof}
For all $x \in M$, let $\cP_{-n}^0(x)$ denotes the element of $\cP_{-n}^0$ containing $x$. By the Shannon--MacMillan--Breiman theorem, we have
\[
- \lim\limits_{n \to \infty} \frac 1n \log \mu_g(\cP_{-n}^0(x)) = h_{\mu_g}(\cP,T) = h_{\mu_g}(T) \quad \text{for  $\mu_g$-a.e. $x \in M$} \, ,
\]
where the last equality follows from the Kolmogorov--Sinai theorem (because $T$ is expansive \cite[Lemma~3.4]{BD2020MME}). Furthermore, since by Proposition~\ref{prop:mug_is_ergodic} $\mu_g$ is ergodic, then $\frac 1n \log e^{-S_n^{-1} g}$ converges to $-\mu_g(g)$ as $n$ goes to infinity. Thus
\begin{align}\label{eq:shannon}
- \lim\limits_{n \to \infty} \frac 1n \log \left( e^{-S_n^{-1} g (x)} \mu_g(\cP_{-n}^0(x)) \right) = h_{\mu_g}(T) + \int g \, \mathrm{d}\mu_g \quad \text{for  $\mu_g$-a.e. $x \in M$} \, .
\end{align}
Now, by Lemma~\ref{lemma:sup_less_poly_inf}, there exists a constant $C$ such that for all $x \in M$ and all $y,z \in \cP_{-n}^0(x)$, we have $|{S_n^{-1} g(y)} - {S_n^{-1} g(z)}| \leqslant C$. Thus 
\begin{align*}
e^{-C} \leqslant \frac{\mu_g \left( e^{-S_n^{-1} g} \mathbbm{1}_{\cP_{-n}^0(x)} \right)}{ e^{-S_n^{-1}g(x)} \mu_g(\cP_{-n}^0(x)) } \leqslant e^C \, ,
\end{align*}
and so we can replace $e^{-S_n^{-1} g (x)} \mu_g(\cP_{-n}^0(x))$ in \eqref{eq:shannon} by $\mu_g \left( e^{-S_n^{-1} g} \mathbbm{1}_{\cP_{-n}^0(x)} \right)$.

Now, we want to replace $\cP_{-n}^0(x)$ with a dynamical ball and use Proposition~\ref{prop:measure_bowen_balls}. To do so, recall that for all $\varepsilon < \varepsilon_0$, the dynamical ball $B_n(x,\varepsilon)$ is included in a single element of $\cM_0^n$, which is itself included in at most $C$ elements of $\cP_0^n$, for some $C$ independent of $x$. Thus, using time reversals
\begin{align*}
- \liminf_{n \to \infty} \frac{1}{n} \log \mu_g \left( e^{-S_n^{-1}g} \mathbbm{1}_{B_n^{-1}(x,\varepsilon)} \right) \leqslant h_{\mu_g}(T) + \int g \, \mathrm{d}\mu_g, 
\end{align*}
On the other hand, for $\varepsilon$ small enough, we get by Proposition~\ref{prop:measure_bowen_balls},
\begin{align*}
- \liminf_{n \to \infty} \frac{1}{n} \log \mu_g \left( e^{-S_n^{-1}g} \mathbbm{1}_{B_n^{-1}(x,\varepsilon)} \right) \geqslant P_*(T,g).
\end{align*}
Combining these last two inequalities, we get 
$ h_{\mu_g}(T) + \int g \, \mathrm{d}\mu_g \geqslant P_*(T,g)$, which ends the proof.
\end{proof}

\begin{proposition}\label{prop:mu_g_is_bernoulli}
Under the assumptions of Proposition~\ref{prop:mug_is_ergodic}, $(T,\mu_g)$ is Bernoulli.
\end{proposition}

\begin{proof}
The proof follows the arguments in Section~5 and 6 in \cite{ChH}, relying on the of notion vwB partitions introduced by Ornstein in \cite{ornstein70vwb}. Actually, we can apply the same modifications as in the proof of the analogous result \cite[Proposition~7.19]{BD2020MME}, replacing $\mu_*$ by $\mu_g$.
\end{proof}

\subsection{Uniqueness of the equilibrium state}\label{sect:uniqueness}

This subsection is devoted to the uniqueness of the equilibrium state $\mu_g$ (Proposition~\ref{prop:uniqueness}). The proof relies on exploiting the fact that while the lower bound on weighted Bowen balls (or the thermodynamic sums over the elements of $\cM_{-n}^0$) cannot be improved for $\mu_g$-almost every $x$, however, if one fixes $n$, most elements of $\cM_{-n}^0$ (in the sense of thermodynamic sums) should either have unstable diameter of a fixed length, or have previously been contained in an element of $\cM_{-j}^0$ having this property, for some $j< n$ (Lemma~\ref{lemma:cardinality_B2n}). Such elements collectively satisfy stronger lower bounds on their measure, when weighted accordingly (Lemma~\ref{lemma:lower_bound_mug_on_good_sets}). As we have established a good control over the sums on $\cM_{-n}^0$ and $\cM_0^n$ in Section~\ref{sect:growth_lemma}, we will work with these partitions instead of Bowen balls.

The strategy we adopt here is similar to that deployed in \cite[\S~7.7]{BD2020MME}, where Baladi and Demers proved uniqueness for the case $g \equiv 0$.

Recalling \eqref{eq:complex}, choose $m_1$ large enough so that $(Km_1+1)^{1/m_1} < e^{\frac{1}{4}(P_*(T,g) - \sup g)}$. Now, choose $\delta_2 > 0$ sufficiently small that for all $n,\, k \in \mathbbm{N}$, if $A \in \cM_{-n}^k$ is such that 
\[
\max \{ \diam^u(A),\, \diam^s(A) \} \leqslant \delta_2,
\]
then $A \smallsetminus \cS_{\pm m_1}$ consists of at most $Km_1 + 1$ connected components.

For $n \geqslant 1$, define
\begin{align*}
B_{-2n}^0 \coloneqq \{ A \in \cM_{-2n}^0 \mid \,\, &\forall \, 0 \leqslant j \leqslant n/2 \, , \,  T^{-j}A \subset E \in \cM_{-2n+j}^0 \text{ such that } \diam^u(E) < \delta_2 \}, \\
\end{align*}
and its time reversal
\begin{align*}
B_0^{2n} \coloneqq \{ A \in \cM_0^{2n} \mid \,\, &\forall \, 0 \leqslant j \leqslant n/2 \, , \, T^{j}A \subset E \in \cM_{0}^{2n-j} \text{ such that } \diam^s(E) < \delta_2 \}. 
\end{align*}
Next, set $B_{2n} = \{ A \in \cM_{-2n}^0 \mid \text{either } A \in B_{-2n}^0 \text{ or } T^{-2n}A \in B_0^{2n} \}$. Define $G_{2n} = \cM_{-2n}^0 \smallsetminus B_{2n}$.

Our first lemma shows that the thermodynamic sum over elements of $B_{2n}$ is small relative to the one over elements of $\cM_{-2n}^0$, for large $n$. Let $n_1 \geqslant 2m_1$ be chosen so that for all $A \in \cM_{-n}^0$, $\diam^s(A) \leqslant C \Lambda^{-n} \leqslant \delta_2$ for all $n \geqslant n_1$.  

\begin{lemma}\label{lemma:cardinality_B2n}
There exists $C> 0$ such that for all $n \geqslant n_1$, 
\[ \sum_{A \in B_{2n}} |e^{S_{2n}^{-1} g}|_{C^0(A)} \leqslant C e^{\frac{3}{2}nP_*(T,g)} e^{\frac{1}{2}n\sup g} (Km_1+1)^{\frac{n}{m_1}+1} \leqslant C e^{\frac{7}{4} nP_*(T,g) + \frac{1}{4} n\sup g } \, . \]
\end{lemma}
Notice that since $P_*(T,g) - \sup g > 0$, we have that $\frac{7}{4}P_*(T,g) + \frac{1}{4} \sup g < 2P_*(T,g)$.

\begin{proof}
Let $n \geqslant n_1$ and $A \in B_{-2n}^0 \subset \cM_{-2n}^0$. For all $0 \leqslant j \leqslant \lfloor n/2 \rfloor$, call $A_j \in \cM_{-\lceil 3n/2 \rceil -j}^0$ the unique element containing $T^{-\lfloor n/2 \rfloor +j} A$. By definition of $B_{-2n}^0$, we have that $\diam^u(A_j) \leqslant \delta_2 $, meanwhile $\diam^s(A_j) \leqslant \delta_2$ by choice of $n_1$.

By choice of $\delta_2$, we have that $A_0$ is the union of at most $Km_1+1$ elements of $\cM_{-\lceil 3n/2 \rceil}^{m_1}$. Thus the number of connected components of $T^{m_1}A_0$ is at most $Km_1+1$. Notice that this fact not only applies to $A_0$, but also to $A_{m_1}, \ldots, A_{lm_1}, A_{\lfloor n/2 \rfloor}$, where $\lfloor n/2 \rfloor = lm_1 + i$, $0 \leqslant i < m_1$. Thus, we get
\[ \# \{ A' \in B_{-2n}^0 \mid T^{-\lfloor n/2 \rfloor} A' \subset A_0 \} \leqslant (Km_1 +1)^{l+1} \leqslant (Km_1 +1)^{\frac{n}{m_1}+1} \, . \] 
We are now able to estimate the thermodynamic sum over $B_{-2n}^0$:
\begin{align*}
\sum_{A \in B_{-2n}^0} &|e^{S_{2n}^{-1} g} |_{C^0(A)} = \!\!\!\!\!\! \sum_{A_0 \in \cM_{- \lceil 3n/2 \rceil}^0} \sum_{\substack{A' \in B_{-2n}^0 \\ T^{-\lfloor n/2 \rfloor }A' \subset A_0}} \!\!\!\!\!\! |e^{S_{2n}^{-1} g} |_{C^0(A')} = \sum_{A_0} \sum_{A'} \left|e^{S_{\lceil 3n/2 \rceil}^{-1} g \circ T^{\lfloor n/2 \rfloor } + S_{\lfloor n/2 \rfloor}^{-1} g} \right|_{C^0(A')} \\
&\leqslant \sum_{A_0} \left|e^{S_{\lceil 3n/2 \rceil}^{-1} g } \right|_{C^0(A_0)} \sum_{A'} \left|e^{ S_{\lfloor n/2 \rfloor}^{-1} g} \right|_{C^0(A')} 
\leqslant e^{\frac{1}{2}n \sup g } (Km_1 +1)^{\frac{n}{m_1}+1} \sum_{A_0} \left|e^{S_{\lceil 3n/2 \rceil}^{-1} g } \right|_{C^0(A_0)} \\
&\leqslant C e^{\frac{7}{4}nP_*(T,g) + \frac{1}{4}n \sup g},
\end{align*}
where we used Proposition~\ref{prop:almost_exponential_growth} for the last inequality.

Now, notice that $B_0^{2n}$ is the time reversal of $B_{-2n}^0$, thus
\[ \sum_{A \in B_0^{2n}} |e^{S_{2n} g}|_{C^0(A)} \leqslant C e^{\frac{7}{4}nP_*(T^{-1},g)+\frac{1}{2}n \sup g} = C e^{\frac{7}{4}nP_*(T,g) + \frac{1}{4}n \sup g}. \]
Hence
\[ \sum_{A \in B_0^{2n}} |e^{S_{2n}^{-1} g}|_{C^0(T^{-2n}A)} = \sum_{A \in B_0^{2n}} |e^{S_{2n} g}|_{C^0(A)} \leqslant C e^{\frac{7}{4}nP_*(T,g) + \frac{1}{4}n \sup g} . \]
Finally, we get
\[ \sum_{A \in B_{2n}} |e^{S_{2n}^{-1} g} |_{C^0(A)} \leqslant 2C e^{\frac{7}{4}nP_*(T,g) + \frac{1}{4}n \sup g} . \]
\end{proof}

Next, the following lemma establishes the importance of long pieces in providing good lower bounds on the measure of weighted elements of the partition.

\begin{lemma}\label{lemma:lower_bound_mug_on_good_sets}
There exists $C_{\delta_2}>0$ such that for all $j \geqslant 1$ and all $A \in \cM_{-j}^0$ such that $\diam^u(A) \geqslant \delta_2$ and $\diam^s(T^{-j}A) \geqslant \delta_2$, we have
\[ \mu_g(e^{-S_j^{-1} g} \mathbbm{1}_A) \geqslant C_{\delta_2} e^{-j P_*(T,g)} \, . \]
\end{lemma}

\begin{proof}
Let $R_1, \ldots, R_k$ be Cantor rectangles such that $\mu_g(R_i) > 0$ for all $1 \leqslant i \leqslant k$, and such that any unstable or stable curve of length more than $\delta_2$ crosses at least one of them. Note $\cR_{\delta_2} = \{ R_1, \ldots, R_k \}$ this family.

Let $j >0$ and $A \in \cM_{-j}^0$ such that $\diam^u(A) \geqslant \delta_2$ and $\diam^s(T^{-j}A) \geqslant \delta_2$. By choice of $\cR_{\delta_2}$, $A$ crosses some rectangle $R_i$ and $T^{-j}A$ also crosses some rectangle $R_{i'}$. Note $\Xi_i$ the index set for the family of stable manifolds $W_{\xi}$ of $R_i$. For $\xi \in \Xi_i$, let $W_{\xi,A} \coloneqq W_\xi \cap A$. Since $T^{-j}A$ properly crosses $R_{i'}$ in the stable direction, and that $T^{-j}$ is smooth on $A$, it follows that $T^{-j} W_{\xi,A}$ is a single curve containing a stable manifold of $R_{i'}$.

Let $l_{\delta_2}$ denote the length of the smallest stable manifold among the ones in the family of Cantor rectangles $\cR_{\delta_2}$. Thus, for all $\xi \in \Xi_i$
\begin{align*}
\int_{W_{\xi,A}} e^{-S_j^{-1} g} \nu = e^{-jP_*(T,g)} \int_{T^{-j}W_{\xi,A}} \nu \geqslant e^{-jP_*(T,g)} \bar{C}l_{\delta_2}^{\bar{C_2}(P_*(T,g) - \sup g)}.
\end{align*}
Finally, let $D(R_i)$ be the smallest solid rectangle containing $R_i$. Since $\mu_g^{W}$ and $\nu$ are equivalent on $\mu_g$-a.e. $W \in \hW^s$, we get
\begin{align*} 
\mu_g(e^{-S_{j}^{-1} g} \mathbbm{1}_A) &\geqslant \mu_g(e^{-S_{j}^{-1} g} \mathbbm{1}_{A \cap D(R_i)})
\geqslant C''  \nu (e^{-S_{j}^{-1} g} \mathbbm{1}_{A \cap W_\xi})  \geqslant C_{\delta_2}' e^{-jP_*(T,g)}.
\end{align*}
where we used \eqref{eq:mug_geq_nu} (with $\bar E = A$ and $E = A \cap W_\Xi$) for the second inequality. Since the family $\cR_{\delta_2}$ is finite, this proves the lemma.
\end{proof}

\begin{proposition}\label{prop:uniqueness}
If $g$ is a $(\cM_0^1,\alpha_g)$-H\"older potential with $P_*(T,g) - \sup g > s_0 \log 2$, having SSP.1 and SSP.2, then the measure $\mu_g$ is the unique equilibrium state for $T$ under the potential $g$.
\end{proposition}

\begin{proof}
Usually, given a known equilibrium state (thus ergodic) $\mu_g$, in order to prove uniqueness it suffices to check that for all $T$-invariant measure $\mu$ singular with respect to $\mu_g$, we have $h_{\mu}(T) + \mu(g) < h_{\mu_g}(T) +\mu_g(g)$ -- see for example \cite[Section 20.3]{katok1997introduction}. This is the strategy we adopt.

Let $\mu$ be a $T$-invariant Borel probability measure, singular with respect to $\mu_g$, that is there exists a Borel set $F \subset M$ with $T^{-1}F =F$ and $\mu_g(F) = 0$ but $\mu(F)=1$.

For each $n \in \mathbbm{N}$, we consider the partition $\cQ_n$ of maximal connected components of $M$ on which $T^{-n}$ is continuous. By \cite[Lemma 3.2 and 3.3]{BD2020MME}, $\cQ_n$ is $\cM_{-n}^0$ plus isolated points whose cardinality grows at most linearly with $n$. Thus $G_{2n} \subset \cQ_{2n}$ for each $n$. Define $\tilde B_{2n} = \cQ_{2n} \smallsetminus G_{2n}$. The set $\tilde B_{2n}$ contains $B_{2n}$ plus isolated points, and so its associated thermodynamic sum is bounded by the expression in Lemma~\ref{lemma:cardinality_B2n} plus $\#\{ \text{isolated points} \}e^{2n \sup g}$. Since $P_*(T,g) - \sup g > 0$, we have that $\tfrac{7}{4}P_*(T,g) + \tfrac{1}{4} \sup g > 2\sup g$, and thus the contribution of isolated points is small compared to the upper bound from Lemma~\ref{lemma:cardinality_B2n}.

By uniform hyperbolicity of $T$, the diameters of the elements of $T^{\lfloor n/2 \rfloor} \cQ_n$ tend to zero as $n$ goes to infinity. This implies the following fact.

\begin{sublemma}\label{sublemma:sym_diff}
For each $n \geqslant n_1$, there exists a finite union $\cC_n$ of elements of $\cQ_n$ such that 
\[ \lim_{n \to + \infty} (\mu + \mu_g)(F \, \triangle \,  T^{-\lfloor n/2 \rfloor} \cC_n) =0 \, .\]
\end{sublemma}

\begin{proof}
The proof is essentially the same as \cite[Sublemma~7.24]{BD2020MME} where the role of $\mu_*$ is played by $\mu_g$. Since notations are introduced in this proof, we write it down for completeness and latter use.

Let $\bar \mu = \mu + \mu_g$ and $\tilde \cQ_n = T^{- \lfloor n/2 \rfloor} \cQ_n$. For $\delta >0$, by regularity of Radon measures, pick compact sets $K_1 \subset F$ and $K_2 \subset M \smallsetminus F$ such that $\max \{ \bar \mu( F \smallsetminus K_1), \, \bar \mu ((M \smallsetminus F) \smallsetminus K_2) \} < \delta$. Since $K_1$ and $K_2$ are disjoint and compact, we have $\eta = \eta_\delta \coloneqq d(K_1, K_2) > 0$. If $\diam(\tilde Q) < \eta /2$, then either $\tilde Q \cap K_1 = \emptyset$ or $\tilde Q \cap K_2 = \emptyset$. Let $n_\delta$ be large enough so that the diameter of $\tilde \cQ_k$ is smaller than $\eta_\delta /2$ for all $k \geqslant n_\delta$. Fix $n=2 n_\delta$ and set $\tilde{ \cC_n}$ to be the union of $\tilde Q \in \tilde \cQ_n$ such that $\tilde Q \cap K_1 \neq \emptyset$. By construction, $K_1 \subset \tilde \cC_n$ and $\tilde \cC_n \cap K_2 = \emptyset$. Hence $\bar \mu(F \, \triangle \, \tilde \cC_n) \leqslant \delta + \bar \mu (K_1 \, \triangle \, \tilde \cC_n) \leqslant \delta + \bar \mu (M \smallsetminus (K_1 \cup K_2)) \leqslant 3 \delta$. Defining $\cC_n = T^{\lfloor n/2 \rfloor} \tilde \cC_n$ completes the proof.
\end{proof}

Remark that since $T^{-1}F=F$, it follows that $(\mu + \mu_g)(\cC_n \, \triangle \, F)$ also tends to zero as $n \to + \infty$.

Since $\cQ_{2n}$ is generating for $T^{2n}$, we have
\begin{align*}
h_{\mu}(T^{2n}) = h_\mu(T^{2n}, \cQ_{2n}) \leqslant H_\mu(\cQ_{2n}) = - \sum_{Q \in \cQ_{2n}} \mu(Q) \log \mu(Q) \, .
\end{align*}
Thus,
\begin{align*}
2n P_\mu(T,g) &= 2n h_\mu(T) + 2n \mu(g) = h_\mu(T^{2n}) + \mu(S_{2n}^{-1}g) \leqslant H_\mu(\cQ_{2n}) + \mu(S_{2n}^{-1}g) \\
&\leqslant \sum_{Q \in \cQ_{2n}} \mu(Q) \left( -\log \mu(Q) + S_{2n}^{-1} g(x_Q) + C_g  \right),
\end{align*}
where $x_Q \in Q$ and $C_g$ is the constant from Lemma~\ref{lemma:sup_less_poly_inf}.

Now, we want to distinguish elements of $\cQ_{2n}$. From the proof of Sublemma~\ref{sublemma:sym_diff}, for each $n$, there exists a compact set $K_1(n)$ that defines $\tilde \cC_n = T^{-\lfloor n/2 \rfloor} \cC_n$, and satisfying $(\mu + \mu_g)(\cup_n K_1(n)) = (\mu + \mu_g)(F)$. We group elements $Q \in \cQ_{2n} \subset \cQ_n$ according to whether $T^{-n}Q \subset \tilde \cC_{n}$ or $T^{-n}Q \cap \tilde \cC_{n} = \emptyset$. This dichotomy is well defined because if $Q$ is not an isolated point, and if $T^{-n}Q \cap \tilde \cC_{n} \neq \emptyset$, then $T^{-n}Q \in \cM_{-n}^n$ is contained in an element of $\cM_{-\lceil n/2 \rceil}^{\lfloor n/2 \rfloor}$ that intersect $K_1(n)$. Thus $Q \subset T^n \tilde \cC_{n} = T^{\lceil n/2 \rceil} \cC_{n}$ -- the case where $Q$ is an isolated point is obvious. Therefore,
\begin{align*}
2n &P_\mu(T,g) \leqslant C_g + \!\!\!\! \sum_{Q \subset T^n \tilde \cC_{n}} \!\! \mu(Q)\left( -\log \mu(Q) + S_{2n}^{-1} g(x_Q) \right) + \!\!\!\!\!\!\!\!\!\! \sum_{Q \in \cQ_{2n} \smallsetminus T^n \tilde \cC_{n}} \!\!\!\!\!\!\!\! \mu(Q)\left( -\log \mu(Q) + S_{2n}^{-1} g(x_Q) \right) \\
&\leqslant C_g + \frac{2}{e} + \mu(T^n \tilde \cC_{n}) \log \left( \sum_{Q \subset T^n \tilde \cC_{n}} e^{S_{2n}^{-1} g(x_Q)} \right) + \mu(M \smallsetminus T^n \tilde \cC_{n}) \log \left( \sum_{Q \in \cQ_{2n} \smallsetminus T^n \tilde \cC_{n}} e^{S_{2n}^{-1} g(x_Q)} \right) 
\end{align*}
where we used in the last line that the convexity of $x \log x$ implies that for all $p_j > 0$ with $\sum_{j=1}^N p_j \leqslant 1$, and all $a_j \in \mathbbm{R}$, we have (see \cite[(20.3.5)]{katok1997introduction})
\[ \sum_{j=1}^N p_j ( - \log p_j + a_j) \leqslant \frac{1}{e} + \sum_{j=1}^N p_j \log \sum_{i=1}^N e^{a_i} \, . \]

Then, since $-2n P_{\mu_g} = ( \mu(T^n \tilde \cC_{n}) + \mu(M \smallsetminus T^n \tilde \cC_{n}) ) \log e^{-2n P_*(T,g)}$, we write for $n \geqslant n_1$

\begin{align}\label{eq:uniqueness-final_estimate}
\begin{split}
&2n(P_{\mu}(T,g) - P_{\mu_g}(T,g)) - \frac{2}{e} - C_g \\
&\leqslant \mu(T^{-n} \tilde \cC_{n}) \left( \log \!\!\!\! \sum_{Q \subset T^n \tilde \cC_{n}} \!\!\!\! e^{S_{2n}^{-1}g(x_Q) - 2n P_*(T,g)} \right) + \mu(M \smallsetminus T^{-n} \tilde \cC_{n}) \left( \log \!\!\!\!\!\!\!\!\! \sum_{Q \in \cQ_{2n} \smallsetminus T^n \tilde \cC_{n}} \!\!\!\!\!\!\!\!\!\! e^{S_{2n}^{-1}g(x_Q) - 2n P_*(T,g)} \right) \\
&\leqslant \mu(\cC_{n}) \log \left( \sum_{\substack{Q \subset T^n \tilde \cC_{n} \\ Q \in G_{2n}}} e^{S_{2n}^{-1}g(x_Q) - 2n P_*(T,g)} + \sum_{\substack{Q \subset T^n \tilde \cC_{n} \\ Q \in \tilde B_{2n}}} e^{S_{2n}^{-1}g(x_Q) - 2n P_*(T,g)} \right) \\
&\quad + \mu(M \smallsetminus \cC_{2n}) \log \left( \sum_{Q \in G_{2n} \smallsetminus T^n \tilde \cC_{n} } e^{S_{2n}^{-1}g(x_Q) - 2n P_*(T,g)} + \sum_{Q \in \tilde B_{2n} \smallsetminus T^n \tilde \cC_{n} } e^{S_{2n}^{-1}g(x_Q) - 2n P_*(T,g)} \right)
\end{split}
\end{align}
where we used that $\cQ_{2n} = G_{2n} \sqcup \tilde B_{2n}$. By Lemma~\ref{lemma:cardinality_B2n} (and the remark concerning the contribution of isolated points), both sums over elements of $\tilde B_{2n}$ are bounded by $C e^{-\frac{1}{4}n(P_*(T,g) - \sup g)}$. 

It remains to estimate both sums over elements of $G_{2n}$. To do so, we want to use Lemma~\ref{lemma:lower_bound_mug_on_good_sets}, that is for each $Q \in G_{2n}$, we want to assign a set $\bar{E}$ satisfying the assumptions of the lemma. Let $Q \in G_{2n}$. Thus $Q \notin B^0_{-2n}$, and so there exists $0 \leqslant j \leqslant \lfloor n/2 \rfloor$ such that $T^{-j}Q \subset E_j \in \cM_{2n + j}^0$ with $\diam^u(E_j) \geqslant \delta_2$. Also, since $T^{-2n}Q \notin B_0^{2n}$, there exists $0 \leqslant k \leqslant \lfloor n/2 \rfloor$ such that $T^{-2n +k} Q \subset \tilde E_k \in \cM_{0}^{2n-k}$ with $\diam^s(\tilde E_k) \geqslant \delta_2$. Thus, both $\tilde E_k \in \cM_0^{2n-k}$ and $T^{-2n +j+k}E_j \in \cM_{-k}^{2n-j-k}$ contain $T^{-2n +k}Q$. In particular, there exists $\tilde E \in \cM_0^{2n-j-k}$ containing both $\tilde E_k$ and $T^{-2n +j+k}E_j$. Let $\bar E = T^{2n-j-k} \tilde E \in \cM_{-2n+j+k}^0$. Notice that by construction $E_j \subset \bar E$ and $\tilde E_k \subset T^{-2n+j+k} \bar E$, therefore $\bar E$ satisfies $\diam^u(\bar E) \geqslant \delta_2$ and $\diam^s(T^{-2n+j+k} \bar E) \geqslant \delta_2$, the assumption from Lemma~\ref{lemma:lower_bound_mug_on_good_sets}. Thus,
\[ \mu_g (e^{-S^{-1}_{2n-j-k}g} \mathbbm{1}_{\bar E}) \geqslant C_{\delta_2} e^{-(2n-j-k)P_*(T,g)} \, . \]

We call $(\bar E, j, k)$ an \emph{admissible triple} for $Q \in G_{2n}$ if $0 \leqslant j,k \leqslant \lfloor n/2 \rfloor$ and $\bar E \in \cM_{-2n + j + k}^0$, with $T^{-j}Q \subset \bar E$ and $\min \{ \diam^u(\bar E), \diam^s(T^{-2n+j+k} \bar E) \} \geqslant \delta_2$. By the above construction, such admissible triples always exist, but there may be many associated to a given $Q \in G_{2n}$. However, we can define the unique \emph{maximal triple} for $Q$ by taking first the maximum $j$, and then the maximum $k$ over all admissible triples for $Q$.

Let $\cE_{2n}$ be the set of maximal triples obtained in this way from elements of $G_{2n}$. For $(\bar E,j,k) \in \cE_{2n}$, let $\cA_M(\bar E,j,k)$ denote the set of $Q \in G_{2n}$ for which the maximal triple is $(\bar E,j,k)$. The importance of the set $\cE_{2n}$ lies in \cite[Sublemma~7.25]{BD2020MME}, which we state, and prove, as follows for completeness.
\begin{sublemma}
Suppose that $(\bar E_1,j_1,k_1),\, (\bar E_2,j_2,k_2) \in \cE_{2n}$ with $j_2 \geqslant j_1$ and $(\bar E_1,j_1,k_1) \neq (\bar E_2,j_2,k_2)$. Then $T^{-(j_2-j_1)} \bar E_1 \cap \bar E_2 = \emptyset$.
\end{sublemma}
\begin{proof}
By contradiction, let $(\bar E_1,j_1,k_1),\, (\bar E_2,j_2,k_2) \in \cE_{2n}$ with $j_2 \geqslant j_1$, $(\bar E_1,j_1,k_1) \neq (\bar E_2,j_2,k_2)$ and $T^{-(j_2-j_1)} \bar E_1 \cap \bar E_2 \neq \emptyset$. Notice that $T^{-(j_2-j_1)} \bar E_1 \in \cM_{-2n + j_2 + k_1}^{j_2-j_1}$ while $\bar E_2 \in \cM_{-2n+j_2+k_2}^0$.

Consider first the case $k_1 \leqslant k_2$. Therefore $T^{-(j_2-j_1)} \bar E_1 \subset E_2$. In particular, any element $Q \in \cA_M(\bar E_1,j_1,k_1)$ satisfies $T^{-j_2} Q \subset \bar E_2$, and so $Q \in \cA_M(\bar E_2,j_2,k_2)$, a contradiction.

Consider now the case $k_1 > k_2$. Therefore $T^{-(j_2-j_1)} \bar E_1$ and $\bar E_2$ are both contained in an element $\bar E' \in \cM_{-2n + j_2+k_1}^0$. Since $\bar E_2 \subset \bar E'$, we have that $\diam^u(\bar E') \geqslant \delta_2$. Also, since $T^{-2n+j_1+k_1} \bar E_1 \subset T^{-2n + j_2+k_1} \bar E'$, we must have that $\diam^s(T^{-2n + j_2+k_1} \bar E') \geqslant \delta_2$. Note that if $Q \in \cA_M(\bar E_1,j_1,k_1) \cup \cA_M(\bar E_2,j_2,k_2)$, then $(\bar E', j_2,k_1)$ is an admissible triple for $Q$. Thus, if $j_1=j_2$, then $\bar E' = \bar E_1$. For $Q \in \cA_M(\bar E_2,j_2,k_2)$, then $Q \subset \bar E_1$ and so $(\bar E_1,j_1,k_1)$ is an admissible triple for $Q$, which contradicts the maximality of $(\bar E_2,j_2,k_2)$ since $k_1 > k_2$. Similarly, if $j_2 > j_1$, then for $Q \in \cA_M(\bar E_1,j_1,k_1)$, the triple $(\bar E',j_2,k_1)$ is admissible for $Q$, which contradicts the maximality of $(\bar E_1,j_1,k_1)$.
\end{proof}

We now prove that if $T^{n} \tilde \cC_{n} \cap \cA_M(\bar E, j,k) \neq \emptyset$, then $\cA_M(\bar E, j,k) \subset T^{n} \tilde \cC_{n}$ and $\bar E \subset T^{n-j} \tilde \cC_n$. Let $Q \in \cA_M(\bar E, j,k)$ be such that $Q \cap T^{n} \tilde \cC_{n} \neq \emptyset$. Then, by definition of $(\bar E,j,k)$, $T^{-n}Q \subset T^{-n+j}\bar E \in \cM_{-n+k}^{n-j}$. Since $0 \leqslant j,k \leqslant \lfloor n/2 \rfloor$, there exists $E' \in \cM_{-\lfloor n/2 \rfloor}^{\lfloor n/2 \rfloor}$ such that $T^{-n+j}\bar E \subset E'$. In particular, we have $E' \in \tilde \cQ_n$ and $E' \cap \tilde \cC_n \neq \emptyset$. Thus, by construction of $\tilde \cC_n$, we have $\tilde \cC_n \supset E' \supset T^{-n+j}\bar E  \supset T^{-n}Q$. In particular, we get $Q \subset T^{n} \tilde \cC_n$, and thus $\cA_m(\bar E,j,k) \subset T^n \tilde \cC_n$. We also get $\bar E \subset T^{n-j} \tilde \cC_n$.

On the other hand, we prove that if $T^{n} \tilde \cC_{n} \cap \cA_M(\bar E, j,k) = \emptyset$, then $\cA_M(\bar E, j,k) \subset M \smallsetminus T^{n} \tilde \cC_{n}$ and $T^{-n+j} \bar E \subset M \smallsetminus \tilde \cC_n$. Let $Q \in \cA_M(\bar E, j,k)$. Then, by assumption, $T^{-n}Q \cap \tilde \cC_n = \emptyset$. As above, there exists $E' \in \cM_{-\lfloor n/2 \rfloor}^{\lfloor n/2 \rfloor}$ containing both $T^{-n}Q$ and $T^{-n+j}\bar E$. In particular, $E' \in \tilde \cQ_n$ and $E' \cap \tilde \cC_n = \emptyset$. By construction of $\tilde \cC_n$, we get that $E' \in  M \smallsetminus \tilde \cC_n$. Thus $Q \in M \smallsetminus T^n \tilde \cC_n$, and so $\cA_M(\bar E, j,k) \subset M \smallsetminus T^n \tilde \cC_n$. Also, $T^{-n+j} \bar E \subset M \smallsetminus \tilde \cC_n$.

The only last step we have to do before estimating the sums over $G_{2n}$ is to prove that for each $(\bar E,j,k) \in \cE_{2n}$, we have 
\begin{align}\label{eq:upper_bound_cA_M}
\sum_{Q \in \cA_M(\bar E,j,k)} |e^{S_{2n}^{-1} g}|_{C^0(Q)}  \leqslant C e^{(j+k)P_*(T,g)} |e^{S_{2n-j-k}^{-1}g}|_{C^0(\bar E)}
\end{align}
where $C>0$ is a constant depending only on the potential $g$. To do so, notice that if $Q \in \cA_M(\bar E,j,k)$, then by construction, $T^{-j} Q \subset \bar E$. Thus $T^{-n}Q \in \cM_{-n}^n$ is a subset of $T^{-(n-j)} \bar E \in \cM_{-n+k}^{n-j}$. Decomposing $T^{-n}Q = Q_- \cap Q_+$ with $Q_- \in \cM_{-n}^0$ and $Q_+ \in \cM_{0}^{n}$, and $T^{-(n-j)} \bar E = E_- \cap E_+$ with $E_- \in \cM_{-n+k}^0$ and $E_+ \in \cM_{0}^{n-j}$, we see that $Q_- \subset E_-$ and $Q_+ \subset E_+$. Thus

\begin{align*}
\sum_{Q \in \cA_M(\bar E,j,k)} \!\!\! |e^{S_{2n}^{-1} g}&|_{C^0(Q)} = \!\!\! \sum_{Q \in \cA_M(\bar E,j,k)} \!\!\! |e^{S_{2n}^{-1} g \circ T^n}|_{C^0(T^{-n}Q)} \leqslant \!\!\! \sum_{\substack{Q_- \in \cM_{-n}^0 \\ Q_- \subset E_-}}  \! \sum_{\substack{Q_+ \in \cM_{0}^n \\ Q_+ \subset E_+}} \!\!\! |e^{S_n^{-1}g + S_n g \circ T} |_{C^0(Q_- \cap Q_+)} \\
&\leqslant \sum_{\substack{Q_- \in \cM_{-n}^0 \\ Q_- \subset E_-}} |e^{S_n^{-1}g}|_{C^0(Q_-)}  \sum_{\substack{Q_+ \in \cM_{0}^n \\ Q_+ \subset E_+}} |e^{S_n g \circ T} |_{C^0(Q_+)} \\
&\leqslant \sum_{\substack{Q_- \in \cM_{-n}^0 \\ Q_- \subset E_-}} |e^{S_n^{-1}g \circ T^{n-k} }|_{C^0(T^{-n+k}Q_-)}  \sum_{\substack{Q_+ \in \cM_{0}^n \\ Q_+ \subset E_+}} |e^{S_n g \circ T \circ T^{-(n-j)}} |_{C^0(T^{n-j}Q_+)} \, .
\end{align*}
Now, notice that $T^{-n+k} Q_- \in \cM_{-k}^{n-k}$ is a subset of $T^{-n+k}E_- \in \cM_{0}^{n-k}$. Thus $T^{-n+k} Q_-$ must be of the form $\tilde Q_- \cap T^{-n+k}E_-$ for some $\tilde Q_- \in \cM_{-k}^0$. Similarly, $T^{n-j}Q_+$ must be of the form $\tilde Q_+ \cap T^{n-j}E_+$ for some $\tilde Q_+ \in \cM_0^j$. Thus
\begin{align*}
&\sum_{Q \in \cA_M(\bar E,j,k)} \!\!\!\! |e^{S_{2n}^{-1} g}|_{C^0(Q)} 
\leqslant \!\!\!\! \sum_{\tilde Q_- \in \cM_{-k}^0} \!\!\!\! |e^{S_n^{-1}g \circ T^{n-k} }|_{C^0(\tilde Q_- \cap T^{-n+k}E_-)}  \!\!\! \sum_{\tilde Q_+ \in \cM_0^j} \!\!\! |e^{S_n g \circ T \circ T^{-(n-j)}} |_{C^0(\tilde Q_+ \cap T^{n-j}E_+)} \\
&\qquad \leqslant \!\!\! \sum_{\tilde Q_- \in \cM_{-k}^0} \!\!\! |e^{S_k^{-1}g}|_{C^0(\tilde Q_-)} |e^{S_{2n-j-k}^{-1}g - S_{n-j}^{-1}g}|_{C^0(T^{n-j}E_-)} \!\!\! \sum_{\tilde Q_+ \in \cM_{0}^j} \!\!\! |e^{S_{j}g \circ T}|_{C^0(\tilde Q_+)} |e^{S_{n-j}^{-1}g}|_{C^0(T^{n-j}E_+)} \, .
\end{align*}
Now, using Lemma~\ref{lemma:sup_less_poly_inf}, the supermultiplicativity from Lemma~\ref{lemma:supermultiplicativity} and the exact exponential growth from Proposition~\ref{prop:almost_exponential_growth}, we get the upper bound \eqref{eq:upper_bound_cA_M} with $C = 2 \, C_g \, e^{\sup g - \inf g}$.

We can now estimate the sums over elements of $G_{2n}$.
\begin{align*}
&\sum_{\substack{Q \in G_2n \\ Q \subset T^n \tilde \cC_{n}}} e^{S_{2n}^{-1}g(x_Q) - 2n P_*(T,g)} 
\leqslant \sum_{\substack{ (\bar E,j,k) \in \cE_{2n} \\ \bar E \subset T^{n-j} \tilde \cC_n }} \sum_{Q \in \cA_M(\bar E,j,k)} e^{S_{2n}^{-1}g(x_Q) - 2n P_*(T,g)} \\
&\leqslant \!\!\! \sum_{\substack{ (\bar E,j,k) \in \cE_{2n} \\ \bar E \subset T^{n-j} \tilde \cC_n }} \!\!\! C e^{-(2n-j-k)P_*(T,g)} |e^{S_{2n-j-k}^{-1}g}|_{C^0(\bar E)} 
\leqslant \!\!\! \sum_{\substack{ (\bar E,j,k) \in \cE_{2n} \\ \bar E \subset T^{n-j} \tilde \cC_n }} \!\!\! C C_{\delta_2}^{-1} \mu_g( e^{-S_{2n-j-k}^{-1}g} \mathbbm{1}_{\bar E}) |e^{S_{2n-j-k}^{-1}g}|_{C^0(\bar E)} \\
&\leqslant C C_{\delta_2}^{-1} C_g  \sum_{\substack{ (\bar E,j,k) \in \cE_{2n} \\ \bar E \subset T^{n-j} \tilde \cC_n }} \mu_g(\bar E) \leqslant C C_{\delta_2}^{-1} C_g  \sum_{\substack{ (\bar E,j,k) \in \cE_{2n} \\ \bar E \subset T^{n-j} \tilde \cC_n }} \mu_g(T^{-n+j}\bar E) 
\leqslant C' \,  \mu_g(\tilde \cC_n)
\end{align*}
where $C'=C C_{\delta_2}^{-1} C_g$.

Similarly, 
\begin{align*}
&\sum_{Q \in G_{2n} \smallsetminus T^n \tilde \cC_{n} } e^{S_{2n}^{-1}g(x_Q) - 2n P_*(T,g)} 
\leqslant \sum_{\substack{(\bar E,j,k) \in \cE_{2n} \\ \bar E \subset M \smallsetminus T^{n-j} \tilde \cC_n }} \sum_{Q \in \cA_M(\bar E,j,k)} e^{S_{2n}^{-1}g(x_Q) - 2n P_*(T,g)} \\
&\qquad \leqslant \sum_{\substack{(\bar E,j,k) \in \cE_{2n} \\ \bar E \subset M \smallsetminus T^{n-j} \tilde \cC_n }} C e^{-(2n-j-k)P_*(T,g)} |e^{S_{2n-j-k}^{-1}g}|_{C^0(\bar E)} \\
&\qquad \leqslant \sum_{\substack{(\bar E,j,k) \in \cE_{2n} \\ \bar E \subset M \smallsetminus T^{n-j} \tilde \cC_n }} C C_{\delta_2}^{-1} \mu_g( e^{-S_{2n-j-k}^{-1}g} \mathbbm{1}_{\bar E}) |e^{S_{2n-j-k}^{-1}g}|_{C^0(\bar E)} \\
&\qquad \leqslant C C_{\delta_2}^{-1} C_g \sum_{\substack{(\bar E,j,k) \in \cE_{2n} \\ \bar E \subset M \smallsetminus T^{n-j} \tilde \cC_n }} \mu_g(\bar E) \leqslant C C_{\delta_2}^{-1} C_g  \sum_{\substack{ (\bar E,j,k) \in \cE_{2n} \\ \bar E \subset M \smallsetminus T^{n-j} \tilde \cC_n }} \mu_g(T^{-n+j}\bar E)
\leqslant C' \,  \mu_g(M \smallsetminus \tilde \cC_n) \, .
\end{align*}

Putting these bounds together allows us to complete our estimate of \eqref{eq:uniqueness-final_estimate},
\begin{align*}
2n(P_\mu(T,g) - P_{\mu_g}(T,g)) - \frac{2}{e} - C_g &\leqslant \mu(\cC_n) \log \left( C' \mu_g(\cC_n) + Ce^{-\frac{1}{4} n(P_*(T,g)- \sup g)} \right) \\
&\,\, + \mu(M\smallsetminus \cC_n) \log \left( C' \mu_g(M \smallsetminus \cC_n) + Ce^{-\frac{1}{4} n(P_*(T,g)- \sup g)} \right).
\end{align*}
Since by Sublemma~\ref{sublemma:sym_diff} $\mu(\cC_n)$ tends to $1$ as $n \to +\infty$, while $\mu_g(\cC_n)$ tends to $0$ as $n \to +\infty$, the limit of the right-hand side tends to $-\infty$. This yields a contradiction unless $P_{\mu}(T,g) < P_{\mu_g}(T,g)$.
\end{proof}

\section{The Billiard Flow}\label{sect:billiard_flow}

Throughout this section, we see the billiard flow $\phi_t$ as the vertical flow in the space 
\[ \tilde\Omega = \{ (x,t) \in M\times \mathbbm{R} \mid 0 \leqslant t \leqslant \tau(x) \} /\sim, \]
where the equivalence relation is defined by $(x,\tau(x))\sim(T(x),0)$. In other words, we see $\phi_t$ as the suspension flow over $T$ under the return time $\tau$. Furthermore, transporting the Euclidean metric on $\cQ \times \mathbbm{S}^1$ onto $\tilde \Omega$, the flow $\phi_t$ is uniformly hyperbolic.

\begin{proposition}\label{prop:billiard_flow_k-system}
Let $g$ be a $(\cM_0^1,\alpha_g)$-H\"older potential such that $P_*(T,g) - \sup g > s_0 \log 2$, with SSP.1 and SSP.2. Let $\bar \mu_g = (\mu_g(\tau))^{-1} \mu_g \otimes \lambda$, where $\lambda$ is the Lebesgue measure. Then $(\phi_t, \bar \mu_g)$ is K-mixing.
\end{proposition}

\begin{proof}
The ergodicity of $(\phi_t,\bar \mu_g)$ follows directly from the ergodicity of $(T,\mu_g)$ proved in Proposition~\ref{prop:mug_is_ergodic}.

To prove the K-mixing, we follow closely the method used in Sections~6.9, 6.10 and 6.11 from \cite{chernov2006chaotic}. In fact, replacing $\mu$ and $\mu_\Omega$ with $\mu_g$ and $\bar \mu_g$ throughout these sections, we only have to check that \cite[Exercise~6.35]{chernov2006chaotic} is still true in order to apply verbatim the arguments. This is what we prove here.

To do so, we first need to recall some of the construction done in \cite[Section~6.9]{chernov2006chaotic}. If $x_1$ and $x_3$ are two nearby points in $M$ such that 
\begin{align}\label{eq:condition_x1}
\{ x_2 \} \coloneqq W^u(x_1) \cap W^s(x_3) \neq \emptyset \, , \quad \{ x_4 \} \coloneqq W^s(x_1) \cap W^u(x_3) \neq \emptyset \, ,
\end{align}
we then construct the $4$-loop $Y_1$, $Y_2$, $Y_3$, $Y_4$, $Y_5 \in \Omega$ as follow. Let $Y_1 = X_1 = (x_1,t)$ and $X_3=(x_3,t)$. Define
\begin{align*}
Y_2 = W^u(Y_1) \cap W^{ws}_{\rm loc}(X_3), \quad Y_3 = W^s(Y_2) \cap W^{wu}_{\rm loc}(X_3), \\
Y_4 = W^u(Y_3) \cap W^{ws}_{\rm loc}(X_1), \quad Y_5 = W^s(Y_4) \cap W^{wu}_{\rm loc}(X_1),
\end{align*}
where $W^u$ and $W^s$ are unstable and stable manifolds for the flow, and $W^{wu}_{loc}$ and $W^{ws}_{loc}$ are local weak unstable and local weak stable manifolds for the flow. We always assume that this construction stays under the ceiling function $\tau$. Actually, as proven in \cite[Lemma~6.40]{chernov2006chaotic} there exists $\sigma$ such that $Y_5 = \phi_\sigma (Y_1)$, with $|\sigma|=\musrb(K)$ where $K$ is the rectangle in $M$ with corners $x_1$, $x_2$, $x_3$, $x_4$. Thus the $4$-loops are always open.

For $x \in M$, let $\cL_x = \{ \phi_t(x) \mid 0<t<\tau(x) \}$. Then the partition $\{ \cL_x \mid x \in M \}$ of $\tilde \Omega$ is measurable and the conditional measures of $\bar \mu_g$ on $\cL_x$ are uniform. Call $\lambda_x$ the Lebesgue probability measure on $\cL_x$. Let $D \subset \Omega$ be such that $\bar \mu_g(D) =1$ and let $E_1 = \{ x \in M \mid \lambda_x(\cL_x \smallsetminus D) = 0 \}$. Clearly, $\mu_g(E_1)=1$. 

We call a point $x_1 \in E_1$ \emph{rich} if for any $\varepsilon > 0$ there exists another point $x_3 \in E_1$ such that $0 < d(x_1,x_3)<\varepsilon$ and \eqref{eq:condition_x1} holds with $x_2$ and $x_4 \in E_1$. Denote $E_2 \subset E_1$ the set of rich points.

The analogous of \cite[Exercise~6.35]{chernov2006chaotic} is to prove that $\mu_g(E_2)=1$. Let $\{R_j\}_{j \geqslant 1}$ be the cover of $M^{\mathrm{reg}}$ into Cantor rectangles (discarding the ones with zero $\mu_g$-measure). Let $R$ be one of those Cantor rectangle and denote $\mu_R$ the conditional measure of $\mu_g$ on $R$. It is enough to prove that $\mu_R(E_2)=1$. Since $\mu_g(E_1)=1$ we have that $\mu_R(E_1)=1$. Furthermore, since the partition of $R$ into stable manifolds is measurable, we can disintegrate $\mu_R$ with respect to this partition, with conditional measure $\mu_s^W$ on $W \in R \cap \cW^s$. It follows that for $\mu_R$-a.e. point $x \in E_1 \cap R$, if $W=W(x) \in \cW^s$ contains $x$ then $\mu_s^W(W \cap E_1) = 1$. Similarly, for $\mu_R$-a.e. point $x \in E_1 \cap R$, if $W=W(x) \in \cW^u$ contains $x$ then $\mu_u^W(W \cap E_1) = 1$, where $\mu_u^W$ is the conditional measure on $W$ in the disintegration of $\mu_R$ with respect to the measurable partition $R \cap \cW^u$ of $R$. Then $\mu_R(E_R)=1$, where $E_R$ denotes the set of points $x$ in $R$ such that both stable and unstable conditional measure on leaves containing $x$ give measure $1$ to $E_1$. 

Let $E_2^R \subset E_2$ be the set of rich points $x_1$ such that $x_3$ belongs to $R \cap E_1$ (and therefore $x_2$ and $x_4$ also belong to $R \cap E_1$ by the properties of a Cantor rectangle). By contradiction, assume that $\mu_R(E_2^R) \neq 1$. Define the sets 
\begin{align*}
C_2^R = \{ x_1 \in E_1 \cap R \mid \exists \varepsilon >0, \forall x_3 \in E_1 \cap R, \,\text {if } 0<d(x_1,x_3)<\varepsilon \text{ then } x_2 \notin E_1 \cap R \}, \\
C_4^R = \{ x_1 \in E_1 \cap R \mid \exists \varepsilon >0, \forall x_3 \in E_1 \cap R, \,\text {if } 0<d(x_1,x_3)<\varepsilon \text{ then } x_4 \notin E_1 \cap R \}.
\end{align*}
Note that we don't have to introduce in these definitions the condition \eqref{eq:condition_x1} since it is automatically satisfied by the construction of Cantor rectangles. Thus, we have $(E_1 \cap R) \smallsetminus E_2^R = C_2^R \cup C_4^R$, so that $\mu_R(C_2^R \cup C_4^R) >0$. Assume first that $\mu_R(C_2^R) > 0$. Define the family of sets
\begin{align*}
C_{2,n}^R = \{ x_1 \in C_2^R \mid \varepsilon \geqslant \tfrac{1}{n} \}.
\end{align*}
Since $\bigcup_{n \geqslant 1} C_{2,n}^R = C_{2}^R$ is an increasing union, there is some $n$ such that $\mu_R(C_{2,n}^R)>0$. Let $x_1 \in C_{2,n}^R \cap E_R$ and $W \in \cW^u$ be such that $x_1 \in W$. Let $x_3 \in E_1 \cap R \cap E_R$ be such that $0 < d(x_1,x_3) < \tfrac{1}{n}$. Let $W_0 \in \cW^u$ be the unstable manifold containing $x_3$. By construction of $E_R$, we have $\mu_u^{W_0}(W_0 \cap E_1)=1$, and since $\mu_u^{W_0}$ have support $W_0$ (otherwise, $\mu_g$ would not have total support because of the absolute continuity of the holonomy), in fact we have that
\begin{align*}
\mu_u^{W_0}(W_0 \cap E_1 \cap B(x_1,\tfrac{1}{n})) > 0.
\end{align*}
Thus $\mu_u^W(\Theta_W(W_0 \cap E_1 \cap B(x_1,\tfrac{1}{n}))) >0$. Now, if $\tilde x_3 \in W_0 \cap E_1 \cap B(x_1,\tfrac{1}{n})$, then $\tilde x_2 \notin E_1$. In other words, $E_1 \cap \Theta_W(W_0 \cap E_1 \cap B(x_1,\tfrac{1}{n})) = \emptyset$. Since $x_1 \in E_R$, we have that $\mu_u^W(W \cap E_1) = 1$, so that $\mu_u^W(\Theta_W(W_0 \cap E_1 \cap B(x_1,\tfrac{1}{n}))) = 0$, a contradiction. Thus $E_R \cap C_{2,n}^R = \emptyset$, so that $\mu_R(C_2^R)=0$. We proceed similarly, exchanging the role of $\cW^s$ and $\cW^u$, in order to prove that $\mu_R(C_4^R)=0$. Finally, we get that $\mu_R(E_2^R)=1$, the contradiction closing the proof.
\end{proof}

\begin{proposition}\label{prop:billiard_flow_bernoulli}
Under the assumptions of Proposition~\ref{prop:billiard_flow_k-system}, $(\phi_t, \bar \mu_g)$ is Bernoulli.
\end{proposition}

\begin{proof}
The idea of the proof is to bootstrap from the K-mixing following the techniques of \cite{ChH} with modifications similar to those in \cite[Proposition~7.19]{BD2020MME}. The proof in \cite{ChH} proceeds in two steps.

\smallskip
\noindent
\emph{Step 1. Construction of $\delta$-regular coverings.} Given $\delta >0$, the idea is to cover $\tilde \Omega$, up to a set of $\bar \mu_g$-measure at most $\delta$, by small Cantor boxes -- essentially a set of the form Cantor rectangle times interval -- such that $\bar \mu_g$ restricted to each Cantor box is arbitrarily close to a product measure. The basis of the boxes will be very similar to the covering $\{R_i\}_{i \in \mathbbm{N}}$ from Lemma~\ref{lemma:countable_cover_with_Cantor_rectangles}, however, some adjustments must be made in order to guarantee uniform properties of the Jacobian of the relevant holonomy map. 

Above a Cantor rectangle $R$ with $\mu_g(R)>0$, we construct a Cantor box $B$ following the construction of \emph{P-sets} from \cite[Section~3]{ornstein73}. Let $W_1^s$ and $W_2^s$ be the stable sides of the smallest solid rectangle $D(R)$ containing $R$. Let $W$ be a stable manifold for $\phi_t$ projecting on $W^s_1$ through the map $P_-: (x,t) \in \Omega \mapsto x \in M$ if $t<\tau(x)$, and being such that $W \subset \tilde \Omega_0 \coloneqq \{(x,t) \mid 0<t<\tau(x) \}$. Consider the set $W_R \subset W$ of points $(x,t) \in W$ such that $x \in R$. Let $t_0$ be small enough so that $S = \bigcup_{t=0}^{t_0} \phi_t(W_R) \subset \Omega_0$. Now, $B_0$ is obtained by moving $S$ along the unstable manifolds of $\phi_t$ to another surface of that type, spanned by $W_2^s$. That is, for each $(x,t) \in S$, take the unstable manifold $W(x,t)$ of $\phi_t$ passing by $(x,t)$, and projecting on the unstable manifold for $T$ passing by $x \in R$. Let $B_0 = \bigcup_{(x,t) \in S} W(x,t)$ and let $B \subset B_0$ be the set of points $(x,t) \in B_0$ such that $x \in R$. Notice that, up to subdividing $R$ into smaller rectangles taking a smaller $t_0$, we can assume that $B \subset \tilde \Omega_0$. Thus, by construction, the set $B$ has the property that for all $x,y \in B$, the local unstable manifold of $x$ and the local weakly stable manifold of $y$ intersect each other at a single point which lies in $B$. This is the crucial property of what Ornstein and Weiss, in \cite{ornstein73}, called a \emph{rectangle}.

Since $\mu_g(R)>0$, we have $\bar \mu_g(B)=t_0 \mu_g(R)>0$, so that the conditional measure $\bar \mu_B$ of $\bar \mu_g$ restricted to $B$ makes sense. Now, we want to equip $B$ with a product measure, absolutely continuous with respect to $\bar \mu_B$. We proceed as follows. 
Since the partition of $B$ into unstable manifolds is measurable, we can disintegrate $\bar \mu_B$ into conditional measures $\bar \mu ^{W_\xi}$, on $W_\xi \cap B$ with $\xi \in Z_\phi$, and a factor measure $\hat{\bar\mu}$ on the set $Z_\phi$ parametrizing the unstable manifolds of $B$. Fix a point $z \in B$, and consider $B$ as the product of $W^u(z) \cap B$ with $W^{ws}(z) \cap B$, where $W^u(z)$ is the local unstable manifold of $z$ and $W^{ws}(z)$ is the local weak stable manifold of $z$. Define $\bar \mu_B^p = \bar \mu ^{W^u(z)} \otimes \hat{\bar\mu}$, and note that we can view $\hat{\bar\mu}$ as inducing a measure on $W^{ws}(z)$. We still have to prove that $\bar \mu_B^p \ll \bar \mu_B$.

Similarly, let $\mu_R$ be the conditional measure of $\mu_g$ restricted to $R$. Since the partition into unstable manifolds $W_\xi$, $\xi \in Z$, is measurable, we can disintegrate $\mu_R$ into the conditional measures $\mu^W$ on $W \cap R$ and a factor measure $\hat \mu$ on $Z$. We want to relate the disintegration $\bar \mu_B$ with the one of $\mu_R$. 
Notice that we can view $Z_\phi$ as the set $Z \times [0,t_0]$, where $Z$ parametrize the set of unstable manifolds of $R$ through the map associating $\xi_\phi \in Z_\phi$ with the pair $(\xi,t)$ where $\xi \in Z$ is such that $P_-(W_{\xi_\phi}) = W_\xi \subset D(R)$ and $t$ is the value in the definition of $S$ where $W_{\xi_\phi}$ and $S$ intersect. Considering sets $A \subset B$ of the form $A = P_-(A) \times [t_-,t_+]$, we get that
\begin{align*}
\int_{\xi_\phi \in Z_\phi} \bar \mu ^{W_{\xi_\phi}} (A) \,\mathrm{d}\hat{\bar\mu}(\xi_\phi) &= \bar \mu_B(A) = \int_{t_-}^{t_+} \mu_R(P_-(A)) \,\mathrm{d}t = \int_{t_-}^{t_+} \int_{\xi \in Z} \mu ^{W_\xi}(P_-(A)) \,\mathrm{d}\hat\mu(\xi) \,\mathrm{d}t.
\end{align*}
Thus, we can identify $\bar \mu ^{W_{\xi_\phi}}$ with $\mu ^{P_-(W_{\xi_\phi})}$, and $\mathrm{d}\hat{\bar\mu}$ with $\mathrm{d}\hat{\mu} \mathrm{d}t$. From this identifications, we deduce that the projection map $P_{W,-}$ from some $W$ to $P_-(W)$, and its inverse are absolutely continuous. The absolute continuity of the holonomy map $\bar \Theta_W$ between unstable manifolds $W_0$ and $W$ in $B$ thus follows directly from the absolute continuity of the holonomy map between unstable manifolds in $R$ since $\bar \Theta_W = P_{W,-}^{-1} \circ \theta_{P_-(W)} \circ P_{W,-}$. This implies that $\bar\mu_B^p$ is absolutely continuous with respect to $\bar \mu_B$, and thus, also with respect to $\bar \mu_g$. The following definition is taken from \cite{ChH}.

\begin{definition}\label{def:delta-regular_covering}
For $\delta > 0$, a $\delta$-regular covering of $\Omega$ is a finite collection $\mathfrak{B}$ of disjoint Cantor boxes for which\footnote{The corresponding definition in \cite{ChH} has a third condition, but it is satisfied in our setting since the stable and unstable manifolds are one-dimensional and have bounded curvature.},
\begin{enumerate}
\item[a)] $\bar \mu_g (\bigcup_{B \in \mathfrak{B}} B) \geqslant 1-\delta$.
\item[b)] Every $B \in \mathfrak{B}$ satisfies $\left| \frac{\bar\mu _{B}^p(B)}{\bar\mu_g(B)} -1 \right| < \delta$. Moreover, there exists $G \subset B$, with $\bar\mu_g (G) > (1-\delta) \bar \mu_g(B)$, such that $\left| \frac{\mathrm{d} \bar\mu _{B}^p }{\mathrm{d}\bar\mu_g} (x) - 1 \right| < \delta$ for all $x \in G$.
\end{enumerate}
\end{definition}
By \cite[Lemma~5.1]{ChH}, such coverings exist for any $\delta >0$, and for Cantor boxes arbitrarily small. The proof essentially uses the covering of $M^{\mathrm{reg}}$ from Lemma~\ref{lemma:countable_cover_with_Cantor_rectangles} to build Cantor boxes, up to finite subdivision of the covering to ensure a). To get b), subdivide the boxes into smaller ones on which the Jacobian of the holonomy map between unstable manifolds is nearly $1$. This argument relies on Lusin's theorem and goes through in our setting with no changes.

\smallskip
\noindent
\emph{Step 2. Proof that $\bar \alpha_i$ is vwb}. First, define $\bar \alpha_i$ to be the partition of $\tilde \Omega$ into sets of the form $\tilde \Omega_0 \cap (A \times [\tfrac{l}{2^i},\tfrac{l+1}{2^i}))$, where $A \in \cM_{-1}^1$ and $l \in \mathbbm{N}$. Then $\bar \alpha_0 \leqslant \bar \alpha_1 \leqslant \bar \alpha_2 \leqslant \ldots$ is such that $\bigvee_{i=1}^{\infty} \bigvee_{n=-\infty}^{+\infty} \phi_n \bar \alpha_i$ generates the whole $\sigma$-algebra on $\tilde \Omega$. Using Theorems~4.1 and 4.2 from \cite{ChH}, we only need to prove that each partition $\bar \alpha_i$ is vwB in order to prove that $(\phi_t,\bar \mu_g)$ is Bernoulli.

Using $\cM_{-1}^1$ as the basis elements of $\bar \alpha_i$ allows us to apply the bounds \eqref{eq:estimate_mug_neighbourhood_Sk} directly since $\partial \cM_{-1}^1 = \cS_1 \cup \cS_{-1}$. We can now apply the same arguments as in \cite[Section~6.2]{ChH} with the modifications described in the second part of the proof of \cite[Proposition~7.19]{BD2020MME}. Actually, the only place where we need to be careful is \cite[Eq.~(7.33)]{BD2020MME} because of our additional horizontal cuttings. We finish the proof by dealing with this equation. We first have to recall some notations from \cite{BD2020MME} first.

Fix some $i \in \mathbbm{N}$, and let $\bar \alpha = \bar \alpha_i$. Let $\varepsilon >0$ and define $\delta = e^{-(\varepsilon/C')^{2/(1-\gamma)}}$ (recalling that $\gamma > 1$), where $C'>0$ is the constant from \eqref{eq:condition_C'} below.

Let $\mathfrak{B}=\{B_1,B_2,\ldots,B_k\}$ be a $\delta$-regular cover of $\tilde \Omega$ such that the diameters of the $B_i$'s are less than $\delta$. Define the partition $\pi = \{ B_0,B_1,B_2,\ldots,B_k\}$, where $B_0 = \Omega \smallsetminus \cup_{i=1}^k B_i$. For each $i \geqslant 1$, let $G_i \subset B_i$ denote the set identified in Definition~\ref{def:delta-regular_covering}(b).
Since $(\phi_{-1}, \bar \mu_g)$ is K-mixing, there exists an even number $N=2m$, such that for any integers $N_0$, $N_1$ such that $N<N_0<N_1$, $\delta$-almost every atom $A$ of $\bigvee_{N_0-m}^{N_1-m} \phi_{-i} \bar \alpha$, satisfies 
\begin{align*}
\left| \frac{\bar \mu_g(B|A)}{\bar \mu_g(B)} - 1 \right| < \delta, \quad \text{for all } B \in \pi,
\end{align*}
where $\bar \mu_g(\cdot |A)$ is the measure $\bar \mu_g$ conditioned to $A$. Now let $m$, $N_0$, $N_1$ be given as above and define $\omega = \bigvee_{N_0-m}^{N_1-m} \phi_{-i} \bar \alpha$. In order to prove that $\bar \alpha$ is vwb, we need to give estimates on elements of $\omega$. To do so, we identify as in \cite[Section~6.2]{ChH} sets of bad atoms, whose union will have measure less than $c \ve$. As in \cite{ChH}, we call these sets $\hat F_1$, $\hat F_2$, $\hat F_3$, $\hat F_4$. Since the estimates on the $\bar \mu_g$-measure of the bad sets $\hat F_1$ and $\hat F_2$ do not change, we only define $\hat F_3$ and $\hat F_4$. Define $F_3$ to be the set of all points $x \in \Omega \smallsetminus B_0$ such that $W^s(x)$ intersects the boundary of the element $\omega(x)$ before it fully crosses the rectangle $\pi(x)$. Thus, if $x \in F_3$, there exists a subcurve of $W^s(x)$ connecting $x$ to the boundary of $(\phi_{-i}\bar\alpha)(x)$ for some $i \in [N_0 - m,N_1-m]$. Then since $\pi(x)$ has diameter less than $\delta$, $\phi_i(x)$ lies within a distance $C \tilde \Lambda^{-i} \delta$ of the boundary of $\bar \alpha$ -- where $\tilde C_1$ and $\tilde \Lambda > 1$ come from the hyperbolicity of the billiard flow. Using the bound \eqref{eq:estimate_mug_neighbourhood_Sk}, the total measure of such points must add up to at most
\begin{align}\label{eq:condition_C'}
\sum_{i=N_0-m}^{N_1-m} \left( \frac{C}{|\log(\tilde C_1 \tilde \Lambda^{-i}\delta)|^\gamma} + C_{\bar \alpha}  \tilde C_1 \tilde \Lambda^{-i}\delta) \right) \leqslant C_1'|\log \delta |^{1-\gamma} + C_2'\delta \leqslant C'|\log \delta |^{1-\gamma},
\end{align}
for some $C'>0$. Letting $\hat F_3$ denote the union of atoms $A \in \omega$ such that $\bar \mu_g(F_3|A) > |\log \delta |^{\tfrac{1-\gamma}{2}}$, it follows that $\bar \mu_g(F_3) \leqslant C'|\log \delta|^{\tfrac{1-\gamma}{2}}$. This is at most $\varepsilon$ by choice of $\delta$.

The same precaution allows us to get the same bound on $\bar \mu_g(\hat F_4)$ as in \cite{BD2020MME}. 

Finally, the bad set to be avoided in the construction of the joining is $B_0 \cup (\cup_{i=1}^4 \hat F_i)$. Its measure is less than $c\varepsilon$ by choice of $\delta$. From this point, once the measure of the bad set is controlled, the rest of the proof in Section~6.3 of \cite{ChH} can be repeated verbatim. This proves that $\bar \alpha$ is vwB.
\end{proof}

\begin{proposition}\label{prop:flow_adapted}
Under the assumptions of Proposition~\ref{prop:billiard_flow_k-system}, the measure $\bar \mu_g$ is flow adapted\footnote{This result is due to Mark Demers. I thank him for allowing me to use his proof.}, that is, $\log ||D\phi_t ||$ is $\mu_g$-integrable.
\end{proposition}

\begin{proof}
Let $\Omega = \{ (x,y, \theta) \in Q \times \mathbb{S}^1 \} \subset \mathbb{T}^3$ denote the phase space for the billiard flow $\Phi_t$ with the usual Euclidean metric denoted by $d_\Omega$. Let $\nu_g$ be the flow invariant measure obtained as the image of $\bar \mu_g$ by the conjugacy map between $\Omega$ and $\tilde \Omega$. Let 
\[ \cS_0^- = \{ \Phi_{-t}(z) \in \Omega \mid z \in \cS_0 \mbox{ and } t \leqslant \tau(T^{-1}z) \}  \]
denote the flow surface obtained by flowing $\cS_0$ backward until its first collision under the inverse flow. Similarly, let
\[
    \cS_0^+ = \{ \Phi_t(z) \in \Omega \mid z \in \cS_0 \mbox{ and } t \leqslant \tau(z) \}
\]
denote the forward flow of $\cS_0$ until its first collision. To show that the measure $\nu_t$ is flow-adapted, it suffices to show that $\int_{\Omega} |\log d_\Omega(x, \cS_0^{\pm}) |\, \mathrm{d}\nu_g(x) < \infty$.  For then this implies that $\log \| D\Phi_t \|$ is integrable for each $t \in [-\tau_{\min}, \tau_{\min}]$ and then by subadditivity for each $t \in \mathbb{R}$.

Let $P^\pm(\cdot)$ denote the projection under the forward (backward) flow of a subset of $\Omega$ until first collision.  Let $\cN^M_\ve(\cdot)$ denote the $\ve$-neighborhood of a set in $M$ in the Euclidean metric $d_M$ and let $\cN^\Omega_\ve(\cdot)$ denote the $\ve$-neighborhood of a set in $\Omega$ in the metric $d_\Omega$. It follows from \cite[Exercise~3.15]{chernov2006chaotic}, that there exists $C>0$ such that for any $\ve >0$,
\begin{equation}
\label{eq:project distance}
P^-(\cN^\Omega_\ve(\cS_0^-)) \subset \cN^M_{C\ve^{1/2}}(\cS_1) \quad \mbox{ and similarly } \quad
P^+(\cN^\Omega_\ve(\cS_0^+)) \subset \cN^M_{C\ve^{1/2}}(\cS_{-1})
\end{equation}
From \eqref{eq:estimate_mug_neighbourhood_Sk}, there exist $C_g>0$ and $\gamma >1$ such that
\begin{equation}
\label{eq:eps bound}
\mu_g(\cN^M_\ve(\cS_{\pm 1})) \leqslant C_g |\log \ve|^{-\gamma} \qquad \mbox{for all $\ve>0$.}
\end{equation}
Putting together \eqref{eq:project distance} and \eqref{eq:eps bound} yields
\begin{equation}
\label{eq:flow bound}
\nu_g(\cN^\Omega_\ve(\cS_0^-)) \leqslant \tau_{\max} C_g |\log C\ve^{1/2}|^{-\gamma} \leqslant C' \tau_{\max} |\log \ve|^{-\gamma} \, .
\end{equation}

For $p>1$ to be chosen below, define for $n \geqslant 1$, $A_n = \cN^\Omega_{e^{-n^p}}(\cS_0^-) \setminus \cN^\Omega_{e^{-(n+1)^p}}(\cS_0^-)$.  If $x \in A_n$, then $|\log d_\Omega(x, \cS_0^-)| \leqslant (n+1)^p$.  Thus we estimate using \eqref{eq:flow bound},
\[
\begin{split}
\int_\Omega |\log d_\Omega(x, \cS_0^-) | \, \mathrm{d}\nu_g
& \leqslant 1 + \log \diam(\Omega) + \sum_{n \geqslant 1} \int_{A_n} |\log d_\Omega(x, \cS_0^-) | \, \mathrm{d}\nu_g \\
& \leqslant 1 + \log \diam(\Omega) + \sum_{n \geqslant 1} (n+1)^p C' \tau_{\max} n^{-\gamma p} \, ,
\end{split}
\]
and the series converges as long as $p > 1/(\gamma-1)$.  A similar argument shows that $\log d_\Omega(x, \cS_0^+)$ is $\nu_g$ integrable so that $\nu_g$ is flow adapted.
\end{proof}

\begin{appendix}

\section{Motivations from uniform hyperbolic dynamics}

We start this note by presenting the usual method the existence of measures of maximal entropy is proved in the case of uniform hyperbolicity. First, we consider a hyperbolic transformation of a compact set, and then the case of an Anosov flow.

\subsection{Hyperbolic maps}

Let $X$ be a compact Riemannian manifold and let $T : X \to X$ be a $\mathcal{C}^r$ diffeomorphism. Assume that $T$ is uniformly hyperbolic, that is \begin{align*}
&\exists \lambda > 1, \exists C > 0, \exists E^s, \, E^u \subset T X \text{ such that }\\
(i)&\, TX = E^s \oplus E^u, \, DT(E^s) \subset E^s, \, DT^{-1}(E^u) \subset E^u, \\
(ii)&\, ||D_x T^n v_s|| \leqslant C \lambda^{-n} ||v_s||, \quad \forall n \geqslant 0, \, \forall v_s \in E^s_x \subset T_x X, \\
(iii)&\, ||D_x T^{-n} v_u|| \leqslant C \lambda^{-n} ||v_u||, \quad \forall n \geqslant 0, \, \forall v_u \in E^u_x \subset T_x X.
\end{align*}
One fundamental theorem about hyperbolic dynamic is the Hadamard--Perron Theorem \cite[Theorem 6.2.8]{katok1997introduction} which states that there exist two unique families of $\mathcal{C}^r$ manifolds, $\{ W^+_m \}_{m \in \mathbbm{Z}}$ and $\{ W^-_m \}_{m \in \mathbbm{Z}}$, everywhere tangent respectively to $E^s$ and to $E^u$, obtained as the graph of some functions, and satisfying some stability and contraction properties. A key tool in the proof is the construction of families of stable and unstable cones.

As a consequence \cite[Corollary 6.4.10]{katok1997introduction}, all such diffeomorphisms are expansive, that is 
\begin{align}\label{expansiveness}
\exists \delta > 0, \, \forall x,y \in X, \, [\mathrm{d}(T^n(x),T^n(y)) < \delta, \, \forall n \in \mathbbm{Z} \Rightarrow x=y ].
\end{align}

From the expansive property, it follows from \cite[Theorem 8.2]{Walters82ergth} that the metric entropy $\mu \mapsto h_{\mu}(T)$ is upper semi-continuous, hence the existence of equilibrium states for every continuous potential -- and in particular existence of measures of maximal entropy for the zero potential. In the proof of \cite[Theorem 8.2]{Walters82ergth}, expansiveness is only use to get the equality $h_{\mu}(T) = h_{\mu}(T, \mathcal{A})$ for partition $\mathcal{A}$ with $\mathrm{diam}(\mathcal{A}) < \delta$ (the expansivity constant of $T$) and any $T$-invariant measure $\mu$. 

As proved by Bowen \cite[Theorem 3.5]{bowen1972entropy}, the expansiveness assumption of \cite[Theorem 8.2]{Walters82ergth} can be weakened to \emph{entropy-expansiveness} (the proof remains unchanged). This weakening will be relevant in the case of Anosov flows.

\subsection{Anosov flows}

Let $X$ be a compact manifold and $\phi = \{ \varphi^t\} : \mathbbm{R} \times X \to X$ be a smooth flow. Assume that $\phi$ is an Anosov flow, that is
\begin{align*}
&\exists \lambda > 1, \exists C > 0, \exists E^s, \, E^u, \, E^c \subset T X \text{ such that }\\
(i)&\, TX = E^c \oplus E^s \oplus E^u, \\
(ii)&\, D\varphi^t(E^{s/u}) = E^{s/u}, \, \mathrm{dim} E^c_x =1, \, \left. \frac{d}{dt} \right|_{t=0} \varphi^t(x) \in E^c_x \smallsetminus \{0\}, \\
(iii)&\, ||D_x \varphi^t _{| E^s_x} || \leqslant C \lambda^{-t}, \, ||D_x \varphi^{-t} _{| E^u_x} || \leqslant C \lambda^{-t}, \, \forall t \geqslant 0.
\end{align*}

In \cite[Proposition 1.6]{bowen1972periodic}, Bowen proves that an Anosov flow is \emph{flow expansive} (in the sense of Bowen--Walters), that is -- as defined in \cite{bowen1972expansive} in the case of a fixed-point free flow,
\begin{align}\label{flow_expansiveness}
\begin{matrix}
\forall \varepsilon > 0, \, \exists \delta > 0, \, \forall x,y \in X, \, \forall h \in \mathcal{C}^0(\mathbbm{R}) \text{ with } h(0)=0, \\ [d(\varphi^t(x), \varphi^{h(t)}(y))< \delta, \, \forall t \in \mathbbm{R} \Rightarrow y \in \varphi^{]-\varepsilon, \varepsilon[}(x) ].
\end{matrix}
\end{align}
The key ingredient of the proof is the local product structure for hyperbolic flows.
From (\ref{flow_expansiveness}), it is easy to see, for $h = id$, that an Anosov flow satisfies the following weaker property 
\begin{align}\label{example_1_6}
\begin{matrix}
\exists \varepsilon > 0, \, \exists s > 0, \, \forall x \in X, \\ \Gamma_{\varepsilon}(x) \coloneqq \{ y \in X \mid \forall t \in \mathbbm{R}, \, d(\varphi^t(x), \varphi^t(y)) < \varepsilon\} \subset \varphi^{[-s,s]}(x).
\end{matrix}
\end{align}
Bowen proved \cite[Example 1.6]{bowen1972entropy} that (\ref{example_1_6}) is a sufficient condition so that every time $\varphi^t$ of the flow is entropy-expansive. Therefore the map $\mu \in \mathcal{M}_X(\varphi^1) \mapsto h_{\mu}(\varphi^1)$ is upper semi-continuous, and so is its restriction to $\mathcal{M}_X(\phi) \subset \mathcal{M}_X(\varphi^1)$. Hence, Anosov flows have equilibrium states for every continuous potential, and in particular for the zero potential, measures of maximal entropy.

\section{Obstructions for the Billiard Flow}

In the previous section, in both situations, proofs of existence of MME use some sort of expansiveness. However, the existence of a local product structure is a key ingredient in order to establish the expansivity property: it gives a scale used as the $\delta$ in (\ref{expansiveness}) and the $\varepsilon$ in (\ref{example_1_6}). Furthermore, the uniform contraction of stable (resp. unstable) manifolds for large positive (resp. negative) times is used, and not some estimates of their lengths in negative (resp. positive) times (such as fragmentation or growth lemmas, see for example \cite{chernov2006chaotic}).

\subsection{Entropy expansiveness}
In Bowen's proof, the local product structure is the main tool in order to prove flow expansiveness. In the case of the billiard flow, their is no such structure. Indeed, stable and unstable manifolds exist only for Lebesgue-almost every point and there is no deterministic control of their length (hence no uniform scale for a local structure). One might argue that a billiard flow admits invariant ``cone" fields \cite[Section 2]{BDL2018ExpDecay} and construct stable and unstable curves, but then the control on the length of those curves when applying the flow is in term of expansion, not in term of contraction.

It then seems that h-expansiveness of each time $\varphi^t$ of the flow is too much to ask for. Still, one might hope that each $\varphi^t$ is \emph{asymptotically h-expansive}, that is $h^*(\varphi^t) \coloneqq \lim\limits_{\varepsilon \to 0} h^*(\varphi^t, \varepsilon)=0$, where $h^*(\varphi^t, \varepsilon) = \sup\limits_{x \in X} h(\varphi^t,\overline{B}(x,\varepsilon))$. This definition was first introduced by Misiurewicz in \cite{misiurewicz1973diffeomorphism} where he proved that the metric entropy of an asymptotic h-expansive transformation is upper semi-continuous.

The quantity $h^*(\varphi^t)$ is usually referred to as the \emph{topological tail entropy} of $\varphi^t$ \cite{downarowicz2011bookentropy}. In the context of smooth dynamics, Buzzi \cite{buzzi1997intrinsic} has shown that if $f \in \mathcal{C}^r(M)$, then $h^*(f) \leqslant \frac{dim(M) R(f)}{r}$ for some constant $R(f)$. In particular, the metric entropy of a $\mathcal{C}^{\infty}$ transformation is upper semi-continuous. Clearly, this result does not apply to billiard flows.

Proving that the topological tail entropy of the billiard flow is zero is enough to prove the upper semi-continuity of the metric entropy, hence the existence of some measure of maximal entropy.

\subsection{Relations with the Collision Map}

In \cite[Theorem 6]{bowen1972expansive}, Bowen and Walters prove that the special flow constructed over a continuous transformation and under a continuous return time function, is flow expansive if and only if the base map is expansive. Since flow expansiveness is an invariant for flow under reparametrization, without loss of generality, the return time function can be chosen constant.

In \cite{BD2020MME}, Baladi and Demers show that the collision map is expansive. However, since the return time is only piecewise continuous, it is not easy to relate the expansivity of the collision map to flow expansiveness of the billiard flow. As shown in Figure~\ref{fig:billiard}, two trajectories can be easily separated by the collision map, but they remain \emph{close} in the phase space of the flow. We see that for a $\delta$ too large in (\ref{flow_expansiveness}) (and a natural choice of $h$), the two trajectories cannot be distinguished. What could be a good choice for $\delta$? The main problem being to find a $\delta$ independent of trajectories (it is \emph{easier} to find a $\delta$ for specific trajectories, such as those ones in Figure~\ref{fig:billiard}, but the $\inf$ of those $\delta$ over all trajectories might be $0$). If such $\delta$ existed, we expect it is controlled in some way by $\tau_{min}$.

For similar reasons, it appears that it is not a simple consequence of the collision map expansiveness for the flow to satisfy condition (\ref{example_1_6}) (which is a weaker than \emph{flow expansiveness}). For example, the two orbits shown in Figure~\ref{fig:billiard} (b) are close in the phase space of the flow, but far apart in the phase space of the collision map (since the collisions they make are distinct).

\begin{figure}
\begin{tabular}{ccc}
\raisebox{-0.5\height}{\includegraphics[width=0.45\textwidth]{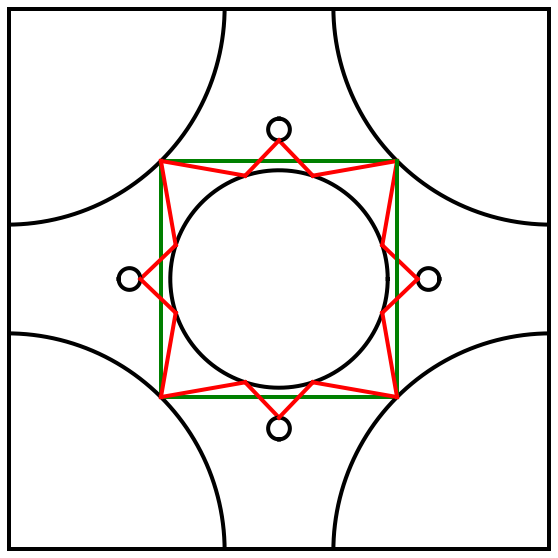}} &~&
\raisebox{-0.5\height}{\includegraphics[width=0.45\textwidth]{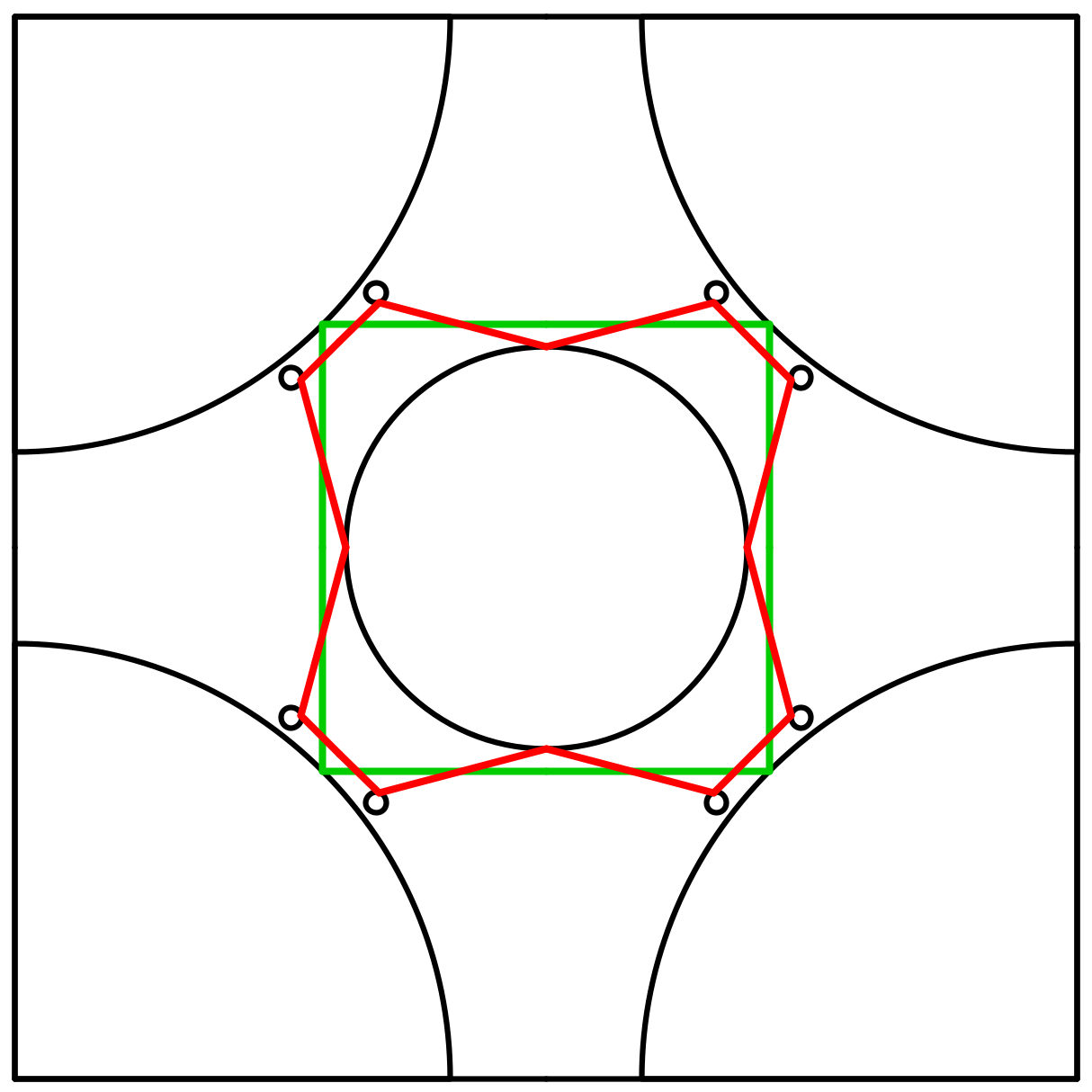}}\\
(a) Some common collisions &~&  (b) Distinct collisions
\end{tabular}
\caption{\label{fig:billiard} Two examples of two periodic trajectories.}
\end{figure}
\end{appendix}

\bibliography{biblio}{}
\bibliographystyle{abbrv}

\end{document}